\documentclass[a4paper,12pt]{article}
\title{Fine compactified moduli of enriched structures on stable curves}

\usepackage{amsmath, amssymb, mathrsfs, amsthm, geometry, bbm, longtable}% amsfonts, url,, tikz, wrapfig, newclude, }

\numberwithin{equation}{subsection}

\usepackage{mathtools}
\usepackage{hyperref}
\usepackage{cleveref}
\let\oref\ref
\AtBeginDocument{\renewcommand{\ref}[1]{\cref{#1}}}

\usepackage{tikz-cd}

\newcommand{\on}[1]{\operatorname{#1}}
\newcommand{\bb}[1]{{\mathbb{#1}}}
\newcommand{\cl}[1]{{\mathscr{#1}}}
\newcommand{\ca}[1]{{\mathcal{#1}}}

\newcommand{\Span}[1]{\left<#1\right>}

\newcommand{\op}{^{\mathrm{op}}}

\newcommand{\ra}{\rightarrow}
\newcommand{\lra}{\longrightarrow}
\newcommand{\hra}{\hookrightarrow}
\newcommand{\sub}{\subseteq}
\newcommand{\tra}{\rightarrowtail}
\newcommand{\iso}{\stackrel{\sim}{\lra}}

\theoremstyle{plain}
\theoremstyle{definition}
\newtheorem{definition}{Definition}[section]

\newtheorem{situation}[definition]{Situation}

\theoremstyle{plain}% default
\newtheorem{proposition}[definition]{Proposition}
\newtheorem{lemma}[definition]{Lemma}
\newtheorem{theorem}[definition]{Theorem}
\newtheorem{corollary}[definition]{Corollary}

\newtheorem{inclaim}{Claim}[definition]

\theoremstyle{remark}

\newtheorem{remark}[definition]{Remark}

\newtheorem{example}[definition]{Example}

\renewcommand{\phi}{\varphi}

\newcommand{\testleft}{\leftarrow\!\shortmid}

\newcounter{temp} 

\author{Owen Biesel and David Holmes}
\date{\today}
 
\newcounter{nootje}
\newcommand\todo[1]{[*\thenootje]\marginpar{\tiny\begin{minipage}
{20mm}\begin{flushleft}\thenootje : 
#1\end{flushleft}\end{minipage}}\addtocounter{nootje}{1}}
\renewcommand{\todo}[1]{}

\newcommand{\beq}{\begin{equation}}
\newcommand{\eeq}{\end{equation}}
\newcommand{\beqs}{\begin{equation*}}
\newcommand{\eeqs}{\end{equation*}}

\begin{document}
\newcommand{\Pictdz}{\on{Pic}^{[0]}}
\newcommand{\Picz}{\on{Pic}^{0}}
\newcommand{\M}{\ca{M}}
\newcommand{\Mbar}{\overline{\ca{M}}}
\newcommand{\Mtil}{\widetilde{\ca{M}}}
\newcommand{\MN}{{\ca{N}}}
\newcommand{\Menr}{\overline{\ca{M}}^{\varepsilon}}
\newcommand{\Mdenr}[1]{\overline{\ca{M}}^{\varepsilon(#1)}}
\newcommand{\cat}[1]{\mathbf{#1}}

\maketitle
\begin{abstract} 
Enriched structures on stable curves over fields were defined by Main\`o in the late 1990s, and have played an important role in the study of limit linear series and degenerating jacobians. In this paper we solve three main problems: we give a definition of enriched structures on stable curves over arbitrary base schemes, and show that the resulting fine moduli problem is representable; we show that the resulting object has a universal property in terms of N\'eron models; and we construct a compactification of our stack of enriched structures. %This last part answers a question of Caporaso and others. 
\end{abstract}
\tableofcontents
\section{Introduction}

%TODO:\\
%- mention Tyler Jarvis as Esteves suggested?\\
%- mention BKK paper, and Ana Maria?\\

If $C/k$ is a stable curve over an algebraically closed field with irreducible components $C_1, \dots, C_n$, and $\ca C / k[[t]]$ is a regular smoothing of $C$, then the $C_i$ can be thought of as divisors on $\ca C$, and one defines line bundles $\cl L_i$ on $C$ by the formula $\cl L_i = \ca O_{\ca C}(C_i)|_C$. In \cite{Maino1998Moduli-space-of}, Main\`o defined an \emph{enriched structure} on $C/k$ to be a collection of line bundles $\cl L_1, \dots, \cl L_n$ which arise in the above fashion from some regular smoothing of $C/k$. She gave a more intrinsic characterisation of these enriched structures, and also constructed a coarse moduli space (which we will discuss more later on). Main\`o's work has played an important role in a number of works on stable curves, such as in Osserman's construction of limit linear series on stable curves \cite{Osserman2014Limit-linear-se}, \cite{Osserman2014Dimension-count}, and work of Esteves, Nivaldo, Salehyan and others on related problems \cite{Esteves2002Limit-canonical}, \cite{Esteves2007Limit-Weierstra}. In \cite{Abreu2014Enriched-curves}, Abreu and Pacini study a tropical counterpart of enriched structures, with the explicit hope that it might lead to a modular compactification of the moduli of enriched structures. At the conclusion of her talk at the Summer Research Institute on Algebraic Geometry in Salt Lake City, Caporaso also asked for a compactification of the moduli of enriched structures. The construction of such a compactification is the third of the three main results of this article; we call its points \emph{compactified enriched structures}.

The definition of enriched structure given above only applies to a stable curve over an algebraically closed field, and so the notion of `moduli of enriched structures' is rather awkward. Writing $\overline{M}_{g}$ for the coarse moduli space of stable curves of genus $g$, Main\`o constructs a scheme $En_g \to \overline{M}_g$ and for each $k$-point $p$ of $\overline{M}_g$ (corresponding to a curve $C/k$) gives a bijection of sets between the set of enriched structures on $C/k$ and the $k$-points of $En_g$ lying over $p$. Noting that every stable curve over $k$ admits at least one enriched structure, it is clear that this $En_g/\overline{M}_g$ is highly non-unique - for example, the disjoint union of the fibres of $En_g \to \overline{M}_g$ has the same property. Now Main\`o's construction is much more reasonable than this --- it is connected and quasi-compact, given as an open subscheme of a certain blowup of $\overline{M}_g$ --- but with her definition of enriched structure it was hard to justify why this is a `good' moduli space. 

Our first main result is a definition of an enriched structure on a stable curve over an arbitrary base scheme (allowing marked points and more, see \ref{sec:defining_curves}). Unfortunately our definition is too involved to spell out in the introduction, though in \ref{sec:2-gons} we will describe the special case of the universal deformation of a 2-pointed 2-gon. We show that our definition reduces to Main\`o's when the base is a point. We also show that the resulting stack $\ca E$ over the fine moduli stack $\Mbar_{g,n}$ is algebraic; better, that $\ca E$ is relatively representable by an algebraic space over $\Mbar_{g,n}$. In \ref{sec:maino_blowups} we will adapt Main\'o's construction of her space $En_g$ slightly to work over the fine moduli stack $\Mbar_g$, and then show that it \emph{does} coincide naturally with our $\ca E$. 

Main\`o also proposes a definition of an enriched structure for a stable curve over an arbitrary reduced base scheme (in a way quite different from us), and conjectures that the resulting functor is coarsely represented by $En_g$. However, in \ref{eg:maino_not_rep} we show that her functor is not  representable, disproving the conjecture. In \ref{sec:our_version-of_maino} we will construct a variant on her definition, which is representable by $En_g$, but the result is somewhat clumsy. In this way, all of Main\'o's conjectures are resolved. 

% proposes two definitions of enriched structures over reduced base schemes (neither of which bears any resemblance to our definition), but does not establish whether they are equivalent. 

Our second main result is to give a universal property for the stack $\ca E \to \Mbar_{g,n}$ in terms of N\'eron models of jacobians. We show that the pullback to $\ca E$ of the universal stable curve $\overline{\ca C}_{g,n}$ is regular. We show that the jacobian of the universal smooth curve $\ca C_{g,n} \to \ca{M}_{g,n}$ admits a N\'eron model $N$ over $\ca E$ (in the sense of \cite{Holmes2014A-Neron-model-o}), and further that $\ca E$ is universal with respect to this property; more precisely we show that if $t\colon T \to \Mbar_{g,n}$ is a morphism such that $t^*\overline{\ca C}_{g,n}$ is regular and the jacobian of $t^*\ca C_{g,n}$ admits a N\'eron model over $T$, then $T \to \Mbar_{g,n}$ factors uniquely via $\ca E \to \Mbar_{g,n}$.

This relates $\ca E$ to the `universal N\'eron model admitting morphism' constructed in \cite{Holmes2014A-Neron-model-o} (more precisely, it shows that $\ca E$ is the `first level' in the tower of universal N\'eron model admitting morphisms, see \ref{sec:UNMA_morphisms}). From this we can deduce for free some nice properties of $\ca E$:
\begin{itemize}
\item
$\ca E$ is regular (non-singular);
\item $\ca E \to \Mbar_{g,n}$ is separated;
\item The pullback of Caporaso's balanced Picard stack (cf.\ \cite{Caporaso2008Neron-models-an}) $\ca{P}_{d,g}/\Mbar_{g,n}$ to $\ca E$ admits a canonical group/torsor structure (here we need conditions on $(d,g)$ for $\ca P_{d,g}$ to be defined);
\item Let $d = (d_1, \dots, d_n)$ be a vector of integers with sum zero. Writing $J_{g,n}\to \ca{M}_{g,n}$ for the universal jacobian, the $d$th Abel-Jacobi map is a section $\alpha\colon  \ca{M}_{g,n} \to J_{g,n}$. Then $\alpha$ extends uniquely to a section $\ca E \to N$ (recall $N / \ca E$ is the N\'eron model of $J$ over $\ca E$). In \ref{sec:DRC} we will briefly discuss an application of this to the double ramification cycle. 
%The N\'eron model is quasi-compact, so some finite multiple lands in the fibrewise connected component of the unit section, so after taking this multiple we get an $d$th Abel-Jacobi map from $\ca E$ to a semiabelian variety over $\ca E$. 
\end{itemize}

Alternatively, if one is initially more interested in the N\'eron models side, then we can view this result as giving an explicit description of the functor of points of $\ca E$, the first level in the tower of universal N\'eron model admitting morphisms. 

The second part of the paper is devoted to constructing a compactification of $\ca E$ over $\Mbar_{g,n}$. More precisely, we construct a proper relative algebraic space $\overline{\ca E} \to \Mbar_{g,n}$ and an open immersion $\ca E \to \overline{ \ca E}$ over $\Mbar_{g,n}$. In view of the `extension of the Abel-Jacobi map' property discussed just above, we plan to use such a compactification to define Gromov-Witten invariants of $B\bb G_m$ (cf.\ \cite{Frenkel2009Gromov-Witten-g}, where `admissibility' is required). Such a construction will need further work, for example showing that $\overline{\ca E}$ is regular, so that the intersection theory is well-behaved, and we postpone this to a future paper. 

The definition of $\overline{\ca E}$ is again too involved to describe in this introduction, but we can nonetheless outline the main differences from $\ca E$. As in Main\`o's original construction, our definition of an enriched structure includes certain invertible sheaves. When moving to $\overline{\ca E}$, these are unsurprisingly replaced by torsion free rank 1 sheaves (see \ref{def:TFR1}). However, there is a substantial complication. Main\`o's definition (spelled out in \ref{def:M_ES}) gives an invertible sheaf for each irreducible component of $C/k$ satisfying certain combinatorial conditions, and requires that the tensor product of all these invertible sheaves together be trivial. In our more general definition of an enriched structure a similar condition remains; at some point the tensor product of certain invertible sheaves is required to be trivial. When replacing the invertible sheaves by torsion free rank 1 sheaves this no longer makes much sense, since if a tensor product of sheaves is trivial then they were by definition invertible! We thus replace the `tensor product is trivial' condition by a rather more involved condition in terms of certain auxiliary data (see \ref{def:compatibility}). We must then work quite hard to verify that, in the case of invertible sheaves, this reduces to the original definition. 

\subsection{Future work on the closure of the double ramification cycle}\label{sec:DRC}
Given the connection between moduli of enriched structures and N\'eron models, it is perhaps unsurprising that the moduli of enriched structures is relevant in the study of the closure of the double ramification cycle. Given a tuple of $n$ integers summing to zero, one constructs a section $\alpha$ to the universal jacobian over ${\ca{M}}_{g,n}$ by taking the formal linear combination of the tautological sections with the given integers as coefficients, and viewing this as a divisor on the universal curve. The double ramification cycle on ${\ca{M}}_{g,n}$ is constructed by pulling back the unit section in the universal jacobian along the section $\alpha$ described above. One then wishes to extend this cycle to $\Mbar_{g,n}$ in a natural way, and to compute its class in the tautological ring. On the locus of curves with at most one non-separating node the above definition extends in a natural way, and the class was computed by Gruschevsky and Zakharov \cite{Grushevsky2012The-zero-sectio}, extending work of Hain \cite{Hain2013Normal-function} who treated the compact-type case. A different approach using virtual fundamental classes on a space of stable maps to a `rubber $\bb P^1$' allows  for the construction of an extension of the double ramification cycle to the whole of $\Mbar_{g,n}$, and its class was computed in \cite{Janda2016Double-ramifica}. Cavalieri, Marcus and Wise verified (\cite{Cavalieri2012Polynomial-fami}, \cite{Marcus2013Stable-maps-to-}) that on the locus of curves of compact type this definition in terms of stable maps to rubber $\bb P^1$ coincides with the original definition in terms of sections to the universal jacobian. In future work, we plan to extend this result to the whole of $\Mbar_{g,n}$ using a (slightly modified) moduli stack of enriched structures to provide a space where the section $\alpha$ will always extend. It was suggested in \cite[\S 1.5]{Cavalieri2012Polynomial-fami} that such a construction should be carried out, but to our knowledge this has not yet happened. 

%\subsubsection{Mumford-Lear extensions, height jumping, and $b$-divisors}
%Given a continuously metrised line bundle on $\ca{M}_{g,n}$, a \emph{Mumford Lear} extension is an extension to a line bundle on $\overline{\ca{M}}_{g,n}$ such that the metric has at-worst logarithmic poles along the boundary outside some codimension-2 subset. The formation of a Mumford-Lear does \emph{not} in general commute with non-flat base change, and this failure to commute is measured by Hain's \emph{height jump} \cite{Hain2013Normal-function}. 

\subsection{Special case of 2-gons}\label{sec:2-gons}
In this section we gently introduce the definition of enriched structure (and compactified enriched structure) by considering the case of a curve with two irreducible components, meeting in two points, whose dual graph is thus a 2-gon.

We work over a separably closed field $k$ --- no subtleties are missed by taking $k = \bb C$. Let $C_0$ be the curve over $k$ constructed by taking two copies of $\bb P^1$, glueing the points $(0:1)$ together, and the points $(1:0)$ together, and finally marking both points $(1:1)$. We have a stable 2-pointed curve of genus 1. Let $\ca C /\ca{M}$ be its universal deformation as a pointed curve, and choose an isomorphism $\ca{M} \cong \on{Spec} k[[x,y]]$ such that the locus in $\ca{M}$ where $ \ca C \to \ca{M}$ is not smooth is given by $xy = 0$. 

The curve $C_0$ is the fibre over the closed point of $\ca C \to \ca{M}$, and so its two irreducible components give closed subschemes of $\ca C$. If $\Gamma$ is the dual graph of $C_0/k$ (a 2-gon), let $u$ and $v$ be its vertices, corresponding to the two irreducible components of $C_0$. Let $Z(u)$ (resp. $Z(v)$) be the corresponding closed subschemes of $\ca C$, and write $\ca I_u$ (resp. $\ca I_v$) for their ideal sheaves on $\ca C$. 

Now let $s\colon S \to \ca{M}$ be any scheme over $\ca{M}$. If $\ca C_S$ is the pullback of $\ca C$ to a stable curve over $S$, then $s^* \ca I_u$ and $s^*\ca I_v$ are coherent sheaves on $\ca C_S$. We define an \emph{enriched structure on $\ca C_S/S$} to be a pair
\begin{equation*}
(q_u\colon s^* \ca I_u \twoheadrightarrow \cl L_u, q_v\colon s^* \ca I_v \twoheadrightarrow \cl L_v)
\end{equation*}
of invertible quotients of the $s^*\ca I_{-}$, satisfying the condition that $\cl L_u \otimes \cl L_v$ is $S$-locally isomorphic to the trivial bundle on $\ca C_S$. 

For example, suppose $S \to \ca{M}$ factors via the complement of the closed point. Then clearly $s^*\ca I_u$ and $s^*\ca I_v$ are both trivial, hence so are $\cl L_u$ and $\cl L_v$, so there is exactly one enriched structure on $\ca C_S/S$. On the other hand, if $S \to \ca{M}$ factors via the inclusion of the closed point, then there is a bijection between enriched structures on $\ca C_S/S$ and $\bb G_m(S)$. The reader will then not be surprised to know that the moduli space $\ca E / \ca{M}$ of enriched structures can be obtained by blowing up $\ca{M}$ at the closed point, then deleting the two points where the strict transforms of the coordinate axes meet the exceptional fibre. 

If the moduli space of enriched structures is open in the blowup of $\ca{M}$ at the closed point, one might hope that the moduli of compactified enriched structures is simply the blowup of $\ca{M}$ at the origin, and indeed this is the case. To define the functor of compactified enriched structures, one takes pairs
\begin{equation*}
(q_u\colon s^* \ca I_u \twoheadrightarrow \cl F_u, q_v\colon s^* \ca I_v \twoheadrightarrow \cl F_v)
\end{equation*}
of torsion free rank 1 (see \ref{def:TFR1}) quotients of the $s^*\ca I_{-}$, and imposes a compatibility condition. To define the latter, denote by $\sigma_u$ (resp. $\sigma_v$) the tautological section\todo{Is this well-defined? It's not clear to me where, say, the closed point of $\ca{M}$ is sent in $\ca C$. David: Yes, is well-defined, since it is defined as universal deformation as pointed curve, so it comes with sections built in. } of $\ca C/\ca{M}$ through $Z(u)$ (resp. $Z(v)$). We write $\bb K_u$ for the kernel of the map $\sigma_u^*s^* \ca I_u \to \sigma_u^*\cl F_u$ induced by pulling back $q_u$ along $\sigma_u$,\todo{So we're pulling back $\ca I_u$ from $\ca C/\ca{M}$ to $\ca C_S/S$, and then somehow also along $\sigma_u\colon \ca{M}\to \ca C$? David: in the second pullback, we pull back along the base change of $\sigma_u$ to $S$. Not sure if we want another layer of notation for this, or if it is reasonably clear as is...} and define $\bb K_v$ similarly. Note that $\sigma_u^*s^* \ca I_u = s^*(x,y) = \sigma_v^*s^* \ca I_v$ where $(x,y)$ is the defining ideal of the closed point in $\ca{M}$. We then say that the pair $(q_u\colon s^* \ca I_u \twoheadrightarrow \cl F_u, q_v\colon s^* \ca I_v \twoheadrightarrow \cl F_v)$ is compatible (i.e.\ form a compactified enriched structure) if the canonical closed immersion 
\begin{equation*}
\on{Supp}_{\on{Fitt}}\frac{s^*(x,y)}{\bb K_u + \bb K_v} \to S
\end{equation*}
is an isomorphism (here $\on{Supp}_{\on{Fitt}}$ denote the Fitting-support of a coherent sheaf --- this is a slight variation on the usual support designed to behave better with respect to non-flat base-change, see \ref{def:fitting_support}). Note that the formation of $\bb K_u$ and $\bb K_v$ is stable under base-change by the $S$-flatness of $\cl F_u$ and $\cl F_v$. 

Again, if $S \to \ca{M}$ factors via the complement of the closed point then the $s^* \ca I_-$ are both trivial, so $\bb K_u = \bb K_v = 0$ and the pair is automatically compatible, so there is exactly one compactified enriched structure. If $S \to \ca{M}$ is the inclusion of the closed point then $s^*(x,y)$ is a free $k$-module of rank 2, $\bb K_u$ and $\bb K_v$ are both submodules of rank 1, so that pair is compatible if and only if $\bb K_u = \bb K_v$ as submodules of $s^*(x,y)$. More generally we find that the moduli space of compactified enriched structures is naturally isomorphic to the blowup of $\ca{M}$ at the closed point.

\subsection{Some of the complications arising in the general case}
The story told above for the 2-gon was hopefully straightforward (aided by the omission of all the proofs). In this section we outline some of the complications which arise in treating the general case, and how we handle them. 

\subsubsection{Graphs make no sense}
The crucial ideal sheaves $\ca I_u$ and $\ca I_v$ were defined in terms of dual graphs and irreducible components, but for arbitrary stable curves neither of these are defined. We begin by defining enriched structures on stable curves $C/S$ such that the tautological map $S \to \Mbar_{g,n}$ factors via an \'etale chart $U \to \Mbar_{g,n}$ such that the tautological curve $C_U/U$ is `controlled', i.e.\ some fibre has a dual graph which dominates the dual graph of all the other fibres. We show the definition is independent of the choice of $U$. We show that such $S \to \Mbar_{g,n}$ form a base for the big \'etale site on $\Mbar_{g,n}$ and that the resulting notion of enriched structure forms a sheaf on that base, and hence extends uniquely to a sheaf on the whole of the big \'etale site. We define this latter sheaf to be the sheaf of enriched structures. In fact, we carry out this procedure over the stack of all prestable curves for maximal generality. 

\subsubsection{What if the graph is not circuit--connected?}
Having one $\ca I_v$ for each vertex $v$ does not work well if the graph is not circuit-connected (\ref{def:circuit_connected}). We work instead with `relative components' of the graph, namely pairs $(v,G)$ where $v$ is a vertex and $G$ is a connected component of the complement of $v$. For example, if the graph has only a single vertex then the set of relative components is empty, so there is exactly one enriched structure (the empty tuple), as one would expect for a smooth curve.

When defining compactified enriched structures this setup is not quite general enough, essentially because of how graphs can change under specialisation (see below). We thus replace relative components with `hemispheres' of the graph, namely connected subgraphs whose complements are also connected. If $(v,G)$ is a relative component then $G$ is a hemisphere, but not vice-versa in general. 

\subsubsection{How to pull back as the graph changes?}To define a functor of enriched structures we need a good notion of pullback. However, the dual graph can change dramatically under pullback, as contracting edges may delete loops or break circuit-connectedness. In the case of enriched structures we work around these problems by showing that the set of enriched structures is empty unless the only pullbacks are of a very simple combinatorial shape. In the case of compactified enriched structures our notion of hemispheres allows us to work around this. 

\subsection{Why do we bother with non-compactified enriched structures?}
We will prove that the open substack of compactified enriched structures where the torsion free rank 1 quotients are invertible is canonically isomorphic to the stack of enriched structures. One might reasonably then ask why we bother with a separate definition of enriched structures with its own notion of compatibility. There are several possible answers:
\begin{enumerate}
\item The definition of enriched structures is rather simpler and easier to work/compute with than that of compactified enriched structures;
\item The definition of enriched structures is much closer to Main\`o's original definition;
\item To show that invertible compactified enriched structures over fields are the same as those arising from Main\`o's definition we would have to do much of the same work as we do in any case comparing enriched structures and invertible compactified enriched structures;
\item For the universal property concerning N\'eron models, it is much easier to work with enriched structures. 
\end{enumerate}

\subsection{Acknowledgements}

OB would like to thank Raoul Wols for pointing out a valuable reference for the appendix. DH would like to thank Martin Bright for his help in coming up with the definition of compatibility for compactified enriched structures, David Rydh for suggestions on proving representability, and the organisers of the AMS Summer Institute in Algebraic Geometry in Utah for an inspiring conference which provided a lot of motivation and context for this work. Both authors would like to thank Bas Edixhoven for a number of helpful comments and corrections. 

\subsection{Table of Notation}\label{sec:table_of_notation}

%Owen: I'm going to add every type of notation I see. That way we can try to detect namespace collisions. I realise this list will look ridiculous at first; we can pare it down later.

%Notation of section 2.

For the convenience of the reader, we list here some of the notation most frequently used in the paper. 

Graph theory notation:
\begin{longtable}{c|p{12cm}}
$\Gamma$						& A graph.\\
 $\on{Vert}\Gamma, \on{Edge}\Gamma$										&	The sets of vertices and edges of a graph $\Gamma$.\\
% $E$									& An arbitrary subset of the edges of a graph.\\
% $V, W, G, H$										& Subsets of the set of vertices of a graph.\\
 $G, G^c$									& $G$ is a set of vertices of a graph $\Gamma$, and $G^c$ the set of vertices of $\Gamma$ which are not in $G$.\\
% $v, v', w, g, g'$						& A vertex of a graph.\\
 %$e,e'$									& An edge of a graph.\\
 $e: v-v'$							    & An edge $e$ between vertices $v$ and $v'$ of a graph.\\
 $\gamma$								& A circuit of a graph (see \ref{def:graphs}).\\
 $E(V)$									& The set of edges of a graph with exactly one endpoint in a subset $V$ of the vertices.\\
 $\pi_0(\Gamma)$ & The set of connected components of graph $\Gamma$.\\
 $\Upsilon$ & A circuit-connected component of a graph (see \ref{lem:circuit_conn_partiton}). 
\end{longtable}

Commutative algebra and algebraic geometry notation:
\begin{longtable}{c|p{12cm}}
 $R$ & A ring, assumed to be commutative with unity.\\
 $\widehat{R}$ & The completion of a local ring $R$ at its maximal ideal.\\
 $\on{Spec} R$						& The spectrum of a ring $R$.\\
% $S, T$ & A scheme (or algebraic stack).\\
 $\bar s \to S$ & A geometric point of a scheme (or more generally algebraic stack) $S$.\\
 $\ca O_S$ & The structure sheaf on $S$.\\
 $M$ & A module over a ring, or a quasicoherent sheaf of modules.\\
 $I$ & An ideal of a ring, or a quasicoherent sheaf of ideals.\\
 $V(I)$	& The closed subscheme cut out by an ideal (sheaf) $I$.\\
 $\on{Fitt}_k(M)$ & The $k$th Fitting ideal of module (sheaf) $M$ (see \ref{sec:fitting_ideals}).\\
 $\on{Supp}(M)$ & The set-theoretic support of sheaf of modules $M$.\\
 $\on{Supp}_{\on{Fitt}}(M)$ & The closed subscheme cut out by the $0$th Fitting ideal of $M$ (see \ref{sec:fitting_ideals}).\\
 $Z, Z(\cdot)$ & A closed subscheme (possibly one determined by some data ``$\cdot$'').\\
  $\ca{O}_{S,\bar s}^{et}$	& The \'etale stalk of the structure sheaf of scheme $S$ at a geometric point $\bar s$.\\
%  $\sigma,\sigma'$ & Strata in the partitition of an affine scheme given by an SNC closed immersion.\\
  $\on{Irr}(C)$ & The set of irreducible components of a reduced scheme $C$ (often a prestable curve).
%  $A, X_i$									& Formal variables in a blowup ring.\\
%  $f$										& An element of a formal power series ring $R[[x,y]]$.\\
%  $f_x$									& The coefficient of $x$ in power series $f$.\\
\end{longtable}

Prestable curve notation:
\begin{longtable}{c|p{12cm}}
 $C/S$ & A prestable curve $C$ over scheme (or algebraic stack) $S$ (see \ref{sec:defining_curves}).\\
 $\frak M$ & The algebraic stack of all prestable curves (see \ref{sec:defining_curves}).\\
 $\frak C/\frak M$ & The universal prestable curve (see \ref{sec:defining_curves}).\\
 $\Mbar_{g,n}$ or $\ca M^{stab}_{g,n}$ & The Deligne-Mumford stack of stable curves of genus $g$ with $n$ marked points. \\%\Ocomment{Do I have this right?}{Almost; changed semistable to stable. The terminology is not completely consistent in the literature, but generally `semistable' means reductive automorphism group, and `stable' means finite automorphism group. If we want a DM stack we want finite automorphism groups. }\\
% $C_T$ & The pullback of a prestable curve $C$ (often the universal one) to a scheme $T$ over $\frak M$.\\
 $\Gamma_x$							& The dual graph of a prestable curve $C/S$ over a point $x\in S$.\\
 $\Omega_{C/S}$					& The sheaf of relative differentials of a prestable curve $C/S$.\\
 $\on{Sing}(C/S)$				& A canonical closed subscheme of curve $C/S$ where $C \to S$ is not smooth.\\
 $\frak s$		& A controlling point of $S$ for a prestable curve $C/S$.\\
  $\frak U\to\frak M, C/\frak U$ & A simple chart (see \ref{def:simple_chart}).\\
%  $\frak U_S$ & The immediate neighborhood of a controlled curve $C/S$ with controlled neighborhood $S\to\frak U$.
 $\frak U_S$								& The immediate neighbourhood of stack $S$ in chart $U$, see \ref{def:immediate_nhd}\\
% $Z, Z'$							& A closed subscheme.\\
 $Z(D)$								& The closed subscheme of a simple chart $U$ cut out by labels of edges in $D$ (see \ref{def:ideal_sheaves}).\\
 $\ca J_W$						& A sheaf of ideals on $\frak U$, see \ref{def:ideal_sheaves}.\\
% $\ca C_u$						& The tautological curve over a point $u$ of a simple chart of $\frak M$.\\
% $\eta(v)$						& The generic point of the irreducible component $v$ of a curve over a field.\\
% $\sigma_e$						& The map $Z(V)\to C_{Z(V)}$ corresponding to node $e\in E(V)$.\\
% $\widetilde{C}_{Z(V)}$	& The blowup of $C_{Z(V)}$ along the union of the sections $\sigma_e$ for $e\in E(V)$.\\
% $\bar{V}$							& The union of the connected components of $\widetilde{C}_{Z(V)}$ containing at least one $\eta(v)$ with $v\in V$.\\
 $\ca I_V$							& A sheaf of ideals on $C_\frak U$ (see \ref{def:ideal_sheaves}). \\
\end{longtable}

Miscellaneous notation: 

\begin{longtable}{c|p{12cm}}
% $\ca O_{S,s}$					& The Zariski local ring of algebraic stack $S$ at point $s\in S$.\Dcomment{This does not exist - check what is intended! }\\
% $\widehat{\ca O_{S,\bar s}}$	& The completed local ring of stack $S$ at geometric point $\bar s$.\\
% $D$ 									& A divisor.\\
% $I$									& A finite index set.\Dcomment{So this clashes with the ideal above, but I'm not sure if this can/should be avoided?}\\
% $f$									& A morphism of schemes.\Dcomment{Above it was an element of a power series ring, but morally they're the same kind of object, so I think that's OK. }\\
% $f^\ast D$						& The pullback of divisor $D$ on $S$ along $f:S'\to S$.\\
% $B, B_0$									& A certain collection of closed subsets of scheme $U$.\Dcomment{Removed}\\
% $\Delta$							& The boundary of a prestable curve.\\
% $\frak U_u$								& The largest open subset of $U$ on which the graph over $u$ makes sense. \\

% $s, t$									& A controlled neighborhood $S\to U$ of a curve $S\to\frak M$.\Dcomment{The $t$ looks odd... }\\
% $s^\ast\ca I_V$	& The pullback under $s: S\to U$ of an ideal sheaf $\ca I_V$ on $C_U$ to an $\ca O_C$-module.\\
% $\cl{L}, \cl{L}'$			& An invertible sheaf of $\ca O_C$-modules.\\
% $q, q', \psi$				& A map of $\ca O_C$-modules.\\
% $\pi_0(\Gamma)$		& The set of connected components of a graph $\Gamma$.\\
 $\ca E(S)$					& The set of enriched structures on $S$.\\
% $\on{sp}$					& A specialization map of graphs contracting certain edges.\\
%  $\{T_i\}_I$					& A family of controlled curves $T_i$ for each $i\in I$.\\
 $\mathbf{Set}$			& The category of sets.\\
 $\mathbf{Sch}_S$	& The category of schemes over algebraic stack $S$.\\
$\# F$ &  the cardinality of a (finite) set $F$.\\
% $\ca I_{V}$					& Given an inclusion of simple charts $U\subseteq U'$ with graph contraction $\Gamma'\to \Gamma$ and a subset $V\subset \Gamma$, the pullback of $\ca I_{V'}$ on $\ca C_{U'}$ to $\ca C_U$, where $V'$ is the set of vertices of $\Gamma'$ contracted into $V$.\Dcomment{Is this not the same as the obvious other definition of $\ca I_V$? }\\
% $\ca E_\Gamma(T)$		& The set of $\Gamma$-enriched structures on $T\to S$, where $S'$ is a controlled curve with graph $\Gamma'$, $\Gamma$ is a given contraction of $\Gamma'$, and $S$ is the open substack of $S$ cut out by the labels of edges contracted in $\Gamma'$ to obtain $\Gamma$.
$\tra$ & a closed immersion. \\
$\hra$ & an open immersion.\\
\end{longtable}

 We use lowercase letters both for points of schemes and for morphisms of schemes, adopting the view that a point $x\in X$ is a type of morphism to $X$, part of the ``functor of points'' view that every morphism of schemes $T\to X$ is a ``$T$-point'' of $X$. We often use $\eta$ for a point of a scheme that is meant to be ``generic'' in some way, either specifically as the generic point of a given irreducible subscheme of $X$, or more generally as any point of $X$ that specializes to a given $x\in X$.

%\subsection{General notation}
%\label{sec:general_notation}

\subsection{Notation for pullbacks}\label{sec:pullback_notation}
Suppose we are given a pullback diagram
\begin{equation*}
 \begin{tikzcd}
  X \times_S T \arrow[r] \arrow[d]& X \arrow[d, "\pi"]\\
  T \arrow[r, "t"] & S, \\
\end{tikzcd}
\end{equation*}
and a quasicoherent sheaf $\ca F$ on $X$, which we can pull back to $X \times_S T$. Often one would write $X_T$ for $X \times_S T$, and $\ca F_T$ for the pullback of $\ca F$ to $X_T$. However, it is really the map $t$ which is important, not just the data of the scheme $T$ --- it will often happen that $T = \on{Spec} k$ for some field $k$, and there are many possibilities for $t$, so we want to emphasise our choice. Because of this, we write $t^*X \coloneqq X \times_S T$, $t^*\pi$ for the natural map $t^*X \to T$, and symmetrically $\pi^*t$ for the natural map $X \times_S T \to X$. With this arranged, the pullback of $\ca F$ to $t^*X$ should correctly be denoted $(\pi^*t)^*\ca F$, but this is clumsy; we denote it simply $t^* \ca F$, which is shorter and unambiguous. 

This notation makes it important to remember where various sheaves live. To aid in this, we make the following conventions: if $C/S$ is a curve, we use $\ca I$ for ideal sheaves on $C$, and $\ca J$ for ideal sheaves on $S$. Similarly, we use $K$ for kernels of maps on $C$, and $\bb K$ for kernels of maps on $S$.

\part{Fine moduli of enriched structures}

\section{Defining enriched structures}

%\subsection{Background}
We start by briefly recalling some standard definitions, and fixing notation. % from \cite{Holmes2014Neron-models-an} and \cite{Holmes2014A-Neron-model-o}, where more details can be found.

\subsection{Graphs}

\begin{definition}\label{def:graphs}
By a \emph{graph} $\Gamma$ we shall mean an \emph{undirected multigraph}, that is, $\Gamma$ has a set of vertices $\on{Vert}(\Gamma)$ and a set of edges $\on{Edge}(\Gamma)$, and to each edge is associated an unordered pair of vertices called its \emph{endpoints}. A \emph{loop} is an edge whose endpoints are the same (in particular, we allow loops), and a \emph{circuit} is a path of positive length which starts and ends at the same vertex, and which does not repeat any edges or any other vertices.

We will also use the notion of a \emph{contraction} of graphs $\Gamma \to \Gamma'$, which means that $\Gamma'$ has been obtained from $\Gamma$ by removing a subset $E$ of the edges of $\Gamma$ and identifying the endpoints of each edge in $E$. %This is the only kind of mapping between graphs that we will consider; we will therefore use the phrase \emph{contraction} interchangeably with \emph{morphism of graphs}, although many notions of graph morphism exist in the literature.
\end{definition}

We now introduce two notions of connectedness for a graph: ordinary path-connectedness as well as circuit-connectedness.

\begin{definition}
The \emph{connected components} of a graph $\Gamma$ are the equivalence classes of $\on{Vert(\Gamma)}$ under the equivalence relation generated by the edges of the graph. A graph is \emph{connected} if it has exactly one connected component, i.e. if it is non-empty and if every two vertices can be connected by a path. 
\end{definition}
\begin{definition}
\label{def:circuit_connected}
Let $\Gamma$ be a graph. To any subset $E \sub \on{Edge}(\Gamma)$ we associate the unique subgraph of $\Gamma$ with edges $E$ and with no isolated vertices --- we will often fail to distinguish between the set $E$ of edges and the subgraph of $\Gamma$ obtained in this way. We say $E$ is \emph{circuit-connected} if it is non-empty and, for every pair $e$, $e'$ of distinct edges in $E$, there is at least one circuit $\gamma \sub E$ such that $e \in\gamma$ and $e' \in \gamma$. 
\end{definition}
\begin{remark}
Note that circuit-connectedness is not equivalent to the notion of 2-vertex-connectivity found in the literature, for example a graph with one vertex and one edge (a loop) is circuit-connected, but not 2-vertex-connected. 
\end{remark}
\begin{lemma}[{\cite[lemma 7.2]{Holmes2014A-Neron-model-o}}]\label{lem:circuit_conn_partiton}
For each graph $\Gamma$, the maximal circuit-connected subsets (\emph{circuit-connected components}) of $\on{Edge}(\Gamma)$ form a partition of $\on{Edge}(\Gamma)$. 
\end{lemma}
This partition can be thought of as breaking the graph into loops and 2-vertex-connected components, cf.\ \cite[\S 7.1]{Holmes2014A-Neron-model-o}. 

\subsection{Fitting ideals}\label{sec:fitting_ideals}

Here we recall the definition and key properties of Fitting ideals for finite modules over schemes, which among other uses provide a well-behaved notion of scheme-theoretic support for a finite quasicoherent sheaf of modules.

\begin{definition}\label{def:fitting_ideal}
Let $R$ be a ring (commutative, with unity) and $M$ a finitely generated $R$-module. Choose a presentation of $M$ with a finite list of generators $m_1,\dots,m_n$ and a (possibly infinite) family of relations $\{a_{1j}m_1 + \dots a_{nj}m_n : j\in J\}$. The \emph{$k$th Fitting ideal} of $M$ is the ideal $\on{Fitt}_k(M)\subseteq R$ generated by the $(n-k)\times (n-k)$ minors of the $n\times J$ matrix with entries $a_{ij}$. It is independent of choice of presentation; see \cite[\href{https://stacks.math.columbia.edu/tag/07Z8}{Tag 07Z8}]{stacks-project}.
\end{definition}

We note the following elementary results about Fitting ideals:

\begin{lemma}\label{fitt_0_in_ann} Let $R$ be a ring and $M$, $M'$ finite $R$-modules.
\begin{enumerate}
\item $\on{Fitt}_0 M \subseteq \on{Ann}_R M$ as ideals of $R$.
\item $V(\on{Fitt}_0) = \on{Supp}(M)$ as subsets of $\on{Spec}(R)$.
\item Given an $R$-algebra $R'$, if $I$ is the $k$th Fitting ideal of $M$ then $R'I$ is the $k$th Fitting ideal of $M' = R'\otimes_R M$.
\item If $M\to M'$ is a surjection, $then \on{Fitt}_k(M) \sub \on{Fitt}_k(M')$. 
\end{enumerate}
\end{lemma}
\begin{proof}\cite[\href{https://stacks.math.columbia.edu/tag/07ZA}{Tag 07ZA}]{stacks-project}
\end{proof}

From here, it is a straightforward matter to generalize Fitting ideals to quasicoherent modules over a scheme:

\begin{lemma}
 Let $S$ be a scheme and $M$ a finite type %{finite type? i.e. as in \cite[\href{https://stacks.math.columbia.edu/tag/01B5}{Tag 01B5}]{stacks-project}}
 quasicoherent $\ca O_S$-module. For each natural number $k$, there is a unique quasicoherent $\ca O_S$-ideal sheaf that assigns to each affine subscheme $X\subseteq S$ the ideal $\on{Fitt}_k(M(X))$. We denote this sheaf of ideals by $\on{Fitt}_k(M)$ as well and call it the \emph{$k$th Fitting ideal sheaf} of $M$.
\end{lemma}

\begin{proof}
 \cite[\href{https://stacks.math.columbia.edu/tag/0CZ3}{Tag 0CZ3}]{stacks-project}
\end{proof}

\newcommand{\Fsupp}{\on{Supp_{Fitt}}}
\begin{definition}\label{def:fitting_support}
For $S$ a scheme and $M$ a finite quasicoherent $\ca O_S$-module, we define the \emph{Fitting support} $\Fsupp M$ to be the closed subscheme of $S$ cut out by the ideal sheaf $\on{Fitt}_0(M)$. 
\end{definition}

\begin{remark}From the above, we see that the Fitting support coincides set-theoretically with the usual support, and its formation commutes with arbitrary base-change.\end{remark}

%\begin{remark}
% In particular, the closed subscheme $Z$ of $\on{Spec}(R)$ cut out by $\on{Fitt}_0(M)$ has $\on{Supp}(M)$ as its set-theoretic support, and if we form the closed subscheme $Z'$ of $\on{Spec}(R')$ cut out by $\on{Fitt}_0(M')$, we have $Z' = Z \times_{\on{Spec}(R)} \on{Spec}(R')$. Because of this compatibility with base-change, we regard the closed subscheme cut out by $\on{Fitt}_0(M)$ as the canonical notion of scheme-theoretic support of $M$, rather than that cut out by $\on{Ann}(M)$.
%\end{remark}

%\begin{lemma}
%Over a field, $\on{Fitt}_0M = 0 \iff M \neq 0$. Or maybe just for free modules? \Ocomment{Yes, this is true for free modules over a general ring. This fact doesn't seem to be used later, at least not explicitly. Should we just remove it? I'm having a hard time seeing when we wouldn't just use $\on{Fitt}_0 M \subseteq \on{Ann}_R M$.}
%\end{lemma}

\subsection{Prestable curves and their moduli}\label{sec:defining_curves}

\newcommand{\upsC}{\frak C}

\begin{definition}
Let $S$ be an algebraic stack (in the sense of \cite[\href{http://stacks.math.columbia.edu/tag/026O}{Tag 026O}]{stacks-project}). A \emph{prestable curve} over $S$ is a proper, flat, finitely presented morphism to $S$ which is relatively representable by algebraic spaces and whose geometric fibres are reduced, connected, have all irreducible components of dimension 1, and which have at worst nodal singularities.
\end{definition}

By \emph{curve}, we will always mean `prestable curve'. We denote by $\frak M$ the stack of all prestable curves. In other words, $\frak M$ is a stack in groupoids over the category $\cat{Sch}_{\bb Z}$ of schemes over $\bb Z$, whose objects are pairs $(S, C/S)$ with $S$ a scheme and $C/S$ a prestable curve. This is an algebraic stack, smooth over $\bb Z$, and is the disjoint union of open substacks of prestable curves of genus $g$, see \cite[\href{https://stacks.math.columbia.edu/tag/0E6S}{Section 0E6S}]{stacks-project}. 

We write $\upsC/\frak M$ for the universal prestable curve. By definition, to give a morphism from a scheme $S \to \frak M$ is the same as to give a prestable curve $C/S$. On the other hand, we can view $S$ as a stack (namely, as the category of schemes over $S$ together with the forgetful functor to $\cat{Sch}_\bb Z$), and then the functor $\cat{Sch}_S \to \frak M$ sends $T \to S$ to the pair $(T, C\times_S T /T) \in \on{ob} \frak M$. Now, the universal curve $\upsC /\frak M$ is a stack whose objects are triples $(T, C/T, f)$ of a scheme, a prestable curve, and a section $f\in C(T)$. If $C/S$ is a prestable curve, then we can view $C$ as a stack (namely, as $\cat{Sch}_C$), and define a functor $\cat{Sch}_C \to \upsC$ by sending a scheme $T \to C$ to the triple $(T, C\times_S T, f)$ where $f\colon T \to C \times_S T$ is the section induced by the map $T \to C$. 

Given a morphism from a scheme $S$ to $\frak M$, we write $C_S/S$ for the tautological prestable curve over $S$ (the pullback of $\upsC$ from $\frak M$ to $S$). If we start with a prestable curve $C/S$, this induces a morphism $S \to \frak M$, and thus a prestable curve $C_S/S$; reassuringly, by the above discussion we have $C_S = C$; these curves are \emph{canonically} identified as schemes over $S$. 
For example, given a prestable curve $C/S$ and a geometric point $\bar s \to S$ we may compose $s\to S \to \frak M$ and obtain the fiber $C_{\bar s}$ of $C$ over $\bar s$; more generally, we can apply this process to any map $T\to S$ and obtain the pullback $C_T/T = C\times_S T/T$.

The stack $\frak M$ is not quasi-compact or Deligne-Mumford, but does come with a very pleasant smooth cover by a Deligne-Mumford stack:
\begin{equation}\label{eq:cover_of_stack_of_prestable}
\bigsqcup_{(g,n): 2g-2 + n >0} \Mbar_{g,n} \to \frak M
\end{equation}%\Ocomment{The $2g-g$ makes me think there is a typo here. I can never remember what the actual condition is, but I'd love to know if it's different for stable/semistable.}{Thanks, well-spotted! Changed to $2g-2$. For the stable vs semistable thing: let's write $\ca M^{stab} = \Mbar$ for stable, and $\ca M^{ss}$ for semistable. Then $\ca M^{stab}_{g,n}$ and $\ca M^{ss}_{g,n}$ both make sense for all $g \ge 0$ and $n \ge 0$. However, $\ca M^{stab}_{g,n}$ is nonempty exactly when $2g-2 +n > 0$, whereas $\ca M^{ss}_{g,n}$ is nonempty exactly when $2g-2 + n  > -1$. E.g. a genus 0 curve with 0, 1 or 2 marked points can never have finite automorphism group, and a genus 0 curve with 0 or 1 marked points can never have reductive automorphism group (to see these, just think about trees of $\bb P^1$s). So actually, the disjoint union above could perfectly well be taken over all $g \ge 0$ and $n \ge 0$, as the extra bits would be empty. }
where we simply forget the markings. This map is smooth since the map $\Mbar_{g,n} \to \frak M$ is relatively representable by an open substack of the $n$-fold fibre product of smooth locus the universal curve over $\frak M$, and it is surjective since every prestable curve admits many sections \'etale-locally on the base. 

For many applications, it seems more interesting to work with the universal stable curve over $\Mbar_{g,n}$ for some $g$, $n$ in place of the stack $\frak M$ of prestable curves. The formation of the stack of enriched structures (and its compactification) are largely carried out smooth-locally on the target, and their formation commutes with base-change. Because of this, the reader who is only interested in the case  of stable curves will lose nothing by taking $\frak M$ to \emph{be} $\Mbar_{g,n}$ (over $\bb Z$, or their favourite ground field).

%In summary, almost nothing will be lost by the reader who takes $\ca{M}$ to be the stack of stable curves (perhaps of genus $g$) over $\bb C$, and such a reader can read `curve over $S$' to mean `stable curve (of genus $g$) over $S$' for any scheme $S$ over $\bb C$, and let the map $S \to \ca{M}$ be the tautological map. 

\subsection{Boundaries and labelled graphs of curves}

 Given a prestable curve $C/S$ over a scheme, and a geometric point $\bar s \to S$, the \emph{graph} $\Gamma_{\bar s}$ of $C/S$ at $\bar s$ is defined as in \cite[definition 10.3.17]{Liu2002Algebraic-geome} (who writes `semistable' where we write `prestable'); there is one vertex for each irreducible component of $C_{\bar s}$ and for every nodal point an edge connecting the vertices corresponding to the components it lies on. If $e$ is an edge of $\Gamma_{\bar s}$, i.e.\ a nodal point of $C_{\bar s}$, then we label $e$ by the principal ideal $(a) \sub \ca{O}^{et}_{S,\bar s}$ such that 
\begin{equation}\label{eq:local_structure_label}
\widehat{\ca{O}}_{C,e}^{et} \cong \frac{\widehat{\ca{O}}_{S,\bar{s}}^{et}[[x,y]]}{(xy-a)}
\end{equation}
as $\widehat{\ca{O}}_{S,\bar{s}}^{et}$-algebras --- see \cite[proposition 2.5]{Holmes2014Neron-models-an}.

\todo{Do we still want this definition? Isn't it cleaner to talk about the map to $\frak M$ being smooth? }

\begin{definition}
Let $C/S$ be a prestable curve over a scheme, and $\bar s \to S$ a geometric point. We say $C/S$ has \emph{normal crossings singularities at $\bar {s}$} if for every subset $E\subseteq \on{Edge}\Gamma_{\bar s}$, the closed subscheme of $\on {Spec} \ca O_{S,\bar s}^{et}$ cut out by the labels of the edges in $E$ is regular as a scheme, and has codimension equal to the cardinality of $E$. 

We say $C/S$ has \emph{normal crossings singularities} if it has them at all geometric points of $S$. The notion of having normal crossings singularities is local on $S$ in the smooth topology, so the definition naturally extends to prestable curves $C/S$ where $S$ is an algebraic stack. 
\end{definition}

\begin{definition}
Let $C/S$ be a prestable curve over a scheme, with relative differentials $\Omega_{C/S}$. The first Fitting ideal (see \cite[Chapter 20]{Eisenbud1995Commutative-alg} or \ref{def:fitting_ideal}) of $\Omega_{C/S}$ is a sheaf of ideals on $C$, hence cuts out a closed subscheme $\on{Sing}(C/S)$ of $C$, whose complement is exactly the locus where the map $C\to S$ is smooth. The \emph{boundary} of $C/S$ is defined as the closed subscheme of $S$ cut out by the 0th Fitting ideal of the structure sheaf of $\on{Sing}(C/S)$. 
\end{definition}
%Let $C/S$ be a prestable curve over a scheme, with relative differentials $\Omega_{C/S}$. The first Fitting ideal (see \cite[Chapter 20]{Eisenbud1995Commutative-alg}) of $\Omega_{C/S}$ is a sheaf of ideals on $C$, hence cuts out a closed subscheme $\on{Sing}(C/S)$ of $C$, whose complement is exactly the locus where the map $C\to S$ is smooth. The \emph{boundary} of $C/S$ is defined as the schematic image of $\on{Sing}(C/S)$ in $S$; equivalently, it is the closed subscheme of $S$ cut out by the annihilator ideal sheaf of the structure sheaf of $\on{Sing}(C/S)$. 
%Owen: Should we use the Fitting ideal again for the boundary, instead of the annihilator? It's possible that they're the same in this case.

\begin{remark}\label{rem:boundary}
Set-theoretically the boundary consists of those points of $S$ over which $C$ is not smooth. The boundary is locally cut out by a single equation. Suppose $S$ is strictly hensellian local (in particular the graph over the closed point makes sense), and that $C/S$ has node $e$ of the graph of the closed fibre labelled by $(\alpha(e)) \sub \ca O_S(S)$. Then a simple computation from \ref{eq:local_structure_label} and the definition of the 1st Fitting ideal yields
\begin{equation*}
\on{Sing}(C/S) = \on{Spec}\left( \prod_{e \in nodes} \ca O_S(S)/(\alpha(e))\right). 
\end{equation*}
The 0th fitting ideal of this module is then simply $\left(\prod_{e \in nodes} \alpha(e)\right)$. %Note that as $S$ is strictly hensellian, the residue field of the closed point is separably closed, and hence the graph at this point makes sense (nodes, irreducible components etc are all defined over separable extensions). 

%\Ocomment{So far we have only discussed the graphs over each geometric point. I assume we are choosing an algebraic closure of the residue field over the closed point of $S$, but then how do we get the full singular locus?}{Since $S$ is strictly hensellian, the residue field of the closed point is separably closed, and this is enough to define the graph (nodes and irreducible components etc are all defined ver separable extensions). }

To specialise further, assume $S = \on{Spec} k[[t]]$ and $C/S$ is regular, with $n$ nodes on the special fibre. Then the boundary is cut out by $(t^n)$. This is of course \emph{not} in general equal to the schematic image of $\on{Sing}(C/S)$ in $S$ (the latter being cut out by $(t)$. However, we will mostly use this concept in the case where $C/S$ has normal crossings singularities, in which case the boundary is easily seen to coincide with the schematic image in $S$ of $\on{Sing}(C/S)$. The advantage of the definition of the boundary given here is that its formation commutes with arbitrary base-change. 
\end{remark}

%{Fun - they're not the same, for example for a regular curve over a DVR with more than one node. The annihilator is just $t$, the fitting ideal gives $t$ to the number of nodes. Which do we want? Well, we only talk about it in the context of a `normal crossings boundary', whose definition is independent of which we use, and which implies they are the same. It is rather clear that this is worthy of a little more explanation...}

%Set-theoretically the boundary consists of those points of $S$ over which $C$ is not smooth. The boundary is locally cut out by a single equation. If $C/S$ has normal crossings singularities then the boundary is a normal crossings divisor, but the converse is not true (think of any generically smooth curve over a trait with more than one singular point on the special fibre). 

\begin{definition}
A prestable curve $C/S$ is \emph{quasisplit} if every irreducible component of every fibre is geometrically irreducible, and if the structure morphism $\on{Sing}(C/S) \ra S$ of the non-smooth locus is an immersion Zariski-locally on the source. 
\end{definition}
A prestable curve over a strictly Henselian local scheme is always quasi-split, and by a limiting argument we find that every prestable curve becomes quasi-split after some \'etale cover. If $C/S$ is quasisplit and $s \in S$ is a point then the graph of $C/S$ at $s$ makes sense without making a choice of algebraic closure of the residue field, and its labels naturally lie in the Zariski local ring $\ca{O}_{S,s}$ (see \cite{Holmes2014A-Neron-model-o} for more details). 

Let $C/S$ be a quasisplit prestable curve. If $s$, $\eta$ are two points in $S$ with $s \in \overline {\{\eta\}}$ then we have a natural `specialisation map' $\Gamma_s \to \Gamma_\eta$ which contracts those edges whose labels become units at $\eta$, and replaces the labels on the other edges by their images in the local ring at $\eta$, see \cite{Holmes2014A-Neron-model-o}. 

\subsection{Controlled curves and charts}
\newcommand{\schart}{\frak U}

%\subsection{The largest open on which a graph makes sense}
\subsubsection{Definitions}
\begin{definition}
Let $C/S$ be a quasi-split curve. If $s$, $x$ are two points in $S$ we say \emph{$x$ is controlled by $s$} if there exists a point $\eta_x \in S$ such that $x$ and $s$ both lie in $\overline{\{\eta_x\}}$ and such that the specialisation map $\Gamma _x \ra \Gamma_{\eta_x}$ is an isomorphism. We say \emph{$s$ is a controlling point for $S$} if every $x \in S$ is controlled by $s$.
\end{definition}
 Roughly, this says that every graph of $C/S$ is canonically a contraction of $\Gamma_s$. In the case of the universal deformation $\ca C/\ca M$ of the 2-gon considered in \ref{sec:2-gons}, the unique controlling point is the closed point of $\ca{M}$. We usually denote controlling points by fraktur letters. 

Suppose $C/S$ has a controlling point $\frak{s} \in S$. Let $\frak{s}'$ be any other controlling point. Then the graphs $\Gamma_\frak{s}$ and $\Gamma_{\frak{s}'}$ are canonically identified, so we have a graph for the whole $C/S$, without reference to a specific controlling point, motivating 
\begin{definition}
If $C/S$ has a controlling point we call it \emph{weakly controlled}, and we call the graph a \emph{controlling graph}.
\end{definition} 
Every $C/S$ is weakly controlled $S$-\'etale locally, and if $C/S$ is weakly controlled then in particular $S$ is connected.

\begin{definition}\label{def:SNC}
We say a closed immersion to an affine scheme $Z \tra X$ is \emph{SNC}   if there exists a finite set $F \sub \ca O_X(X)$, such that 
\begin{enumerate}
\item
the ideal sheaf of $Z$ is the ideal generated by the product of all the elements of $F$; 
\item For every subset $F_0 \subseteq F$, the closed subscheme cut out by the ideal generated by $F_0$ is smooth over $\bb Z$, is connected, and is of codimension $\# F_0$ in $X$.
\end{enumerate}
\end{definition}
This concept is referred to in the literature both as (relative) \emph{strict normal crossings} and as \emph{simple normal crossings}, but the acronym is happily the same. 

By taking $F_0$ empty we see that $X$ itself must be smooth and connected, in particular it is non-empty, jacobson and equidimensional. The prototypical example of $Z \tra X$ is the inclusion of the union of the coordinate hyperplanes into affine space; and every other example is \'etale-locally isomorphic to this one. How unique is this $F$? If $F'$ is another set of elements satisfying the same conditions, then there is a unique bijection $\phi\colon F \to F'$ such that $f$ differs from $\phi(f)$ by multiplication by a unit in $\ca O_X(X)$. 

An SNC closed immersion $Z \tra X$ induces a partition of $X$ into smooth connected locally-closed subschemes (a \emph{stratification}), by taking intersections and complements of subschemes defined by the vanishing of the elements of $F$. By our connectedness assumptions, there is exactly one closed and one open stratum. We order the strata of $X$ by `inclusion in the closure', i.e. we set
\begin{equation*}
\sigma \le \sigma' \iff \sigma \sub \overline {\sigma'}, 
\end{equation*}
so the closed stratum is the minimal element, and the open stratum is the maximal element. More generally, the union of a downward-closed set of strata is closed, and the union of an upward-closed set of strata is open. 

\begin{definition}\label{def:simple_chart}
A \emph{simple chart} is a smooth morphism from an affine scheme $\schart \to \frak M$ such that the boundary of $C_\schart/\schart$ is  SNC, and such that $C_\schart/\schart$ is quasi-split. 
\end{definition}
 The symbol $\schart$ will be reserved for simple charts. We will soon prove that simple charts are weakly controlled. 
 
 \begin{definition}\label{def:contolled}
We say a curve $C/S$ is \emph{controlled} if
\begin{enumerate}
\item $C/S$ is weakly controlled; %  (i.e.\ has a controlling point $\frak{s}\in S$);
\item there exists a factorisation $S \ra \schart \ra \frak{M}$ where $\schart \ra \frak{M}$ is a simple chart (we call this a \emph{controlled neighbourhood}). 
\end{enumerate}
\end{definition}
We will soon prove that simple charts are controlled in the sense of definition \ref{def:contolled}, making the terminology reasonable. 

\subsubsection{Basic properties of simple charts} 
 
\begin{lemma}
Let $C/\schart$ be a simple chart, and $u \in \schart$ be a point with graph $\Gamma$. Let $F$ be as in \ref{def:SNC}, and let $F_u \sub F$ be the set of those elements vanishing at $u$. Then the graph over $u$ has exactly $\# F_u$ edges, and the labels of these edges are generated by the images of the elements of $ F_u$ in $\ca O_{\schart, u}$. 
\end{lemma}
\begin{proof}
We can check these claims after passing to the strict henselisation of $\schart$ at $u$, whereupon they follow from \ref{rem:boundary}. 
\end{proof}

Since no edges are contracted when we move within a stratum, we have
\begin{lemma}
Let $C/\schart$ be a simple chart, and $u$, $u' \in \schart$ points in the same stratum, with $u \in \overline{\{ u'\}}$. Then the contraction $\Gamma_{u'} \to \Gamma_u $ is an isomorphism. 
\end{lemma}
Since each stratum has a unique generic point, the graphs at all points within the same stratum are \emph{canonically} identified. In this way, each stratum $\sigma$ has an associated graph $\Gamma_\sigma$, and the edges can be seen as having labels taken from the elements of $F$. If a stratum $\sigma$ is contained in the closure of another $\sigma'$ (i.e. $\sigma \le \sigma'$, then we have a contraction $\Gamma_{\sigma} \to \Gamma_{\sigma'}$, contracting exactly those edges whose labels are units on the stratum $\sigma'$. Since the closed stratum lies in the closure of the generic point of every other stratum, we in particular obtain
\begin{lemma}
Let $C/\schart$ be a simple chart, and $u \in \schart$ a point lying in the closed stratum. Then $u$ is a controlling point for $C/\schart$. 
\end{lemma}
A simple chart $\schart \to \frak M$ is therefore itself controlled (as a simple chart for $\schart$, just take $\schart \stackrel{id}{\to}\schart$!).

%
%\begin{lemma}
%Let $C/\schart$ be a simple chart, and let $u$, $u'$ lie in the same stratum, with canonically identified graphs $\Gamma = \Gamma'$. Let $e $ be an edge of $\Gamma$, with label $\ell(e)$ in the local ring $\ca O_{\schart, u}$, and label $\ell'(e)$ in the local ring $\ca O_{\schart, u'}$. Write $Z(\ell(e))$ for the closed subscheme of $\on{Spec}\ca O_{\schart, u}$ cut out by $\ell(e)$, and $\overline{Z(\ell(e))}$ for its schematic closure in $\schart$. Similarly, define $Z(\ell'(e))$ to be the closed subscheme of $\on{Spec}\ca O_{\schart, u'}$ cut out by $\ell'(e)$, and $\overline{Z(\ell'(e))}$ for its closure in $\schart$. Then 
%\begin{equation*}
%\overline{Z(\ell(e))} = \overline{Z(\ell'(e))}. 
%\end{equation*}
%\end{lemma}
%\begin{proof}
%We may assume $u$ to lie in the closure of $u'$. Replacing $\schart$ by a neighbourhood of $u$, we may assume $\schart$ to be affine. 
%
%\end{proof}

%The key objects we will be interested in are \emph{controlled curves} (\ref{def:contolled}), which are themselves weakly controlled, \emph{and} which admit `controlled smooth neighbourhoods'; for example, to define enriched structures we will start by defining them for such curves, then extend formally to the general case (see \ref{sec:general_ES}). 

\begin{lemma}\label{lem:locally-controlled}
Let $S \ra \frak{M}$ be any curve, with $S$ a scheme. There exists an \'etale cover $\bigsqcup_{i \in I} S_i \ra S$ such that each $S_i \ra \frak{M}$ is controlled. %{DO we want this $S$ a stack or scheme? If a stack, need to replace etale by smooth. If a scheme, needs an argument (see \cite[\href{https://stacks.math.columbia.edu/tag/055S}{Tag 055S}]{stacks-project}). Also fix proof!Later: seems only for $S$ a scheme. }
\end{lemma}
\begin{proof}
Firstly, from (\oref{eq:cover_of_stack_of_prestable}) it is clear that we can find a smooth cover of $\frak{M}$ by simple charts. Moreover,  any family of curves has an \'etale cover by controlled families. The result follows by combining these statements. 
\end{proof}

\subsubsection{Immediate neighbourhoods of controlled curves}
\label{sec:largest_open_where_graph_makes_sense}\label{lem:factor-through-largest-open}\label{thm:univ_prop_U_u}\label{lem:largest-open-containment}

Let $C/S$ be a controlled curve with controlled neighbourhood $S \to \schart \to \frak M$. Since any two controlling points of $C/S$ lie in the closure of a third, and the graphs over these points are canonically identified via specialisation, we see
\begin{lemma}
There exists a stratum $\sigma_S$ of $\schart$ such that every controlling point of $S$ lands in $\sigma_S$. 
\end{lemma}
\begin{definition}\label{def:immediate_nhd}
Let $\schart_S \sub S$ be the union of those strata $\sigma \ge \sigma_S$. Then $\schart_S$ is affine open in $\schart$, and the morphism $S \to \schart$ factors via $\schart_S \hra \schart$; we call $\schart_S$ the \emph{immediate neighbourhood of $S$ in $\schart$}. 
\end{definition}

If $F \sub \ca O_\schart(\schart)$ is as in \ref{def:SNC}, we see that $\schart_S$ is obtained from $\schart$ by inverting those $f \in F$ which pull back to units on $S$; in particular, $\schart_S$ is itself a simple chart for $S$. 

If $S \to \schart$ is a controlled neighbourhood and $S \to \schart_S \hra \schart$ is the immediate neighbourhood of $S$ in $\schart$, then $\schart_S$ is also the immediate neighbourhood of $S$ in $\schart_S$.

\begin{example}
If $\schart \subseteq \on {Spec} k[x,y]$ is a deformation of a 2-gon with boundary $xy=0$, then
\begin{itemize}
\item
If $u$ is the generic point of $\schart$ then $\schart_u = \schart \setminus (xy=0)$;
\item If $u$ is the generic point of the divisor $(x=0)$ then $\schart_u = \schart \setminus (y=0)$;
\item If $u$ is the closed point of $\schart$ then $\schart_u = \schart$. 
\end{itemize}
\end{example}

\subsection{Enriched structures on controlled curves}\label{sec:ES_on_controlled_curves}

Let $\schart$ be a simple chart with controlling point $\frak u$, with graph $\Gamma$. Let $e$ be an edge in $\Gamma$, and let $\ell(e) \in \ca O_\schart(\schart)$ be the label of $e$, defined up to multiplication by a unit. Let $Z = Z(\ell(e))$ be the closed subscheme of $\schart$ cut out by $\ell(e)$. 
\begin{lemma}\label{lem:node_sections}
There exists a unique $\schart$-map $Z \to C_\schart$ landing in the non-smooth locus and which, on restriction to the fibre over $\frak u$, is the inclusion of the node corresponding to the edge $e$. 
\end{lemma}
\begin{proof}
Writing $C_Z = C_\schart \times_\schart Z$, we must construct an appropriate section of the projection $C_Z \to Z$. The node in $C_\frak u$ corresponding to the edge $e$ lies in $\on{Sing}(C_Z/Z)$; we write $T$ for the connected component of $\on{Sing}(C_Z/Z)$ containing it. It suffices to show that the map $T \to Z$ is an isomorphism. Now $T \to Z$ is finite unramified, and by looking \'etale-locally at the equations we can see that it is even \'etale. Thus $T \to Z$ is a finite \'etale morphism of connected schemes smooth over $\bb Z$ (in particular, geometrically unibranch), and by our quasi-splitness assumption it admits a section on some Zariski neighbourhood of $u$, thus by \cite[\href{https://stacks.math.columbia.edu/tag/0BQI}{Tag 0BQI}]{stacks-project} $T \to Z$ is an isomorphism. 
\end{proof}
%\Ocomment{I have no objection to this argument---I certainly can't think of any finite \'etale morphism of connected smooth-over-$\mathbb{Z}$ schemes with a section that isn't an isomorphism---but is there some theorem that says this or is it supposed to be obvious?}{Follows from \cite[\href{https://stacks.math.columbia.edu/tag/0BQI}{Tag 0BQI}]{stacks-project} - reference added. }

Let $D$ be a subset of the edges of $\Gamma$, and write $Z(D)$ for the subscheme of $\schart$ cut out by the labels of edges in $D$. Write $C_{Z(D)} = C_\schart \times_\schart Z(D)$, a prestable curve over $Z(D)$. Write $\Gamma \setminus D$ for the graph obtained from $\Gamma$ by deleting those edges in $D$, and write $\pi_0(\Gamma \setminus D)$ for its set of connected components. Write $\on{Irr}(C_{Z(D)})$ for the set of irreducible components of the (reduced) scheme $C_{Z(D)}$. 
\begin{lemma}\label{lem:connected_irred_components}
There is a unique bijection 
\begin{equation}
\phi\colon \pi_0(\Gamma \setminus D) \to \on{Irr}(C_{Z(D)})
\end{equation}
such that, on the fibre over the controlling point $\frak u$, vertices of a connected component $G$ correspond to irreducible components of $C_{\frak u}$ which lie in $\phi(G)$. 
\end{lemma}
%\Ocomment{This is a placeholder comment to remind myself where I am in the latest read-through.}
\begin{proof}
Write $\eta$ for the generic point of the integral scheme $Z(D)$. Then the graph $\Gamma_\eta$ over $\eta$ is obtained from $\Gamma$ by contracting exactly those edges \emph{not} in $D$. Thus the vertices of $\Gamma_\eta$ are in natural bijection with $\pi_0(\Gamma \setminus D)$. On the other hand, the vertices of $\Gamma_\eta$ are also in natural bijection with the irreducible components of $C_\eta$. It thus suffices to construct a suitable bijection $\on{Irr}(C_\eta) \to \on{Irr}(C_{Z(D)})$. 

Applying \ref{lem:node_sections} yields for every edge in $D$ a section of $C_{Z(D)} \to Z(D)$ landing in the non-smooth locus. Write $\tilde C_{Z(D)}$ for the blowup of $C_{Z(D)}$ at the union of these sections (this is the same as the normalisation of $C_{Z(D)}$, as can be checked by an \'etale-local computation). Each connected component of $\tilde C_{Z(D)}$ is a generically-smooth prestable curve over $Z(D)$, and so is irreducible, hence the connected components of $\tilde C_{Z(D)}$ are naturally in bijection with the irreducible components of $C_{Z(D)}$. These connected components also correspond bijectively to the irreducible components of $C_\eta$ (since blowups commute with flat base-change, and the claim is true for a curve over a field). 
\end{proof}

Now let $W$ be a subset of the vertices of $\Gamma$, and write $E(W)$ for the set of edges with exactly one endpoint in $W$. Then $Z(E(W))$ is the closed subscheme of $\schart$ cut out by the labels of edges in $E(W)$, and we have a bijection 
\begin{equation}
\phi\colon \pi_0(\Gamma \setminus E(W)) \to \on{Irr}(C_{Z(E(W))}). 
\end{equation}
Let $T_W$ be the disjoint union of the $\phi(G)$ as $G$ runs over connected components of $\Gamma \setminus E(W)$ which contain a vertex in $W$; symbolically, 
\begin{equation*}
T_W = \bigsqcup \{\phi(G) : G \in \pi_0(\Gamma \setminus E(W)),  G \cap W \neq \emptyset\}. 
\end{equation*}
Equipping $T_W$ with its reduced induced scheme structure, we have a natural map $T_W \to C_\schart$. 

\begin{lemma}
The natural map $T_W \to C_\schart$ is a closed immersion. 
\end{lemma}
\begin{proof}
It suffices to show that the natural map $T_W \to C_{Z(E(W))}$ is a closed immersion. Since the latter map is the inclusion of a disjoint union of irreducible components, it suffices to check that these irreducible components have no intersection. But this is clear from the definition of $E(W)$, since no two connected components of $\Gamma \setminus E(W)$ can be connected by an edge. 
\end{proof}

\begin{definition}\label{def:ideal_sheaves}
Given $W$ as above, we define $\ca J_W$ to be the sheaf of ideals on $\schart$ cutting out $Z(E(W))$, and $\ca I_W$ to be the sheaf of ideals on $C_\schart$ cutting out $T_W$. 
\end{definition}
The sheaf $\ca I_W$ will play a central role in everything that follows; $\ca J_W$ will only become important in \ref{part:compactifying}.

\begin{example}[The 2-gon]\label{example:2_gon_T_W}
With a deformation of the 2-gon as discussed in \ref{sec:2-gons}, the graph is a 2-gon with labels $(u)$ and $(v)$. We consider several possible sets $W$:
\begin{enumerate}
\item
If $W$ consists of a single vertex then $Z(E(W))$ is the closed point, and $T_W$ will be the corresponding irreducible component of the special fibre;
\item
If $W$ consists of both vertices then $Z(E(W))$ is the whole of $\schart$ and $T_{W}$ is the whole of $C$;
\item
If $W$ is empty then $Z(E(W))$ is again the whole of $\schart$, but $T_{W}$ is empty. 
\end{enumerate}

\end{example}
\subsubsection{Enriched structures}
\newcommand{\gdiff}{-}

If $\Gamma$ is a graph with vertex set $V$ and $W \subseteq V$, we write $\Gamma \gdiff W$ for the subgraph of $\Gamma$ induced by $V \setminus W$. If $W = \{ v \}$ is a singleton, we denote this simply $\Gamma \gdiff v$. 

\begin{definition}\label{def:relative_component}
Let $\Gamma$ be a connected graph. A \emph{relative component} of $\Gamma$ is a pair $(v,G)$ where $v$ is a vertex of $\Gamma$ and $G$ is a connected component of $\Gamma \gdiff v$. An edge from $v$ to $G$ is called a \emph{separating edge} of the relative component. We write $G^c$ (the `complement of $G$') for  the set of vertices of $\Gamma$ that are not in $G$. 
%Let $\schart \ra \ca{M}$ a  simple chart, $u \in \schart$, and $\schart_u \sub Q$ as above. Let $\Gamma$ be the graph over $u$. Given a connected set $V$ of vertices of $\Gamma$, a \emph{component of $\Gamma$ relative to $V$} is an element of the set $\pi_0(\Gamma \gdiff V)$. For this part of the paper we will be mainly interested in the case where $V$ contains a single vertex. We will then just talk about `connected components relative to $v$' (rather than to $\{v \}$); we will distinguish this by using lower-case $v$. We call a pair $(V,G)$ consisting of a connected set $V$ of vertices and a connected component of $\Gamma$ relative to $V$ a \emph{relative component} of $\Gamma$. 
%Given a relative component $(V,G)$ of $\Gamma$, we write $\ca{I}_{G^c}$ for the ideal sheaf associated to the set of vertices of $\Gamma$ \emph{not} contained in $G$ (so in particular this always includes the vertices in $V$, and is equal to $V$ if $\Gamma \gdiff V$ is connected). 
\end{definition}
Note that the set of relative components is empty if and only if $\Gamma$ has exactly one vertex. Note also that $G$ determines $v$ uniquely, but not vice-versa in general (though $v$ does determine $G$ if $\Gamma$ is circuit-connected). 

\begin{definition}[The set of enriched structures]\label{def:controlled_ES}%{Actually want something a bit more general: let $\Gamma' \ra \Gamma$ be a surjection on to the graph of $S$. Then we should define enriched structures with respect to this graph. Define an ES to be one where the surjection if identity. Then show that if the surjection satisfies (*) then it actually gives an equivalent definition. See p28 of red book. Later: This df should some later, and remark that we can recover this one. }
Let $C/S$ be a controlled curve with controlling point $\frak{s}$ and $S \stackrel{s}{\ra} \schart \ra \frak{M}$ a controlled neighbourhood. Write $u = s(\frak{s})$, and let $\Gamma \coloneqq \Gamma_\frak s$ be the graph of $\frak s$, and $\schart_S$ be the immediate neighborhood of $S$ in $\schart$.  For each relative component $(v,G)$ of $\Gamma$ we have an ideal sheaf $\ca I_{G^c}$ on $C_{\schart_S}$, which we can pull back along $s$ to $C$ (c.f. \ref{sec:pullback_notation}).

An \emph{enriched structure} on $C/S$  with respect to $\schart$ consists of, for each relative component $(v,G)$ of $\Gamma$, an invertible quotient 
\begin{equation}
q_{v,G}\colon s^*\ca{I}_{G^c} \twoheadrightarrow \cl{L}_{(v,G)}, 
\end{equation}
(i.e. a surjective map of sheaves on $C$ where $\cl{L}_{(v,G)}$ is invertible), such that $S$-locally there exists an isomorphism 
\begin{equation}\label{es:compatibility}
\bigotimes_{\stackrel{v \in \on{Vert} \Gamma}{G\in \pi_0(\Gamma \gdiff v)}} \cl{L}_{(v,G)} \cong \ca{O}_C.
\end{equation}
We say two enriched structures $(q_{v,G}\colon s^*\ca{I}_{G^c} \twoheadrightarrow \cl{L}_{(v,G)})_{v,G}$ and $(q'_{v,G}\colon s^*\ca{I}_{G^c} \twoheadrightarrow \cl{L}'_{(v,G)})_{v,G}$ (with respect to $\schart$) are \emph{equivalent} if for every vertex $v$ and every $G \in \pi_0(\Gamma\gdiff v)$ there exists an isomorphism $\psi_{v,G}\colon \cl{L}_{(v,G)} \ra \cl{L}'_{(v,G)}$ making the obvious triangle commute. 
We write $\ca E(C/S)$ for the set of equivalence classes of enriched structures on $C/S$.
\end{definition}
Later, we will extend the construction of $\ca E$ to a functor, and extend it to all curves (not just controlled ones). 

Note that the set $\ca E(C/S)$ can be empty. Indeed, in \ref{sec:maino_blowups} we will see that the moduli of enriched structures  on $C/S$ is sub-functor of a blowup of $S$, so we should not expect it to have $S$-points in general. 

%For example, in the setting of \ref{example:2_gon_T_W}, we will see in \ref{asd} that the moduli of enriched functors is a sub-functor of the blowup of $\schart$ at the ideal $(x,y)$; in particular, $\ca E(\schart)$ is empty. 
%Morphisms of enriched structures will be defined in \ref{subset:pullback_of_ES}. 

\begin{remark}[Changing the neighbourhood]\leavevmode\label{rem:changing_the_neighbourhood}
\begin{enumerate}
\item
If $S \ra \schart' \ra \frak{M}$ is another controlled neighbourhood of $S \ra \frak{M}$, then $\schart \times_\frak{M} \schart'$ is also a controlled neighbourhood. 
\item Let $S \ra \schart' \ra \schart \ra \frak{M}$ be a refinement of controlled neighbourhoods. 
%Let $\frak s \in S$ be a controlling point, mapping to $u' \in \schart'$ and $u \in \schart$. 
Then (see \ref{thm:univ_prop_U_u}) the map $\schart'_{S} \ra \schart$ factors via $\schart_S$. Write $\Gamma$ for the graph of $\frak s$ (equivalently the graph of $u$ or $u'$) and let $(v,G)$ be a relative component. Write $\ca{I}_{G^c}$ for the ideal sheaf over $\schart_S$ and $\ca{I}'_{G^c}$ for that over $\schart'_{S}$. Then it is clear that the pullback of $\ca{I}_{G^c}$ to the curve over $\schart'_{S}$ is exactly $\ca{I}'_{G^c}$. As such, we see that enriched structures on $C/S$ with respect to $\schart$ are canonically the same as those taken with respect to $\schart'$. 
\item Combining the above two points, we see that the definition of an enriched structure does not depend on the choice of controlled neighbourhood $\schart$. 
\end{enumerate}
\end{remark}

\begin{example}[Compact type]\label{eg:cpct_type}
Suppose that $C/S$ is of compact type, so that the graph $\Gamma$ is a tree. Then $C_{\schart_S}/\schart_S$ is also of compact type. If $(v,G)$ is a relative component, then there is exactly one edge between $G$ and $G^c$, from which it follows that $T_G^c$ is a Cartier divisor, so the ideal sheaf $\ca I_{G^c}$ on $C_{\schart_S}$ is invertible. The pullback to $C$ (denoted $s^*\ca I_{G^c}$ is therefore also invertible, so that it has a unique invertible quotient (namely, itself). This immediately implies that there is at most one enriched structure on a curve of compact type, and with a little more work one can check that the compatibility \ref{es:compatibility} holds, so that there is exactly one enriched structure. 
\end{example}

\subsection{Pulling back enriched structures}\label{subsec:pullback_of_ES}
In this section we will define the pullback of enriched structures. This will define the functor of enriched structures; it then makes sense to ask whether it is representable (indeed it is; see \ref{cor:ES_representable}). The definition is rather involved; we begin by introducing a condition `1-alignment' on curves. We show that any family of curves which is not 1-aligned does not admit any enriched structures. We then define the pullback when the target is 1-aligned, which suffices by the previous result. 

\begin{definition}\label{def:1_aligned}
Let $C/S$ be a curve and $\bar{s} \ra S$ a geometric point. We say $C/S$ is \emph{1-aligned at $\bar{s}$} if for every circuit $\gamma$ in the graph over $\bar{s}$, all the labels of edges in $\gamma$ are equal. We say $C/S$ is \emph{1-aligned} if it is so at every geometric point of $S$. 
\end{definition}

\begin{lemma}\label{lem:ES_implies_1aligned}
Let $C/S$ be a controlled curve, and suppose the set of enriched structures on $C/S$ is non-empty. Then $C/S$ is 1-aligned. 
\end{lemma}
\begin{proof}
%We may assume $S$ is strictly henselian local, write $R = \ca{O}_S(S)$. 
Let $\bar s\to S$ be a geometric point of $S$ with $\Gamma$ the graph of $C_{\bar s}/\bar s$. Write $R = \widehat{\ca O_{S,\bar s}}$. Suppose the set of enriched structures is non-empty. Given a vertex $v \in \Gamma$ and $G$ a connected component of $\Gamma \gdiff v$, we will show that all the edges from $v$ to $G$ have the same label.
This is vacuously true if there are no such edges; otherwise, write $p=p_0,p_1,\dots,p_n$ for the singular points of $C_{\bar s}$ corresponding to these edges, and $(a)=(a_0), (a_1), \dots, (a_n)$ for their labels. We will show that in fact every $(a_i)$ is equal to $(a)$. We identify the completed local ring $\widehat{\ca O_{C, p_i}}$ at the singular point $p_i$ corresponding to the edge labelled $(a_i)$ with $R[[x,y]]/(xy-a_i)$. 

%Was a note below saying: {If $u$ is the image of $\bar s$, then we should check or assume that $\schart_u$ makes sense for $u$ a geometric point of $\schart$.} This is OK, because $\schart$ is by definition quasi=split, so just take the image of the geometric point. I don't think this needs to be spelt out, but maybe I am wrong? 
Write $s\colon S \ra \schart_S \sub \schart$ for an immediate neighbourhood. Using that $\ca{M}$ has normal crossings singularities it is easy to write down a presentation of the ideal sheaf $\ca I_{G^c}$ restricted to the completed local ring in the tautological curve $C_{\schart_S}$ at the image of $p$. Pulling back this presentation to $C_R/R$ we find that the restriction of $s^*\ca I_{G^c}$ to $R[[x,y]]/(xy-a)$ is isomorphic to 
\begin{equation*}
\frac{R[[x,y]]}{(xy-a)}\frac{\Span{A, A_1, \cdots, A_n, X}}{J} 
\end{equation*}
where the submodule of relations $J$ is generated by the $2 \times 2$ minors of the matrix
\begin{equation}\label{relation_matrix}
\begin{bmatrix}A&A_1&\cdots & A_n & X\\
a&a_1&\cdots & a_n & x
\end{bmatrix}
\end{equation}
together with the element $yX-A$. 
%Deleted for new def: Moreover, if $G'$ is any other connected component of $\Gamma \gdiff v$ we find that the restriction of $s^*I_{G'^c}$ to $R[[x, y]]/(xy-a)$ is trivial. 
Thus our assumption that an enriched structure exists implies (after restriction) that there exists a surjective map of $R[[x,y]]/(xy-a)$-modules
\begin{equation*}
\frac{R[[x,y]]}{(xy-a)}\frac{\Span{A, A_1, \cdots, A_n, X}}{J}  \twoheadrightarrow R[[x,y]]/(xy-a). 
\end{equation*}
From this we will deduce that $a_i \in (a)$ for every $1 \le i \le n$, and by symmetry this proves the lemma. First, identifying $A$, the $A_i$ and $X$ with their images in $\frac{R[[x,y]]}{(xy-a)}$ we obtain elements $A$, $A_i$ and $X$ in $\frac{R[[x,y]]}{(xy-a)}$ satisfying the following conditions:
\begin{enumerate}
\item
at least one of $A$, $A_i$, $X$ is a unit (this uses that the target is local);
\item the minors of the matrix in (\oref{relation_matrix}) vanish; 
\item $yX = A$. 
\end{enumerate}
Now (3) implies that $A$ is not a unit, so at least one of $A_1, \cdots, A_n$, $X$ is a unit. 

\noindent \textbf{Case 1:} $A_i$ a unit for some $1 \le i \le n$. Then $x \in (a_i)$ in $\frac{R[[x,y]]}{(xy-a)}$. Thus there exists $g \in R[[x,y]]$ such that $x-ga_i \in (xy-a)$ in $R[[x,y]]$. So let $f \in R[[x,y]]$ be such that $x-ga_i = f\cdot (xy-a)$. Then $x - ga_i = fxy - fa$ in $R[[x,y]]$, so equating coefficients of $x$ (and writing $f_x$ for the coefficient of $x$ in $f$, similarly for $g$) we find that $1-g_x a_i = 0 - f_x a$. But then $1 = (a,a_i)$, which contradicts that fact that $a$ and $a_i$ must be contained in the maximal ideal of $R$. 

\noindent \textbf{Case 2:} $X$ a unit. Then for every $1 \le i \le n$ we have that $a_i \in (x)$ in $\frac{R[[x,y]]}{(xy-a)}$, in other words $a_i$ lies in the kernel of $R \ra \frac{R[[x,y]]}{(xy-a, x)}$. But this kernel is exactly $(a)$, so we deduce $a_i \in (a)$ as required. 
 \todo{Owen: this holds in the quotient of $R[[x,y]]$ by $(xy-a)$, but why does it hold in $R[[x,y]]$ too? David: a-priori it doesn't. Have slightly altered the proof (old version below (commented out) for comparison). } 
\end{proof}

% Old version\noindent \textbf{Case 2:} $X$ a unit. Then for every $1 \le i \le n$ we have that $a_i \in (x)$ in $R[[x,y]]$. \todo{Owen: this holds in the quotient of $R[[x,y]]$ by $(xy-a)$, but why does it hold in $R[[x,y]]$ too?} But $(x) \cap R$ is just the kernel of $R \ra \frac{R[[x,y]]}{(xy-a, x)}$ which is exactly $(a)$, so we deduce $a_i \in (a)$ as required. 

We will now define the pullback of enriched structures for a map of controlled curves. Let $C_S/S$ be a controlled curve, and $f \colon T \to S$ a morphism such that $C_T = C_S \times_S T$ is also controlled.  By definition there exists a factorisation $S \stackrel{s}\ra \schart \ra \frak{M}$ with $\schart$ a controlled neighbourhood; clearly this $\schart$ is also a controlled neighbourhood of $T$. Without loss of generality (see \ref{rem:changing_the_neighbourhood}) we can replace $\schart$ with $\schart_S$, the immediate neighbourhood of $S$ in $\schart$, and let $T \stackrel{t}\ra \schart_T\subset \schart_S$ then be the immediate neighbourhood of $T$ in $\schart_S$.

Write $\ca{E}(C_S/S)$ for the set of enriched structures on $S$ with respect to $\schart_S$, and $\ca{E}(C_T/T)$ the same for $T$. We will construct a map $\ca{E}(C_S/S) \ra \ca{E}(C_T/T)$ (we leave the verification that the construction of the map does not depend on the choice of $\schart$ to the reader). Note that by \ref{lem:ES_implies_1aligned} it is enough to treat the case where $C_S/S$ is 1-aligned---otherwise $\ca{E}(C_S/S)$ is empty and there is a unique map to $\ca{E}(C_T/T)$! 

Before constructing the map we need some preliminary lemmas. Write $\Gamma_S$ for the graph of $S$ and $\Gamma_T$ for the graph of $T$, so we get a specialisation map $\on{sp}\colon \Gamma_S \ra \Gamma_T$ which contracts exactly those edges whose labels become units on $T$ (equivalently, on $\schart_T$). Let $w$ be a vertex of $\Gamma_T$, and write $W = \on{sp}^{-1}(w)$ for the set of vertices of $\Gamma_S$ which map to $w$. So between any two vertices in $W$ we can find a path all of whose edges have labels which become units on $T$. 

\begin{lemma}
The specialisation map induces a bijection $\pi_0(\Gamma_S \gdiff W ) \ra \pi_0(\Gamma_T \gdiff w)$. 
\end{lemma}
\begin{proof}
Given $G \in \pi_0(\Gamma_S \gdiff W)$ and an element $g \in G$, we see that $\on{sp}(g)$ is not equal to $w$ and is thus contained in some $H \in \pi_0(\Gamma_T \gdiff w)$. This $H$ does not depend on the choice of $g \in G$, and so we define the map to send $G$ to $H$. We leave the verification of the bijectivity to the reader. 
\end{proof}
\begin{lemma}
Assume that the prestable curve $C/S$ is 1-aligned. Then for every $G \in \pi_0(\Gamma_S \gdiff W)$ there exists a unique $v \in W$ such that $G \in \pi_0(\Gamma_S \gdiff v)$. 
\end{lemma}
\begin{proof}
First we show that there is exists $v \in W$ such that every edge between $G$ and $W$ has one endpoint equal to $v$. Indeed, suppose that this is not the case, and let $e:v-g$ and $e':v'-g'$ be two edges with $v,v'\in W$ distinct and $g,g'\in G$. Then we may combine $e$ and $e'$ with a path from $v$ to $v'$ consisting of edges that are contracted by $\on{sp}$, and with a path from $g$ to $g'$ within $G$, making a circuit. The $1$-alignment then implies that every edge in this path has the same label, a label that becomes a unit on $T$, so the specialisation map contracts them all, a contradiction. 

It is clear that $G$ is then a connected component of $\Gamma_S\gdiff v$. Furthermore, since $\Gamma_S$ is connected there is at least one edge from $G$ to $v$, and thus $G$ is not a connected component of $\Gamma_S\gdiff v'$ for any other $v'\in W$.
\end{proof}

\begin{lemma}
Assume that $C/S$ is 1-aligned. Given $H\in \pi_0(\Gamma_T \gdiff w)$ there exists a unique vertex $v \in W$ and $G \in \pi_0(\Gamma_S \gdiff v)$ such that $\on{sp}(G) = H$. 
\end{lemma}
\begin{proof}
Immediate from the previous two lemmas. 
\end{proof}
Thus each relative component of $\Gamma_T$ is the image of exactly one relative component of $\Gamma_S$.

\begin{lemma}%{Owen thinks we can remove the `1-aligned' assumption here. David is fine with that. }Assume that $C_S/S$ is 1-aligned. 
Let $G\in\pi_0(\Gamma_S\gdiff v)$ for some $v \in W$. Then exactly one of the following occurs:
\begin{enumerate}
\item
$G \cap W$ is non-empty, i.e.\ $w\in\on{sp}(G)$;
\item $\on{sp}(G) \in \pi_0(\Gamma_T \gdiff w)$. 
\end{enumerate}
\end{lemma}
\begin{proof}
It is clear that (1) and (2) cannot both happen; we will show that if (1) does not hold then (2) does.
Now suppose that $G$ does not meet $W$; then $\on{sp}(G) \sub \Gamma_T \gdiff w$ and is connected. Let $y$ be a vertex of $\Gamma_T$ with a path to $\on{sp}(G)$ that does not pass through $w$; we will show that $y\in\on{sp}(G)$ and thus that $\on{sp}(G)$ is a connected component of $\Gamma_T\gdiff w$.
By adding back in contracted edges where necessary, we may lift this path to a path in $\Gamma_S\gdiff W\sub \Gamma_S\gdiff v$ between a vertex $x$ and $G$. But since $G$ is a connected component of $\Gamma_S\gdiff v$, this path must belong entirely to $G$. Hence $x\in G$ and $y = \on{sp}(x) \in \on{sp}(G)$ as desired.
\end{proof}

Thus for $1$-aligned curves, we have a bijection
%\begin{equation*}
%\Psi\colon\bigsqcup_{v \in W} \{ G \in \pi_0(\Gamma_S \gdiff v) : G \cap W = \emptyset\} \ra \pi_0(\Gamma_T \gdiff w). 
%\end{equation*}
\begin{equation}\label{eg:rel_comp_bijection}
\begin{aligned}
\Psi&\colon\{\text{Relative components $(v,G)$ of }\Gamma_S : \on{sp}(v)\notin\on{sp}(G)\} \\
&\qquad\quad\ \to \{\text{Relative components $(w,H)$ of }\Gamma_T\}\\
& (v,G) \mapsto \bigl(\on{sp}(v),\on{sp}(G)\bigr).
\end{aligned}
\end{equation}

Recall that we write $f\colon T \ra S$. 
\begin{lemma}\label{lem:bad_component_trivial_sheaf}
Given $v \in W$ and $G\in\pi_0(\Gamma_S \gdiff v)$ with $G \cap W \neq \emptyset$, the coherent sheaf $f^*\ca{I}_{G^c}$ on $C_T$ is canonically isomorphic to $\ca{O}_{C_T}$, and thus its only invertible quotient is itself.
\end{lemma}
\begin{proof}
The closed subscheme of $C_{\schart_S}$ corresponding to $\ca{I}_{G^c}$ lies over the closed subscheme of $\schart_S$ cut out by labels on edges on edges between $G$ and $G^c$ (which are all contracted by $\on{sp}$, i.e.\ become units on $\schart_T$), so $\schart_T \ra \schart_S$ does not meet this closed subscheme. 
Therefore the pullback of $\ca{I}_{G^c}$ to $\schart_T$ is the unit ideal sheaf $\ca{O}_{C_{\schart_T}}$, the pullback to $C_T$ of which is $\ca{O}_{C_T}$.
\end{proof}

We are now in a position to define the pullback of enriched structures $\ca{E}(S)\to\ca{E}(T)$ when $t\colon T\to S$ is a morphism of weakly controlled curves and $S$ has a controlled neighbourhood $s\colon S\to \schart$.
Suppose we are given an enriched structure in $\ca{E}(S)$; this consists of a line bundle $\cl{L}_{(v,G)}$ for every relative component of $\Gamma_S$, together with a surjective map
\begin{equation*}
s^* \ca{I}_{G^c} \twoheadrightarrow \cl{L}_{(v,G)},
\end{equation*}
and such that $\bigotimes_{v,G} \cl{L}_{(v,G)} \cong \ca{O}_{C_S}$. We need to construct the similar data of an invertible quotient $f^*s^*\ca{I}_{H^c}\twoheadrightarrow\cl{L}'_{(w,H)}$ for each relative component $(w,H)$ of $\Gamma_T$. 

By \ref{lem:ES_implies_1aligned}, we know that $C_S/S$ is $1$-aligned, and therefore each such pair $(w,H)$ corresponds to a pair $(v,G)=\Psi^{-1}(w,H)$ for some $v\in\on{sp}^{-1}(w)$ and component $G\in\pi_0(\Gamma_S\gdiff v)$ such that $G\cap \on{sp}^{-1}(w) = \emptyset$, and we define $\cl{L}'_{(w,H)}$ by
\begin{equation*}
\cl{L}'_{(w,H)} \coloneqq f^*\cl{L}_{(v,G)}.
\end{equation*}
Each of these is naturally an invertible quotient of $t^*\ca{I}_{H^c} = f^*s^*\ca{I}_{G^c}$.
%{Owen thinks we should maybe put this identity somewhere. David: confused; it is here... Do you mean it should be a separate lemma or something? Owen: nevermind, it makes sense to me now why this is obvious.} 
We want to define $(t^*\ca I_{H^c} \twoheadrightarrow \cl L'_{(w,H)})_{(w,H)}$ to be the pullback, but we need to check that 
\begin{equation*}
\bigotimes_{w,H} \cl{L}'_{(w,H)} \cong \ca{O}_{C_T}
\end{equation*}
locally on $T$. First the tensor product of all of these $\cl{L}'_{(w,H)}$ is
\begin{align*}
 \bigotimes_{w,H} \cl{L}'_{(w,H)} &= \bigotimes_{w,H} f^*\cl{L}_{\Psi^{-1}(w,H)}\\
 &= \bigotimes_{\mathclap{\stackrel{v,G}{\on{sp}(v)\notin\on{sp}(G)}}} f^*\cl{L}_{(v,G)}.
\end{align*}
Now by \ref{lem:bad_component_trivial_sheaf}, the invertible quotient $f^*\cl{L}_{(v,G)}$ of $f^*\ca{I}_{G^c}$ is trivial if $\on{sp}(v)\in\on{sp(G)}$, so we may include them in the tensor product at no cost, so that locally on $S$ we have:
\begin{align*}
\bigotimes_{w,H} \cl{L}'_{(w,H)} &\cong \bigotimes_{v,G} f^*\cl{L}_{(v,G)} \cong f^*\ca{O}_{C_S} = \ca{O}_{C_T}.
\end{align*}

\subsection{Enriched structures for general $C/S$}\label{sec:general_ES}
%Done (photo takes 18 May). Actually very easy using (what will be) some existing lemmas. 

We want to apply \cref{lem:comparison} to define enriched structures on arbitrary families of curves. In order to do this we need to check two things:
\begin{enumerate}
\item The full subcategory of controlled curves is a base for the \'etale topology on schemes over $\frak{M}$;
\item The functor of enriched structures is a sheaf on the full subcategory of curves over $S$ consisting of controlled curves. 
\end{enumerate}
Condition (1) is immediate from \ref{lem:locally-controlled}, so our task is to check (2). More precisely, we need to show that if $S \to \frak M$ is a controlled curve, $\{T_i\}_I \ra S$ is an \'etale cover by controlled curves, and for each $i, j \in I$ we have a cover $\{T_{i,j,k}\}_{K_{i,j}}$ of $T_i \times_S T_j$ by controlled curves, then the following diagram is an equaliser:
\begin{equation}\label{eq:basic_equaliser}
\ca{E}(S) \ra \prod_i \ca{E}(T_i) \rightrightarrows \prod_{i,j,k} \ca{E}(T_{i,j,k})
\end{equation}
(here we write $\ca E(S)$ in place of $\ca E(C/S)$, since we are working with schemes over $\frak M$, so each comes with a tautological curve). 
First consider the case where $C/S$ is \emph{not} 1-aligned. Then $\ca{E}(S)$ is empty by \ref{lem:ES_implies_1aligned}. Moreover, there exists $i_0 \in I$ such that $T_{i_0}$ hits a controlling point of $S$, and so (using that $T_{i_0} \ra S$ is \'etale) we deduce that $T_{i_0}$ is not 1-aligned, so $\ca{E}(T_{i_0})$ is empty. Similarly some $T_{i,j,k}$ hits a controlling point of $S$ and thus admits no enriched structures. This we see all the three terms of (\oref{eq:basic_equaliser}) are empty, so the diagram is certainly an equaliser! It thus remains to treat the case where $S$ is 1-aligned.

The key tool for this is a slight generalisation of the definition of enriched structure given in \ref{def:controlled_ES} (which will also be very useful in \ref{sec:representability}). 
\begin{definition}\label{def:ES_Gamma}
Let $S$ be a controlled curve with graph $\Gamma_S$, and let $s\colon S \to \schart$ be an immediate neighbourhood. Let $\Gamma_S \to \Gamma'$ be any contraction of $\Gamma$. Define $S' \hra S$ to be the locus in $S$ where the labels of the contracted edges are units, and similarly define $\schart' \sub \schart$ (using that $\Gamma_S$ is also the controlling graph of $\schart$). Then $S' \to \schart$ factors via $\schart'$; write $s'\colon S' \to \schart'$. 

Suppose we are given a vertex $v$ of $\Gamma'$ and a connected component $G' \in \pi_0(\Gamma' \gdiff v)$. Write $G$ for the set of vertices of $\Gamma_S$ which map to $G'$. Slightly abusing notation, define the ideal sheaf $\ca{I}_{G'^c}$ on $C_{\schart'}/\schart'$ to be the restriction to $C_{\schart'}$ of the ideal sheaf $\ca{I}_{G^c}$ on $C_{\schart}/\schart$. 

%Let $S$ be a controlled curve with graph $\Gamma$. Let $s\colon S \to U$ be an immediate neighbourhood, and for each 
%vertex $v$ of $\Gamma$ and connected component $G \in \pi_0(\Gamma \gdiff v)$ define the ideal sheaf $\ca{I}_{G^c}$ on $C_U/U$ as usual.

 Given any morphism $f\colon T \ra S'$ we define a \emph{$\Gamma_S/\Gamma'$-enriched structure} on $T$ to be the data of, for each vertex $v$ of $\Gamma'$ and connected component $G' \in \pi_0(\Gamma' \gdiff v)$, an invertible quotient
\begin{equation*}
f^*(s')^*\ca{I}_{G'^c} \twoheadrightarrow \cl{L}_{(v,G')}
\end{equation*}
of sheaves on $C_T$, such that the tensor product of all the line bundles $\cl{L}_{(v,G')}$ is $T$-locally trivial on $C_T$, and such that each $\cl L_{(v,G')}$ has total degree zero on every fibre of $C_T/T$. The set of $\Gamma_S/\Gamma'$-enriched structures on $T$ is denoted $\ca{E}_{\Gamma_S/\Gamma'}(T)$. Pulling back $\Gamma_S/\Gamma'$-enriched structures is defined in the evident way (just pull back the invertible quotients --- the combinatorial data is unchanged), yielding a presheaf %{It seems to me we only use apply this when $S$ is itself a simple chart; this might allow us to simplify the notation in this section? TODO}
\begin{equation*}
\ca{E}_{\Gamma_S/\Gamma'}\colon \cat{Sch}_{S'}^{op} \to \cat {Set}. 
\end{equation*}
\end{definition}

We are more interested in computing $\ca{E}(T)$ than $\ca{E}_{\Gamma_S/\Gamma'}(T)$, but whereas the latter has the elementary description above, the former will require  \ref{lem:comparison} in order to extend the definition to general $T$. In the special case that $\Gamma'=\Gamma_S$ (so $S'=S$ as well), then the two definitions agree at $S$: $\ca{E}_{\Gamma_S/\Gamma_S}(S) = \ca{E}(S)$. The following lemma shows that the same holds for any weakly controlled $T$ over $S$:

%In this section we are mostly interested in the case where $\Gamma = \Gamma'$ (in which case it is clear that $\ca{E}_\Gamma(S) = \ca{E}(S)$), but the more general case will be important in the next section. In general, we show that if $S$ is $1$-aligned then this equality also holds for any weakly controlled $T$ mapping to $S$:
\begin{lemma}\label{lem:E_is_E_Gamma_if_aligned}
In the setting of \ref{def:ES_Gamma}, with $\Gamma' = \Gamma_S$, suppose $S$ is 1-aligned and $T$ is weakly controlled. Then $\ca{E}_{\Gamma_S/\Gamma_S}(T) = \ca{E}(T)$. 
\end{lemma}
\begin{proof}
Write $\Gamma_T$ for the graph of $T$, so we get a specialisation map $\on{sp}\colon \Gamma_S \ra \Gamma_T$. Since $C_S/S$ is 1-aligned, this map $\on{sp}$ has the property that, if $H$ is a circuit-connected component of $\Gamma_S$, then $\on{sp}$ contracts at least one edge in $H$ if and only if it contracts every edge in $H$. 
%{Owen asks if a $2$-vertex-connected subgraph is a set of edges, vertices, or both. David says: edges, but think of it as including `necessary' vertices so that we can talk about circuits etc. Owen: Okay.}

Moreover, if $v$ is a vertex of $\Gamma_S$ and $G \in \pi_0(\Gamma_S \gdiff v)$ then either
\begin{enumerate}
\item Every edge from $v$ to $G$ is contracted by $\on{sp}$, in which case $f^*s^*\ca{I}_{G^c}$ is canonically trivial on $C_T$, or
\item No edge from $v$ to $G$ is contracted by $\on{sp}$, in which case $H\coloneqq\on{sp}(G)$ is a connected component of $\Gamma_T\gdiff \on{sp}(v)$, and $f^*s^*\ca{I}_{G^c}$ is canonically isomorphic to $t^*\ca{I}_{H^c}$, where $t\colon T \to \schart_T$ is the immediate neighbourhood of $T$ in $\schart$ (noting $\schart_T \sub \schart$). 
Furthermore, every relative component of $\Gamma_T$ arises in this way.
\end{enumerate}
Thus the data of a $\Gamma_S/\Gamma_S$-enriched structure on $T$ is equivalent to the data of a $\Gamma_T/\Gamma_T$-enriched structure on $T$, but this is equivalent to the data of an ordinary enriched structure because $\Gamma_T$ is the graph for $T$. The degree condition is satisfied for elements of $\ca E(T)$ by \ref{lem:structure_of_inv_quotients} (this lemma is not proven until \ref{sec:ES_over_fields}, but this is only for expository reasons --- it is entirely independent of the present section). 
\end{proof}

To prove that  diagram (\oref{eq:basic_equaliser}) is an equaliser, we have already seen that we can reduce to the case where $S$ is 1-aligned. In that situation, it is equivalent by \ref{lem:E_is_E_Gamma_if_aligned} to show that the diagram 
\begin{equation}\label{eq:E_Gamma_equaliser}
\ca{E}_{\Gamma_S/\Gamma_S}(S) \ra \prod_i \ca{E}_{\Gamma_S/\Gamma_S}(T_i) \rightrightarrows \prod_{i,j,k} \ca{E}_{\Gamma_S/\Gamma_S}(T_{i,j,k}). 
\end{equation}
is an equaliser, where as before we write $\Gamma_S$ for the graph of $S$. We are done by the following lemma:
\begin{lemma}\label{lem:representability_of_E_Gamma}
In the setting of \ref{def:ES_Gamma}, the functor $\ca{E}_{\Gamma_S/\Gamma'}$ from schemes over $S'$ to sets is representable; in particular it is a sheaf for every subcanonical topology. 
\end{lemma}
All we need right now is the sheaf property of $\ca E_{\Gamma_S/\Gamma_S}$, which is slightly quicker to prove, but we will need the general representability of $\ca E_{\Gamma_S/\Gamma'}$ very soon anyway. 
\begin{proof}
This is immediate by combining four standard ingredients:
\begin{enumerate}
\item The representability of Quot schemes;
\item The inclusion of invertible sheaves into finitely presented quasi-coherent sheaves is relatively representable;
\item The unit section of the relative Picard scheme is relatively representable. 
\item The subfunctor of the relative Picard space consisting of line bundles of total degree zero is representable. 
\end{enumerate}
Combining (1) and (2) we see that the functor of invertible quotients of each $\ca{I}_{G'^c}$ is representable. We take the fibre product over $S'$ of these functors for each relative component $(v,G')$ of $\Gamma'$, and then (3) and (4) show that the additional conditions that the tensor product of all the $\cl{L}_{(v,G')}$ be trivial and that they all have degree 0 do not break representability. 
\end{proof}

Therefore $\ca{E}$, as we have defined it, is an \'etale sheaf on the full subcategory of controlled curves, and since every curve is \'etale-locally controlled, the comparison \ref{lem:comparison} tells us that $\ca{E}$ extends uniquely (up to unique isomorphism) to an \'etale sheaf on all of $\cat{Sch}_\frak{M}$.
\begin{definition}\label{def:ES_in_general}
In the following, we will use $\ca{E}$ to refer to this unique extension, and for any curve $S\to\frak{M}$ we will call elements of $\ca{E}(S)$ \emph{enriched structures on the tautological curve $C/S$}.
\end{definition}
\section{Representability of the functor of enriched structures}\label{sec:representability}
\newcommand{\smallM}{\schart}
The aim of this section is to prove that $\ca{E}$ is relatively representable by an algebraic space over $\frak{M}$. We will prove representability of $\ca{E}$ by reducing to \ref{lem:representability_of_E_Gamma}, the representability of the various $\ca{E}_{\Gamma_S/\Gamma'}$. Note that $\ca{E}$ is by definition a sheaf  for the \'etale topology, as it was defined by extending a sheaf on the subcategory of controlled curves, so it is enough to show representability locally. Let $\schart \to \frak M$ be a simple chart% a smooth morphism from a scheme such that the tautological curve $\upsC /M$ is controlled; 
; we will show representability of the restricted functor $\ca E_\schart\colon \cat{Sch}_\schart^{op} \to \cat{Set}$. Write $\Gamma_{\schart}$ for the graph over $\schart$. Then $\ca{E}_{\Gamma_{\schart}/\Gamma_{\schart}}$ is certainly representable, but does \emph{not} in general coincide with $\ca{E}_\schart$ because $\schart$ is not in general $1$-aligned. This will be remedied by glueing together various suitable $\ca{E}_{\Gamma_\schart/\Gamma'}$, but before getting into the details we give an example where $\ca{E}_\schart \neq \ca{E}_{\Gamma_{\schart}/\Gamma_{\schart}}$. 
\subsection{Example}
%{This example would have to change, as $\smallM$ has been fixed. But I think it is fine, we're just describing the base change of the stack of $ES$ to this universal deformation. }
Let $k$ be an algebraically closed field, and take three copies of $\bb{P}^1$ over $k$. Glue the point $(0:1)$ on each curve to the point $(1:0)$ on the next so as to form a triangle. Let $\schart \hra k[a,b,c]$ be a controlled chart for this curve, with boundary $abc$, and graph $\Gamma_\smallM$ a triangle with labels $a$, $b$, $c$. 
%, and mark the points $(1:1)$ on each curve. In this way we obtain a stable curve of genus 1 with 3 marked points. Let $C_\schart/{\schart}$ be a as a stable marked curve, but then forget about the markings, as they will no longer be relevant (this is only formally smooth over $\frak M$, but the same considerations apply). 
%It is clear that $C_\schart/\smallM$ is controlled with graph $\Gamma_\smallM$ a triangle. Let $(a)$, $(b)$ and $(c)$ be the labels on the edges, so $\smallM = \on{Spec}k[[a,b,c]]$.
 We consider some examples:
\begin{enumerate}
\item Let $p$ be a point of $\smallM$ where at least two of $a$, $b$ and $c$ are units (so $C_p$ is irreducible). Then $\ca{E}_\smallM(p)$ is the set containing the empty enriched structure, and $\ca{E}_{\Gamma_\smallM/\Gamma_\smallM}(p)$ is a singleton consisting of the trivial enriched structure (since all the pullbacks of ideal sheaves are trivial). In particular, the two sets are naturally in bijection! 
\item Let $p$ be a point where exactly one of $a$, $b$ and $c$ is a unit; say $a$ is a unit for simplicity of notation. Then $\Gamma_p$ is obtained from $\Gamma_\schart$ by contracting the edge labelled $a$. We find that $\ca{E}_{\Gamma_\smallM/\Gamma_\smallM}(p)$ is empty but $\ca{E}_\smallM(p)$ is non-empty, in fact it is isomorphic to $k^*$;
\item If $p$ is a point where none of $a$, $b$ or $c$ is a unit (so $p$ is the closed point) then $\ca{E}_{\smallM/\smallM}(p) = \ca{E}_{\Gamma_\smallM}(p)$ directly from the definitions. 
\end{enumerate}
Summarising, in cases (1) and (3) the functor $\ca{E}_{\Gamma_\smallM/\Gamma_\smallM}$ gives the `correct' answer (i.e.\ the same as $\ca{E}_\smallM$), but in case (2) it does not; instead it yields the empty set. We will fix this by glueing in $\ca{E}_{\Gamma_\smallM/\Gamma'}$ for various contractions $\Gamma_{\smallM}\to \Gamma'$. 

\subsection{Glueing} 

Recall that we are assuming $C_\smallM/\smallM$ to be controlled, with graph $\Gamma_{\smallM}$. Given a contraction of graphs $\Gamma_{\smallM} \ra \Gamma'$, let $U_{\Gamma'} \sub \smallM$ be the open subscheme obtained by deleting the loci where the labels of contracted edges vanish. Then we apply \ref{def:ES_Gamma} with $S = \smallM$ to define a functor $\ca E_{\Gamma_\smallM/\Gamma'} \colon \cat{Sch}_{U_{\Gamma'}}^{op} \to \cat{Set}$. 

%This $U_\Gamma$ need not be controlled, but we can nonetheless imitate \ref{def:ES_Gamma} and define a functor $\ca E_\Gamma\colon \cat{Sch}_{U_\Gamma}^{op} \to \cat{Set}$. 

We have schemes $\ca{E}_{\Gamma_\smallM/\Gamma'}$ as $\Gamma'$ runs over various contractions of $\Gamma_\smallM$, from which we will build a representing object for $\ca{E}_\smallM$. We will carefully specify certain loci along which to glue the various $\ca{E}_{\Gamma_\smallM/\Gamma'}$ in order to obtain a representing object for $\ca{E}_\smallM$. First we need a few preliminaries. 

\begin{definition}
Let $f\colon \Gamma \ra \Gamma'$ be a contraction of graphs. We say $f\colon\Gamma \ra \Gamma'$ is \emph{aligned} if for every circuit-connected component $G$ of $\Gamma$, either
\begin{enumerate}
\item
every vertex of $G$ maps to the same vertex of $\Gamma'$, or
\item the map $f$ is injective on the vertices of $G$ (equivalently, no edge in $G$ is contracted). 
\end{enumerate}
\end{definition}
Note that case (1) is \emph{not} the same as saying that all edges in $G$ are contracted --- rather, it says there exists a connected spanning subgraph of $G$ all of whose edges are contracted. 

\begin{lemma}\label{lem_glueing:empty}
Let $\Gamma_\smallM \ra \Gamma'$ be a contraction which is not aligned. Let $p \in U_{\Gamma'}$. Then $\ca{E}_{\Gamma_\smallM/\Gamma'}(p) = \emptyset$. 
\end{lemma}
\begin{proof}
Since $\Gamma_\smallM\to\Gamma'$ is not aligned, there exists some circuit-connected component $G$ of $\Gamma_\smallM$ such that at least one non-loop edge of $G$ is contracted, but the vertices of $G$ do not all have the same image in $\Gamma'$.
Let $e$ be such a contracted edge, let $v$ be the vertex of $\Gamma'$ to which it is contracted, and let $V$ be the set of vertices of $\Gamma_\smallM$ whose image in $\Gamma'$ is the vertex $v$.

Now not all the vertices of $G$ are in $V$ by assumption; let $e'$ be an edge of $G$ connecting such a vertex to $V$, and let $\gamma$ be a circuit of $G$ containing both $e$ and $e'$.
Thus $\gamma$ has at least one edge $e$ between two vertices in $V$ and another edge $e'$ between $V$ and $V^c$. 

\[\begin{tikzpicture}
 \draw[rounded corners=30pt, fill=lightgray!20] (-3,3) -- (-2,-1) -- (1,-1) -- (3,1) -- (2,5) -- (-1,5) -- cycle;
 \node at (-2.25,4.25) {$V$};
 \draw[{Circle}-{Circle}] (-0.05,1.95) -- (1.05,3.05);
 \node at (.25, 2.75) {$e$};
 \draw[{Circle}-{Circle}] (1.95,.05) -- (3.05,-1.05);
 \node at (2.75, -.25) {$e'$};
 \node at (1.7,0) {$w$};
 \node at (3,-1.25) {$u$};
 \draw[rounded corners = 10pt, dashed] (0,2) -- (0,1) -- (2,1) -- (2,0);
 \draw[-{Circle}, dashed] (1,3) -- (2.55,3);
 \draw[rounded corners = 10pt, dashed] (2.5,3) -- (4,3) -- (4,-1) -- (3,-1);
 \node at (4.1,3.1) {$\gamma$};
 \node at (2.7,3.3) {$w'$};
\end{tikzpicture}\]

We will prove the lemma by a careful analysis of the partial degrees which would appear in an enriched structure (more precisely, an element of $\ca{E}_{\Gamma_\smallM/\Gamma'}(p)$) if one were to exist, and then deriving a numerical contradiction. First we have three claims. 

\begin{inclaim} %the number of edges between $V$ and $V^c$ is strictly less than
\begin{equation}\label{eq:edges1}
\# \text{edges } V \text{ to } V^c > \sum_{w \in V}\sum_{\stackrel{H \in \pi_0(\Gamma_\smallM \gdiff w)}{H \cap V = \emptyset}} \# \text{edges }w \text{ to } H 
\end{equation}
where the condition $H \cap V = \emptyset$ should be read as saying that $H$ and $V$ have no vertex in common. 
\end{inclaim}
The point of the claim is to get a strict inequality; a non-strict inequality here is essentially obvious. 

\noindent\textbf{Proof of claim:}
There is an obvious injective map from edges on the RHS of (\oref{eq:edges1}) to edges on the LHS: every edge between such $w$ and $H$ is in particular an edge between $V$ and $V^c$. To show the strictness of the inequality we need to construct an edge from $V$ to $V^c$ which does not appear on the RHS. The key will be the circuit $\gamma$ defined above. 

Recall that $e'$ is an edge of $\gamma$ which has exactly one endpoint in $V$. Write $w$ for the end of $e'$ which is in $V$, and $u$ for the other end of $e'$. If $e'$ appears in the RHS of (\oref{eq:edges1}) then it must go from $w$ to the connected component $H$ of $\Gamma_\smallM \gdiff w$ that contains $u$. But this $H$ has a vertex $w'$ in common with $V$, namely the next point where $\gamma$ re-enters $V$ (note that $w \neq w'$ since $\gamma$ also contains $e$, an edge whose endpoints are both in $V$). Therefore $e'$ is not on the RHS of (\oref{eq:edges1}), which means that the inequality is strict.

This concludes the proof of the first claim. Now let $(p^*\ca I_{H^c} \twoheadrightarrow \cl{L}_H)_H$ be an element of $\ca{E}_{\Gamma_\smallM/\Gamma'}(p)$. We first compute the sums of partial degrees of $\cl{L}_H$ for $H$ corresponding to vertices in $V$. 
\begin{inclaim}
\begin{equation*}
\sum_{w \in V} \sum_{H \in \pi_0(\Gamma_\smallM \gdiff w)} \on{deg}\cl{L}_H|_v = \sum_{w \in V}\sum_{\stackrel{H \in \pi_0(\Gamma_\smallM \gdiff w)}{H \cap V = \emptyset}} \# \text{edges }w \text{ to } H. 
\end{equation*}
where the condition $H \cap V = \emptyset$ should again be read as saying that $H$ and $V$ have no vertex in common. 
\end{inclaim}
\noindent\textbf{Proof of claim:}
Fix $w \in V$ and $H \in \pi_0(\Gamma_\smallM \gdiff w)$. 
Note that if $H$ and $V$ have a vertex in common, then that vertex can be taken to be an endpoint of an edge to $w$ that contracts to $v$.
For if $w'\in H\cap V$, we can find a path from $w'$ to $w$ using only edges contracted to $v$, and thus $H$ contains all the vertices of this path except for $w$ itself.
Now if an edge from $H$ to $w$ is contracted in $\Gamma'$ then $\cl{L}_H$ is trivial, so has degree 0. 

On the other hand, if $H$ and $V$ are disjoint then by \ref{lem:structure_of_inv_quotients} (this lemma is not proven until the next section, but this is only for expository reasons - it is entirely independent of the present section) we have that 
\begin{equation*}
\on{deg}\cl{L}_H|_v = \# \text{ edges from }w \text{ to } H. 
\end{equation*}
This concludes the proof of the claim. We now compute the sums of partial degree of $\cl{L}_H$ for $H$ corresponding to vertices in $V$. 

\begin{inclaim}
\begin{equation*}
\sum_{w \notin V} \sum_{H \in \pi_0(\Gamma_\smallM \gdiff w)} \on{deg}\cl{L}_H|_v = -\# \text{edges } V \text{ to } V^c. 
\end{equation*}
\end{inclaim}
\noindent\textbf{Proof of claim:}
Let $w \notin V$. Then for each $H \in \pi_0(\Gamma_\smallM\gdiff w)$, we have
\begin{equation*}
\on{deg}\cl{L}_H|_v = -\#\text{ edges from }w \text{ to } V \cap H. 
\end{equation*}
Therefore, since every edge from $w$ to $V$ ends in exactly one connected component of $\Gamma_\smallM\gdiff w$, we have
\begin{equation*}
\sum_{H\in \pi_0(\Gamma_\smallM\gdiff w)}\on{deg}\cl{L}_H|_v = -\#\text{ edges from }w \text{ to } V. 
\end{equation*}
Summing over all $w\in V^c$ concludes the proof of the claim. 

We can now finish off the proof of the lemma by observing that (from the condition that the tensor product of all the line bundles appearing in an enriched structure must be trivial) we have
\begin{equation*}
\sum_{w \in \on{Vert}\Gamma_\smallM} \sum_{H \in \pi_0(\Gamma_\smallM \gdiff w)} \on{deg} \cl{L}_H |_v = 0
\end{equation*}
which is incompatible with the above claims, thus yielding a contradiction to the assumption that $\ca{E}_{\Gamma_\smallM/\Gamma'}(p)$ is non-empty. 
\end{proof}

%Let $f\colon \Gamma \ra \Gamma'$ be a contraction between two contractions of $\Gamma_\smallM$ such that $f$ is aligned. Let $T \ra U_\Gamma'$ be centrally controlled with graph $\Gamma_T$, and let $\Gamma \ra \tilde{\Gamma} \ra \Gamma_T$ be any factorisation of the canonical map $\Gamma \ra \Gamma_T$. Then there is a canonical isomorphism {improve!?}

\begin{lemma}\label{lem:aligned_gives_right_answer}
Let $f\colon \Gamma \to \Gamma'$ be a contraction between two contractions of $\Gamma_\smallM$ such that $f$ is aligned. Let $T \to U_{\Gamma'}$ be any morphism. Then there is a canonical identification \begin{equation*}
\ca{E}_{\Gamma_\smallM/\Gamma}(T) = \ca{E}_{\Gamma_\smallM/\Gamma'}(T). 
\end{equation*}
\end{lemma}
\begin{proof}
Let $v$ be a vertex of $\Gamma'$ and write $V$ for the set of vertices of $\Gamma$ mapping to $v$.  Since $f$ is aligned it induces a bijection 
\begin{equation*}
\phi\colon \bigsqcup_{w \in V}\{H \in \pi_0(\Gamma \gdiff w): H \cap V = \emptyset \} \to \pi_0(\Gamma'\gdiff v). 
\end{equation*}
Then for each $G \in \pi_0(\Gamma' \gdiff v)$ we note that $\ca I_{G^c}|_T = \ca I_{{\phi^{-1}(G)}^c}|_T$ and define $\cl{L}_G = \cl{L}_{\phi^{-1}(G)}$ with the obvious map from $\ca I_{G^c}|_T$. To check that we have an enriched structure with respect to $\Gamma'$ it suffices to verify that the tensor product of all these bundles is trivial. But this is immediate from the same property of the original enriched structure with respect to $\Gamma$, using that $\ca I_{H^c}|_T$ is canonically trivial whenever an edge between $H$ and $H^c$ is contracted by $f$. 
\end{proof}

\begin{definition}\label{def:glueing_locus}
% \todo{Owen: There we define, given $S$, the largest open of $U$ on which the graph of $S$ makes sense. Perhaps we should go further and define the largest open on which any given contraction of the graph of $S$ makes sense? D: Could do. As a stop-gap I have re-worked the sentence below to hopefully be correct now. If it does not come often this may be easier. }
Let $\Gamma_\smallM \to \Gamma_1$ and $\Gamma_\smallM \to \Gamma_2$ be two contractions. To shorten notation, we write $U_i \coloneqq U_{\Gamma_i} \sub \smallM$ (the open subscheme of $\smallM$ obtained by inverting the labels of those edges contracted in passing from $\Gamma_\smallM$ to $\Gamma_i$), and $\ca E_i \coloneqq \ca E_{\Gamma_\smallM/\Gamma_i}$. 
% largest open subscheme on which $\Gamma_i$ makes sense --- more precisely, let $U_i$ be empty if the graph $\Gamma_i$ does not arise for any point in $\smallM$, and otherwise pick such a point in $\smallM$ and apply \ref{sec:largest_open_where_graph_makes_sense}, observing that the resulting $U_i$ is independent of the choice of point. 
Define the open subscheme $U_{1,2} \sub U_1 \times_\smallM U_2$ by
\begin{equation*}
U_{1,2} = \left\{ p \in U_1 \times_\smallM U_2 : \exists \text{ a diagram} 
\begin{tikzcd}[row sep=-1mm, column sep=small]
\Gamma_1 \arrow{dr} & & \\
& \tilde{\Gamma} \arrow{r} & \Gamma_p\\
\Gamma_2 \arrow{ur} & &\\
\end{tikzcd} 
\text{ such that } \begin{tikzcd}[row sep=-1mm, column sep=small]
\Gamma_1 \arrow{dr} &\\
& \tilde{\Gamma}  \\
\Gamma_2 \arrow{ur} & \\
\end{tikzcd} \text{are both aligned}\right\}. 
\end{equation*}
\end{definition}
It is clear that $U_{1,2}$ is constructible in $\smallM$, and it is closed under generalisation essentially by construction, hence $U_{1,2}$ is open in $\smallM$. \todo{Owen: I don't know how to tell why $U_{1,2}$ is constructible, so I'm going to have to trust you on this one.David: Basically it is made by including/excluding a bunch of intersections of boundary divisors. }
\begin{lemma}
We carry over the notation from \ref{def:glueing_locus}. Suppose $T \to U_1\cap U_2$ does not factor via $U_{1,2}$. Then at least one of $\ca{E}_1(T)$ and $\ca{E}_2(T)$ is empty. 
\end{lemma}
\begin{proof}
Since enriched structures admit pullbacks it is enough to treat the case where $T$ is a point. The result is then clear from \ref{lem_glueing:empty}. 
\end{proof}
%\Dcomment{In the next lemma, is $\ca E_{\frak U}$ denoting the base-change to $\ca E$ to $\frak U$? If not, what is it? If so, say so! }
\begin{lemma}\label{lem:glueing_isos}
We carry over the notation from \ref{def:glueing_locus}. Let $T \ra U_{1,2}$ be any morphism. Then the maps in \ref{lem:aligned_gives_right_answer} induce canonical isomorphisms
\begin{equation*}
\ca{E}_1(T) = \ca{E}_\smallM(T) = \ca{E}_2(T). 
\end{equation*}
\end{lemma}
\begin{proof}
Immediate from \ref{lem:aligned_gives_right_answer}. 
\end{proof}

\begin{definition}
We define a scheme $\ca{E}_{glue}/\smallM$ by glueing all the $\ca{E}_{i}$ as $\Gamma_i$ runs over contractions of $\Gamma_\smallM$, where we glue $\ca{E}_{i}$ to $\ca{E}_{j}$ over $U_{i,j}$ along the isomorphisms given in \ref{lem:glueing_isos}. 
\end{definition}

\begin{theorem}
The isomorphisms given in \ref{lem:glueing_isos} induce an isomorphism of functors from $\ca{E}_\smallM$ to $\ca{E}_{glue}$. 
\end{theorem}
\begin{proof}
It is enough to verify that the functors coincide on controlled curves. Let $T \ra \smallM$ be such a morphism, with graph $\Gamma_T$. We have a contraction $f\colon \Gamma_\smallM \to \Gamma_T$. If $\Gamma_\smallM \ra \Gamma$ is any other contraction, then $T \ra \smallM$ factors via $U_\Gamma$ if and only if $f$ factors via $\Gamma_\smallM \ra \Gamma$ (and if not then $\ca{E}_{\Gamma_\smallM/\Gamma}(T)$ is clearly empty). It thus suffices to consider intermediate contractions $\Gamma_\smallM \ra \Gamma \ra \Gamma_T$. Moreover, if $\Gamma \ra \Gamma_T $ is not aligned then $\ca{E}_{\Gamma_\smallM/\Gamma}(T)$ is empty, so again it can be ignored. 

Given two aligned intermediate contractions $\Gamma_\smallM \to \Gamma_1 \to \Gamma_2 \to \Gamma_T$, it is immediate that $T\ra\smallM$ factors via $U_{1,2}$, and so $\ca{E}_{1}(T) = \ca{E}_\smallM(T) = \ca{E}_{2}(T)$ are identified by the glueing as required. 
\end{proof}
\begin{corollary}\label{cor:ES_representable}
The functor $\ca{E}$ of enriched structures is representable. 
\end{corollary}

\section{Enriched structures over separably closed fields}\label{sec:ES_over_fields}
Main\`o defined enriched structures on curves over fields in a way which looks rather different from what we write. We will show that our definition is in fact equivalent in this case. We begin by recalling her definition, which we call a \emph{Main\`o-enriched} structure to distinguish the terminology from what we use. %Actually Main\`o's definition differs by a sign from what we write here, but they are clearly equivalent by taking duals, and we find this version more natural from our perspective where ideal sheaves play a larger role. 
\begin{definition}\label{def:MES1}
Let $k$ be a separably closed field, and $C/k$ be a curve. Write $\Gamma$ for the graph of $C/k$. For $v$ a vertex of $\Gamma$, write $C_v^c$ for the union of the irreducible components of $C$ corresponding to the other vertices. A \emph{Main\`o-enriched structure} on ${C}$ is an collection $(\cl{F}_v)_{v \in \on{vert}(\Gamma)}$ of line bundles on $C$ such that 
\begin{equation}\label{eq:MaES_1}
\bigotimes_{v \in \on{vert}(\Gamma)} \cl{F}_v \cong \cl{O}_C
\end{equation}
and for every $v \in \on{vert}(\Gamma)$, we have
\begin{equation}\label{eq:MaES_2}
\cl{F}_v|_{C_v} \cong 
\ca{O}_{C_v}(-C_v \cap C_v^c)
\end{equation}
and
\begin{equation}\label{eq:MaES_3}
\cl{F}_v|_{C^c_v} \cong  \ca{O}_{C_v^c}(C_v \cap C_v^c). 
\end{equation}
\end{definition}
Note that the symbol $\cong$ means `is isomorphic to' --- we do not specify the isomorphisms. 

Before showing equivalence to our rather more involved definition, we give a slightly modified (but equivalent) version of the above definition, since it will make the comparison rather more direct:
\begin{definition}\label{def:M_ES}
Let $k$ be a separably closed field, and $C/k$ be a curve. Write $\Gamma$ for the graph of $C/k$. An \emph{M-enriched structure} on $C/k$ consists of, for each relative component $(v,G)$ of $\Gamma$ a line bundle $\cl{F}_G$ on $C$ such that 
\begin{equation}\label{eq:MES_1}
\bigotimes_{v, G\in\pi_0(\Gamma\gdiff v)} \cl{F}_G \cong \cl{O}_C
\end{equation}
and for every $(v,G)$ we have
%\begin{equation}\label{eq:MES_2}
%\cl{F}_G|_{G^c} \cong 
%\ca{O}_{G^c}(-G \cap G^c)
%\end{equation}
%and
\begin{equation}\label{eq:MES_3}
\cl{F}_G|_{C_G} \cong  \ca{O}_{C_G}(-C_G \cap C_{G^c})
\end{equation}
(where $C_G$ denotes the union of the components of $C$ corresponding to vertices in $G$) and
\begin{equation}\label{eq:MES_4}
\cl{F}_G|_{C_{G^c}} \cong  \ca{O}_{C_{G^c}}(C_G \cap C_{G^c}).
\end{equation} 
\end{definition}

Given an M-enriched structure $(\cl{F}_G)_G$ it is easy to see that we can obtain a Main\`o-enriched structure by defining $\cl{F}_v = \otimes_{G \in \pi_0(\Gamma \gdiff v)} \cl{F}_G^\vee$ (use that conditions (\oref{eq:MaES_1}) and (\oref{eq:MaES_3}) together imply (\oref{eq:MaES_2})).

\begin{lemma}
The above construction induces a bijection between Main\`o-enriched structures and M-enriched structures. 
\end{lemma}
\begin{proof}
If $V$ is a set of vertices of $\Gamma$ then we are writing $C_V$ for the union of the corresponding components of $C$. If $(v,G)$ is a relative component then $C_{G^c}$ is a closed subscheme of $C$, and we can also define $C \setminus C_G$ to be the open subscheme of $C$ obtained by deleting $C_G$. Then $C_{G^c}$ is the closure of $C\setminus C_G$. 

With the notation out of the way, we can build an M-enriched structure from a Main\`o-enriched structure $(\cl{F}_v)$ as follows. Given a vertex $v$ and component $G \in \pi_0(\Gamma \setminus v)$, write
\begin{equation*}
X \coloneqq C \setminus \cup_{\stackrel{G' \in \pi_0(\Gamma \gdiff v)}{G' \neq G}} C_{G'}. 
\end{equation*}
Let 
\begin{equation*}
g\colon \cl F_v|_{C_v}\iso \ca O_{C_v}(-C_v \cap C_{v^c})
\end{equation*}
be an isomorphism. Then on $X \setminus C_G$, the isomorphism $g$ restricts to a trivialisation of $\cl F_v$. Then we define a line bundle $\cl F_G$ by taking the restriction of the line bundle $\cl{F}^\vee_v$ to $X$, and the trivial bundle on $C \setminus C_G$, and glueing them on $X \setminus C_G$ along the trivialisation $g$. 

We leave it to the reader to check that this is well-defined and an inverse to the construction above. \todo{Owen is going to wait to check this until he already understands the easy direction. David thinks this is reasonable! He has also adjusted the above proof, as there was a small mistake in the construction... However, it does now \emph{seem} to be correct. }
\end{proof}

\subsection{The structure of pullbacks of certain ideal sheaves to fibres}

For the remainder of this section, take a controlled curve $C/S$ where $S= \on{Spec}k$, with immediate neighourhood $s\colon S \to \schart$. We will exhibit an isomorphism between the enriched structures on $C/S$ and the $M$-enriched structures on $C/S$. In particular, this shows that (when the test object is a separably closed point) the data of the surjections from the ideal sheaves is redundant, and the line bundles themselves actually determine the enriched structure. We do not know if this is the case in general. Write $\Gamma$ for the graph of $C/S$. %, and $V$ for its set of vertices. 

\begin{lemma}\label{lem:local_structure_of_pullbacks}
Let $(v,G)$ be a relative component of $\Gamma$. Write $B_G$ for the union of the sections corresponding to separating edges of $(v,G)$, which we think of as a set of smooth points on $C_G$ or $C_{G^c}$, or as a divisor (on either) where each point has multiplicity 1.

%Let $v$ be a vertex of $\Gamma$, and $C_v$ the corresponding irreducible component, etc. Let $G$ be a connected component of $\Gamma \gdiff v$. Write $C_G$ for the (connected) curve which is the union of the vertices in $G$. Write $B = C_v\cap C_v^c$, which we think of as a set of smooth points on $C_v$ or $C_G$, or as a divisor where each point has multiplicity 1. Let $B_G$ denote the subset of $B$ consisting of points corresponding to edges between $v$ and points in $G$. 
%
%So $B = \cup_G B_G$ and $C_v^c = \cup_G C_G$ as $G$ runs over connected components of $\Gamma \gdiff v$. 

Let $T_G$ be the torsion module on $C_G$ given by 
\begin{equation*}
T_G = \bigoplus_{p \in B_G}k_p^{\oplus (B_G - p)}
\end{equation*}
where $k_p$ denotes the skyscraper sheaf $k$ at $p$. Then there exist isomorphisms
\begin{equation}
f_{G}\colon s^*\ca I_{G^c}|_{C_G} \stackrel{\sim}{\ra} \ca{O}_{C_G}(-B_G) \oplus T_G,
\end{equation}
\begin{equation}
f_{G^c}\colon s^*\ca I_{G^c}|_{C_{G^c}} \stackrel{\sim}{\ra} \bigoplus_{p \in B_G}\ca{O}_{C_{G^c}}(p)
\end{equation}
and for each $b \in B_G$ an isomorphism
\begin{equation*}
g_b\colon b^*\left( \ca{O}_{C_G}(-B_G) \oplus T_G \right ) \ra b^*\left( \bigoplus_{p \in B_G}\ca{O}_{C_{G^c}}(p) \right)
\end{equation*} 
such that for each $b \in B_G$ the following diagram commutes:

\begin{tikzcd}
b^*\left( \ca{O}_{C_G}(-B_G) \oplus T_G \right ) \arrow{rr}{g_b} \arrow{d}{b^*f_{G}} & & b^*\left( \bigoplus_{p \in B_G}\ca{O}_{C_{G^c}}(p) \right)\arrow{d}{b^*f_{G^c}}\\
b^*s^*\ca I_{G^c}|_{C_G} \arrow[equals]{r} & (s\circ b)^*\ca I_{G^c} \arrow[equals]{r}& b^*s^*\ca I_{G^c}|_{C_{G^c}}. \\
\end{tikzcd}

Moreover, if we view the isomorphism $g_b$ as being a morphism
\begin{equation*}
b^*\ca{O}_{C_G}(-b) \oplus \bigoplus_{B_G - b} k \stackrel{\sim}{\ra} b^*\ca{O}_{C_{G^c}}(b) \oplus \bigoplus_{B_G - b} k,
\end{equation*}
then $g_b$ is the identity on the $\bigoplus_{B_G-b}k$ parts, and induces an isomorphism $b^*\ca{O}_{C_G}(-b)  \stackrel{\sim}{\ra} b^*\ca{O}_{C_{G^c}}(b)$ (in particular $g_b$ respects the evident direct sum decomposition). 
\end{lemma}
\begin{proof}
We may assume the immediate neighbourhood $\schart$ is affine; write $R \coloneqq \ca{O}_\schart(\schart)$. In the notation of \ref{def:ideal_sheaves}, let $\ca J = (e_p:p \in B_G)$ be the ideal of $R$ generated by the labels of singular points in $B_G$. Write $\pi_u \colon C_\schart \ra \schart$ and $\pi_s\colon C\to S$ for the structure maps, $\iota_{G^c}\colon C_{G^c} \ra C$ and $\iota_G\colon C_G \ra C$ for the inclusions (closed immersions). 

First we analyse $s^*\ca I_{G^c}|_{C_{G^c}}$ outside points in $B_G$. On the open subscheme $C_\schart \setminus C_G$ we find that $\pi_u^*\ca J = \ca I_{G^c}$, so we find that outside $B_G$ we have $s^*\ca I_{G^c} |_{C_{G^c}} = \iota_{G^c}^*\pi_u^*\ca J$. Hence away from $B_B$ we find that $s^*\ca I_{G^c} |_{C_{G^c}}$ is just the pullback along $\pi_s$ of $s^*\ca J$, and the latter is a free $k$-module generated by symbols $e_p$ as $p$ runs over $B_G$. \todo{Owen: I don't understand this paragraph. Isn't $\pi^*_u\ca{J} = \ca{I}/_{G^c}$ a statement about the curve over $\schart$? Then what does it mean to say it holds on the open subscheme $C\setminus C_G$? How do you pull back $\pi_u^*\ca J$ (on $C_\schart$) along $\iota_{G^c}\colon C_{G^c}\to C$? David: Corrected $C \setminus G_C$ to $C_\schart \setminus C_G$. We don't yet have notation for the inclusion of the closed subscheme $C$ into $C_\schart$, perhaps this is causing the confusion about pulling back $\pi_u^*\ca J$ --- we first pul it back to $C$ along this nameless inclusion. Do you think it would make it clearer also to name the inclusion? The notation is quite heavy already, and restricting a sheaf to a closed subscheme maybe doesn't need another letter?}

Next we analyse $s^*\ca I_{G^c}|_{C_G}$ outside points in $B_G$. On the open subscheme $C \setminus C_{G^c}$ we find that $\ca I_{G^c}$ is the trivial bundle, so it follows that $s^*\ca I_{G^c}|_{C_G}$ is canonically trivial outside points in $B_G$. 

It remains to analyse in detail what happens locally at a point in $B_G$. By fpcq descent we may work with completed local rings at such a point, taking care that the morphisms we construct glue suitably. We fix now some $b \in B_G$, and for the remainder of this proof the letter $p$ will be an index running over $B_G - b$. After choosing an isomorphism we may write 
\begin{equation*}
A = \frac{\widehat{R_s}[[x,y]]}{(xy - e_b)}
\end{equation*}
for the completed local ring of $C_\schart$ at $b$. If $\frak{m}$ is the maximal ideal of $\widehat{R_s}$ then locally
\begin{equation*}
C_{G^c} = \on{Spec} \frac{A}{\frak{m} + (x)} = k[[y]]\;\; \text{ and }\;\;C_G = \on{Spec} \frac{A}{\frak{m} + (y)} = k[[x]]. 
\end{equation*}
Next we will compute the sheaves $\iota_{G}^*s^*\ca I_{G^c} = s^*\ca I_{G^c}|_{C_G}$, $b^*\iota_{G}^*s^*\ca I_{G^c} = b^* \iota_{G^c}^*s^*\ca I_{G^c}$ and $\iota_{G^c}^*s^*\ca I_{G^c} = s^*\ca I_{G^c}|_{C_{G^c}}$ explicitly as $A$-modules (here we will slightly abuse notation by conflating $A$-modules and quasi coherent sheaves on $\on{Spec}A$). First we have
\begin{equation*}
\begin{split}
\iota_{G}^*s^*\ca I_{G^c} = (\ca J,x) \otimes_A \frac{A}{(\frak{m} , y)}  & \stackrel{\sim}{\ra} \left( \bigoplus_p k\right) \oplus  x k[[x]] \\
e_b= 0 & \mapsto 0\\
x & \mapsto (0, x)\\
e_p & \mapsto (\delta_p, 0)\\
\end{split}
\end{equation*}
where $\delta_p$ indicates the element of $\bigoplus_p k$ taking value $1$ in the $p$-th position and zero elsewhere. Next we find
\begin{equation*}
\begin{split}
\iota_{G^c}^*s^*\ca I_{G^c} = (\ca J,x) \otimes_A \frac{A}{(\frak{m} , x)}  & \stackrel{\sim}{\ra} \left( \bigoplus_p k[[y]]\right) \oplus  k[[y]] \\
e_b & \mapsto  (0,y)\\
x & \mapsto (0, 1)\\
e_p & \mapsto (\delta_p, 0), \\
\end{split}
\end{equation*}
where $\delta_p$ is as above but taking values in $\bigoplus_p k[[y]]$. Finally we compute
\begin{equation*}
\begin{split}
b^*\iota_{G}^*s^*\ca I_{G^c} = b^*\iota_{G^c}^*s^*\ca I_{G^c} = (\ca J,x) \otimes_A \frac{A}{(\frak{m} , x, y)}  & \stackrel{\sim}{\ra} \left( \bigoplus_p k\right) \oplus  k \\
e_b =0 & \mapsto 0\\
x & \mapsto (0, 1)\\
e_b & \mapsto (\delta_p, 0).  \\
\end{split}
\end{equation*}
The module $b^*\iota_G^*s^*\ca I_{G^c} = b^*\iota_{G^c}^*s^*\ca I_{G^c}$ is naturally a quotient of $\iota_{G}^*s^*\ca I_{G^c}$ and of $\iota_{G^c}^*s^*\ca I_{G^c}$, and the maps are given by
\begin{equation*}
\begin{split}
xk[[x] \ra & k \leftarrow k[[y]]\\
x \mapsto & 1 \testleft 1\\
x^2 \mapsto & 0 \testleft y. \\
\end{split}
\end{equation*}
We leave it to the reader to check that with the natural glueing maps the local descriptions we have given here  combine to yield the lemma. 
\end{proof}

\begin{lemma}\label{lem:structure_of_inv_quotients}
We continue in the notation of the previous lemma. Let $\cl L$ be an invertible quotient of $s^*\ca I_{G^c}$. 
\begin{enumerate}
\item There exists an isomorphism $\cl L |_{C_G} \cong \ca O_{C_G}(-B_G)$. 
\item There exists an isomorphism $\cl L |_{C_G{^c}} \cong \ca O_{C_{G^c}}(B_G + D)$ where $D$ is some effective divisor on $C_{G^c}$;
\item If moreover we assume $\cl L$ to have total degree 0, then the divisor $D$ above is zero. 
\end{enumerate}
\end{lemma}
\begin{proof}
Part (3) follows from parts (1) and (2) by a degree calculation. Part (1) is clear since we know $s^*\ca I_{C_G} \cong \ca O_{C_G}(-B_G) \oplus T_G$ and $T_G$ is torsion so is killed by any map to an invertible sheaf. The content is in (2). 

To give a map $\bigoplus_{b \in B_G}\ca O_{C_{G^c}}(b) \to \cl L|_{C_{G^c}}$ is to give a section of $\cl L(-b)|_{C_{G^c}}$ for each $b \in B_G$. The glueing conditions at the singular points imply that each such section lands in $\cl L(-B_G)|_{C_{G^c}}$ and takes a non-zero value at $b$. Hence $\cl L(-B_G)|_{C_{G^c}}$ admits a non-zero global section, and the result is clear (noting that $C_{G^c}$ is connected). 
\end{proof}
\subsection{Comparing the two notions of enriched structure}
We continue in the notation of the previous section, so $C/k$ is a prestable curve over a separably closed field, with graph $\Gamma$. 
\begin{lemma}\label{lem:ES_to_MES}
Let 
\begin{equation*}
\left(\phi_G\colon s^*I_{G^c} \twoheadrightarrow \cl{L}_G\right)_{(v,G)}
\end{equation*}
(where $(v,G)$ runs over relative components of $\Gamma$) be an enriched structure on $C$. Then $(\cl{L}_G)_{G}$ is an M-enriched structure on $C$. 
\end{lemma}
\begin{proof}
From the definition of an enriched structure we have that $\bigotimes_{(v,G)}\cl{L}_G$ is trivial, so formula (\oref{eq:MES_1}) is satisfied. We must verify that formulae (\oref{eq:MES_3}) and (\oref{eq:MES_4}) hold, which is done by \ref{lem:structure_of_inv_quotients}. 
%
%
%. Let $v$ be a vertex of $\Gamma$ and $G \in \pi_0(\Gamma \gdiff v)$. We must check that $\cl{L}_G|_{C_G}\cong \ca{O}_{C_G}(-G \cap G^c)$. 
%
%Recall that $C_G$ is connected. In this case \cref{lem:local_structure_of_pullbacks} shows that the restriction of the sheaf $s^*I_G$ to $C_{G}$ is of the form $\ca{O}_{C_G}(-G \cap G^c) \oplus T$ where $T$ is some torsion module, so the only possible invertible quotient is $\ca{O}_{G}(-G \cap G^c)$ (up to isomorphism). 
\end{proof}

\begin{lemma}\label{lem:extending_surjections}
Let $v$ be a vertex of $\Gamma$, and $G \in \pi_0(\Gamma \gdiff v)$. Let $\cl{L}$ be a line bundle on $C$ such that 	
\begin{equation*}
\cl{L}|_{C_{G^c}} \cong 
\ca{O}_{C_{G^c}}(G \cap G^c)
\end{equation*}
and
\begin{equation*}
\cl{L}|_{C_G} \cong  \ca{O}_{C_G}(-G \cap G^c). 
\end{equation*}
Let $\phi\colon s^*I_{G^c}|_{C_G}\twoheadrightarrow \cl{L}|_{C_G}$ be any surjection. Then there exists a unique surjection $\Phi\colon s^*I_{G^c} \twoheadrightarrow \cl{L}$ such that $\Phi|_{C_G} = \phi$. 
\end{lemma}
\begin{proof}We adopt the notation of \ref{lem:local_structure_of_pullbacks} (so $B_G = G \cap G^c$), and choose a compatible system of isomorphisms as in the lemma. 
By \cite{Ferrand2003Conducteur-desc} and \ref{lem:local_structure_of_pullbacks}, it is equivalent to show that there is a unique surjection 
\begin{equation*}
\bigoplus_{p \in B_G}\ca{O}_{C_{G^c}}(p) \ra \ca{O}_{C_{G^c}}(B_G)
\end{equation*}
such that suitable compatibilities hold at each of the point $p \in B_G$. Fix one $b \in B_G$. To give a map $\ca{O}_{C_{G^c}}(b) \ra \ca{O}_{C_{G^c}}(B_G)$ is the same as to give a global section of $\ca{O}_{C_{G^c}}(B_G-b)$. Since the sheaf $T$ is torsion we know that it is killed by $\phi \circ f_{G}^{-1}$, which implies that the section of $\ca{O}_{C_{G^c}}(B_G-b)$ must vanish at every point in the support of $B_G-p$. Since $C_{G^c}$ is connected, this further implies that the section is actually constant. To pin the section down completely it is enough to determine its value at a single point. Indeed, its value at $b$ is determined by the restriction of the isomorphism $g_b$ to the first factor, and we are done. 
\end{proof}

\begin{lemma}\label{lem:MES_to_ES} 
Let $(\cl{L}_G)_{G}$ be an M-enriched structure on $C$. Then there exists an enriched structure $(\phi_G\colon s^*I_{G^c} \twoheadrightarrow \cl{L}_G)_{G}$. Moreover, if $(\phi'_G\colon s^*I_{G^c} \twoheadrightarrow \cl{L}_G)_{G}$ is another enriched structure with the same quotients, then it is equivalent to $(\phi_G\colon s^*I_{G^c} \twoheadrightarrow \cl{L}_G)_{G}$. 
\end{lemma}
\begin{proof}
To construct the enriched structure, we first choose surjections $s^*I_{G^c} |_{C_G} \twoheadrightarrow \cl{L}_G|_{C_G}$, which is certainly possible since the $s^*I_{G^c} |_{C_G}$ is isomorphic to the direct sum of $\cl{L}_G|_{C_G}$ with a torsion module (\ref{lem:local_structure_of_pullbacks}). \Cref{lem:extending_surjections} then allows us to extend these surjections to an enriched structure. For the uniqueness, notice that any two enriched structure structures $(\phi_G\colon s^*I_{G^c} \twoheadrightarrow \cl{L}_G)_{G}$ differ by multiplications by elements of $k^*$ since (up to torsion) they are just isomorphisms of line bundles on \emph{connected} curves over a separably closed field. The uniqueness part of \ref{lem:extending_surjections} then shows that the two enriched structure structures are equivalent. 
\end{proof}

Combining \cref{lem:ES_to_MES} and \cref{lem:MES_to_ES} we immediately deduce
\begin{theorem} \label{thm:ES_MES}
The map from enriched structures to M-enriched structures which sends $(\phi_G\colon s^*I_G \twoheadrightarrow \cl{L}_G)_{G}$ to $(\cl{L}_G)_{G}$ induces a bijection between equivalence classes of enriched structures and M-enriched structures. 
\end{theorem}

\section{The stack of enriched structures and universal N\'eron models}\label{sec:UNMA_morphisms}

Suppose that $S$ is a scheme, $U \sub S$ a dense open subscheme, and $A \to U$ an abelian scheme. A \emph{N\'eron model} of $A$ over $S$ is a proper smooth group algebraic space $N/S$ together with an isomorphism $N_U \ra A$, satisfying the \emph{N\'eron mapping property}: for every smooth morphism $T \ra S$, and for every $U$-morphism $f\colon T_U \ra N_ A$\todo{Owen: Shouldn't $N_A$ be just $A$?D: canonically isomoprhisc, and I think it is a little clearer like this... }, there exists a unique $S$-morphism $F\colon T \ra N$ such that $F|_{T_U} = f$. In \cite{Holmes2014Neron-models-an} the second author investigated N\'eron models in the case of jacobians of prestable curves, giving necessary and sufficient criteria for their existence. 

%For the remainder of this section we work over a fixed base scheme $\Lambda$, which we assume to be excellent and regular --- for example, we could take $\Lambda = \on{Spec} \bb{C}$ or $\Lambda  = \on{Spec} \bb{Z}$. We assume {Now no need to assume, it's true!}that $\ca{M}$ is smooth over $\Lambda$ and the boundary is a relative normal crossings divisor (cf.\ \cite{Deligne1969The-irreducibil}). 

We write $\mathring{\frak{M}}$ for the locus of smooth curves in $\frak{M}$, and $J / \mathring{\frak M}$ for the jacobian of the universal smooth curve. 

In \cite{Holmes2014A-Neron-model-o} the second author constructed a \emph{universal N\'eron-model-admitting morphism} $\widetilde{\frak M} \to {\frak M}$, a terminal object in the category of \emph{N\'eron model admitting morphisms}. A N\'eron model admitting morphism is a morphism $t\colon T \to {\frak M}$ such that 
\begin{itemize}
\item $T$ is regular;
\item $t^{-1}\mathring{\frak M}$ is dense in $T$;
\item $t^*J$ admits a N\'eron model over $T$
\end{itemize}
(a morphism satisfying the first two conditions we call `non-degenerate'). In particular, $J$ admits a N\'eron model $N / \widetilde{\frak M}$. In other words, the universal property says that any N\'eron model admitting morphism factors uniquely via $\widetilde{\frak M} \to {\frak M}$. 

This $\widetilde{\frak M}$ is not quasi-compact over ${\frak M}$, but it can be written as a union of a `tower' of quasi-compact pieces. The general construction is more involved, but the first step of the tower is easy to describe; we define ${\frak M}^1 \hra \widetilde{\frak M}$ to be the largest open substack such that the pullback of $\upsC$ to ${\frak M}^1$ is regular (equivalently, $\frak M^1 \to \frak M$ is a terminal object in the category of 1-aligned morphisms to $\frak M$). The aim of this section is to construct an isomorphism between ${\frak M}^1$ and $\ca E$.

First we will define a map $\bb{E}\colon \frak{M}^1 \to \ca{E}$. In fact we will do something equivalent but a bit slicker; we will construct an enriched structure on the tautological curve over every controlled non-degenerate 1-aligned morphism $t\colon T \ra \frak{M}$. Since both $\frak M^1$ and $\ca E$ are separated and birational over $\frak M$ (and $\frak M^1$ is reduced), maps between (open subsets of) these stacks are unique if they exist. Hence if we construct maps locally, they will automatically glue to globally defined maps.

\begin{definition}\label{def:inv_quots_if_1aligned}
Let $t\colon T \ra \frak{M}$ be a non-degenerate morphism from a scheme to $\frak{M}$ which is 1-aligned. Assume $T$ is controlled with graph $\Gamma$. Let $v$ be a vertex of $\Gamma$ and $G \in \pi_0(\Gamma \gdiff v)$. Let $T \ra \schart \ra \frak{M}$ be an immediate neighbourhood, and let $\ca I_{G^c}$ be the ideal sheaf on $C_\schart$ as in \ref{def:ideal_sheaves}. The same recipe defines an ideal sheaf $\ca L_{G^c}$ on $C_T$, and we have a canonical surjection $t^*\ca I_{G^c} \twoheadrightarrow \ca L_{G^c}$. Moreover, alignment of $T \to \frak M$ means that all the edges between $v$ and $G$ have the same label (which is non-zero by non-degeneracy), and so $\ca L_{G^c}$ cuts out a locally-principal closed subscheme, which has codimension 1 since $t^*\mathring{\frak M}$ is dense, and hence this closed subscheme is a Cartier divisor on $C_T$. Hence $\ca L_{G^c}$ is actually a line bundle. The map $t^*\ca I_{G^c} \twoheadrightarrow \ca L_{G^c}$ from above is then the quotient map required in the definition of an enriched structure. One easily checks that the conditions are satisfied. 
\end{definition}

\begin{remark}
This enriched structure can alternatively be defined by looking at the closure of the unit section in the relative Picard space of $C/T$, and imposing degree conditions to cut out the right sections. Note that the closure of the unit section is (\'etale locally on $T$) a union of sections by our alignment assumption, see \cite{Holmes2014Neron-models-an}. 
\end{remark}

The remainder of this section will be taken up with proving that $\bb{E}\colon \frak{M}^1 \to \ca{E}$ is an isomorphism, and observing a few consequences. First a quick lemma, then we do quite some work to prove that, if $\Lambda$ is a point, then $\ca{E}_\Lambda$ is regular. From this it is a small step to show that $\bb{E}$ is an isomorphism. 

\begin{lemma}\label{lem:phi_onto}
The map $\bb{E}\colon \frak{M}^1 \to \ca{E}$ is surjective. 
\end{lemma}
\begin{proof}
Let $k$ be a separably closed field, $C_0/k$ a curve and $E$ an enriched structure on $C_0$. We wish to show the existence of a regular curve $C/k[[t]]$ with $C_0$ isomorphic to the special fibre, and such that $E$ is the enriched structure induced by $C$ as in \ref{def:inv_quots_if_1aligned}. Marking some points on $C_0$ we can assume it to be stable, whereupon this is proven in \cite[Corollary 6.6]{Esteves2002Limit-canonical}, or \cite{Maino1998Moduli-space-of} (who assumes characteristic zero). This $C$ is a regular curve over a dedekind scheme and thus 1-aligned, so the tautological map $\on{Spec} k[[t]] \ra \frak{M}$ lifts canonically to $\frak{M}^1$. It is then immediate from the construction of $\bb{E}$ that the image of the closed point of $\on{Spec}k[[t]]$ maps to the point corresponding to $E$.  
\end{proof}

\subsection{Geometric regularity of the stack of enriched structures}\label{sec:geometric_regularity}
Let $k$ be a separably closed field. In this section, we will show that the base-change $\ca E_k$ is regular.

Note that $\frak{M}_k$ is locally of finite type over $k$, and hence so are $\frak{M}_k^1$ and $\ca{E}_k$. We will show that $\ca{E}_k$ is regular by computing the dimension of its tangent space. 

%{Only locally... Delete this para?}Since regularity is \'etale local we may and will assume that $\ca{M}$ is an integral scheme. By \cite{Holmes2014A-Neron-model-o} it follows that $\ca{M}^1$ is integral, hence by \ref{lem:phi_onto} that $\ca{E}$ is irreducible. 

By our finite-type assumption it is enough to check that $\ca{E}_k$ is regular at all $k$ points. We fix $p$ a $k$-point of $\ca{E}_k$. Write $d$ for the dimension of $\frak{M}_k$ at $p$. 

\begin{lemma}\label{lem:lower_bound_dim}
$\on{dim}_p\ca{E}_k \ge d$
\end{lemma}
We will see later that this is an equality. 
\begin{proof}
Since our schemes are of finite type over a field, this follows from the existence of an irreducible open neighbourhood of $p$ in $\ca{E}_k$, and the dominance of $\ca{E}_k \ra \frak{M}_k$ (which follows from the same property for $\frak{M}_k^1 \ra \frak{M}_k$). 
\end{proof}
Write $\pi\colon \ca{E}_k \ra \frak{M}_k$ for the structure map. Write $\pi_*\colon T_p\ca{E}_k \to T_{\pi(p)}\frak{M}_k$ for the map on tangent spaces induced by $\pi$. 

\begin{lemma}\label{lem:dim_ker}
Write $\Gamma$ for the graph of the curve over $p$. Then
\begin{equation*}
\on{dim} \on{Ker}\pi_* \le \# \{\text{non-loop edges of }\Gamma\} + \# \on{Vert}\Gamma - \sum_{v \in \on{Vert} \Gamma} \# \pi_0(\Gamma \gdiff v) - 1. 
\end{equation*}
\end{lemma}
We will see later that this is an equality. 
\begin{proof}
In this proof we write $k[\epsilon] \coloneqq k[\epsilon ]/\epsilon^2$. We think of the tangent space to $\ca{E}$ at $p$ as the space of pointed maps from $T \coloneqq \on{Spec}k[\epsilon]$ to $\ca{E}$ through $p$. The kernel of $\pi_*$ corresponds exactly to those maps which factor via the inclusion of the fibre over $\pi(p)$; in other words, where the curve does not deform, only the enriched structure on it deforms. 

Let $C_0/k$ be the curve corresponding to $p$, and let $C = C_0 \times_k k[\epsilon]$, a constant deformation. Let $E$ be the enriched structure on $C_0$ corresponding to $p$. We need to show that the dimension of the space of enriched structures on $C$ which restrict to $E$ on $C_0$ is at most 
\begin{equation*}
N \coloneqq \# \{\text{non-loop edges of }\Gamma\} + \# \on{Vert}\Gamma - \sum_{v \in \on{Vert} \Gamma} \# \pi_0(\Gamma \gdiff v) - 1. 
\end{equation*}

Let $(v,G)$ be a relative component of $\Gamma$. Write $t\colon T \ra \frak{M}_k$ for the tautological morphism. Write $B_G$ for the set of sections in $C(T)$ corresponding to the edges between $G$ and $G^c$. We find by \ref{lem:local_structure_of_pullbacks} (and using that $t$ factors via a point) that 
\begin{equation*}
t^*\ca I_{G^c} |_{C_G} \cong \ca{O}_{C_G}(-B_G)\oplus \on{torsion}
\end{equation*}
and so has a unique invertible quotient $\ca{O}_{C_G}(-B_G)$ up to isomorphism. By the condition that the tensor product of all the line bundles be trivial we find (\ref{lem:structure_of_inv_quotients}) that if $\cl{L}_G$ is part of an enriched structure on $C$ then $\cl{L}_G|_{C_{G^c}} \cong \ca{O}_{C_{G^c}}(B_G)$. 
Applying \ref{lem:local_structure_of_pullbacks} again we obtain
\begin{equation*}
t^*\ca I_{G^c} |_{C_{G^c}} \cong \bigoplus_{b \in B_G}\ca{O}_{C_{G^c}}(b), 
\end{equation*}
so as in \ref{lem:structure_of_inv_quotients}, surjective maps to $\cl{L}_G$ correspond to elements of 
\begin{equation*}
\bigoplus_{b \in B_G} H^0(C_{G^c}, \ca{O}(B_G - b))
\end{equation*}
and the glueing conditions imply that the section under consideration lands in 
\begin{equation*}
H^0(C_{G^c}, \ca{O}) \sub H^0(C_{G^c}, \ca{O}(B_G - b)). 
\end{equation*}
Since $C_{G^c} \ra T$ is proper, flat and has reduced connected geometric fibres we deduce that $H^0(C_{G^c}, \ca{O}) = k[\epsilon]$. To specify the surjection to $\cl{L}_G$ is to specify the elements of $k[\epsilon]$ at points in $B_G$ which glue these maps together. The surjections are already specified over the closed point $k$ (since we are looking at the tangent space to a fixed enriched structure), so we get a $k$-worth of choices at each element of $B_G$. We do not claim (at this point) that all of these enriched structures actually exist, we only say that any enriched structure must satisfy these conditions (hence the inequality in the statement of the lemma). 

At this point we can simply count the dimension, imitating the proof of \cite[proposition 2.12]{Maino1998Moduli-space-of}. Let $v_1, \ldots, v_r$ be the vertices of $\Gamma$, and define 
\begin{equation*}
n_i = \#\{\text{edges }v - v^c\} - \# \pi_0(\Gamma \gdiff v_i). 
\end{equation*}
The dimension of the space of choices for $\cl{L}_G$ (for $G \in \pi_0(\Gamma \gdiff v_1)$) is the number of edges from $v_1$ to $G$, minus 1 to allow for adjusting the isomorphism on $G$. Summing over $G \in \pi_0(\Gamma \gdiff v_1)$ we find that the dimension of the space of choices for the $\cl{L}_G:G \in \pi_0(\Gamma \gdiff v_1)$ is $n_1$. 

A similar argument for $v_2$ yields that the dimension of the space of choices for the $\cl{L}_G:G \in \pi_0(\Gamma \gdiff v_2)$ is $n_2$, except that $\# \{\text{edges }v_1 - v_2\} -1$ choices were already determined by our choices for $v_1$ and the condition that the tensor product of all the bundles together be trivial.  Thus the dimension of the space of choices for connected components in the complements of $v_1$ \emph{or} $v_2$ is $n_1 + n_2 - \# \{\text{edges }v_1 - v_2\} +1$. Continuing in this fashion yields that the total dimension is $N$ as required. 
\end{proof}

\begin{lemma}\label{lem:dim_coker}
Write $\on{Part}\Gamma$ for the decomposition of $\Gamma$ into circuit-connected components as in \ref{lem:circuit_conn_partiton}. Then 
\begin{equation*}
\on{dim} \on{coker} \pi_* \ge \sum_{G \in \on{Part} \Gamma} \left(\#\on{edges}G -1\right)
\end{equation*}
\end{lemma}
We will see later that this is an equality. 
\begin{proof}
This follows easily from the fact that enriched structures can only exist when the curve is 1-aligned (\ref{lem:ES_implies_1aligned}). If $G \in \on{Part} \Gamma$ then let $T_G$ be the $(\# \on{edges} G)$-dimensional subspace of $T_{\pi(p)}\frak{M}_k$ spanned by deforming those nodes which are edges in $G$. Then the intersection of $\pi_*T_P\ca{E}_k$ with $T_G$ cannot meet any of the coordinate hyperplanes except at the origin, from which we deduce that the image is contained in\footnote{Later we will deduce that this containment is an equality. It should be possible to give an intrinsic characterisation of this 1-dimensional subspace but we have not yet done so. } a 1-dimensional subspace of $T_G$. Applying this condition for each $G$, the lemma follows. 
\end{proof}
\begin{lemma}
\begin{equation*}
\sum_{G \in \on{Part} \Gamma} \left(\#\on{edges}G -1\right) = \# \{\text{non-loop edges of }\Gamma\} + \# \on{Vert}\Gamma - \sum_{v \in \on{Vert} \Gamma} \# \pi_0(\Gamma \gdiff v) - 1. 
\end{equation*}
\end{lemma}
\begin{proof}
We may assume $\Gamma$ has no loops since they make no contribution to either side. Further, elements of $\on{Part} \Gamma$ which are bridges make no contribution to either side, and both sides of the equality are additive in disjoint unions of graphs, so we may assume $\Gamma$ is circuit-connected (so $\#\on{Part} \Gamma = 1$). Then we compute both sides; the LHS is just the number of edges in $\Gamma$ minus 1. The circuit-connectivity implies that each $\pi_0(\Gamma \gdiff v)$ has exactly 1 element, and the result follows. 
\end{proof}
\begin{proposition}\label{prop:regularity_of_E}
$\ca{E}_k$ is regular at $p$. 
\end{proposition}
\begin{proof}
From the definition of the kernel and cokernel we have that 
\begin{equation*}
\on{dim} T_p\ca{E}_k - \on{dim} T_{\pi(p)}\frak{M}_k = \on{dim} \on{ker} \pi_*  - \on{dim} \on{coker} \pi_*. 
\end{equation*}
Combining with \ref{lem:dim_ker,lem:dim_coker} we deduce that 
\begin{equation*}
\on{dim} T_p\ca{E}_k \le \on{dim} T_{\pi(p)}\frak{M}_k  = d. 
\end{equation*}
But by \ref{lem:lower_bound_dim} we know that $\on{dim}_p\ca{E}_k \ge d$, so in fact all these inequalities are equalities and $\ca{E}_k$ is regular at $p$. 
\end{proof}

\subsection{$\bb{E}$ is an isomorphism}

Returning now to working over $\bb Z$, we prove that $\bb{E}$ is an isomorphism. This mostly consists of mild gymnastics with properties of morphisms --- the main work was carried out in \ref{sec:geometric_regularity}. In fact, it would have been enough to show that $\ca{E}_k$ was reduced, as regularity would follow from what we prove in this section, but in any case it was nice to play with computing tangent spaces, since this is something which our definition of enriched structure over arbitrary base schemes has made possible. Anyway, on with the main result:

\begin{theorem}
$\bb{E}\colon \frak{M}^1 \to \ca{E}$ is an isomorphism. 
\end{theorem}
\begin{proof}
Being an isomorphism is local on $\frak M$, so base-changing we may replace $\frak{M}$ by a scheme. Note that $\frak{M}^1$ and $\ca{E}$ are both of finite presentation over $\frak{M}$. Let $\on{Spec}k \to \on{Spec} \bb Z$ be any geometric point (so $k$ is some algebraically closed field). Now $\mathring{\frak{M}}_k$ is dense in $\frak{M}_k^1$, so it is dense in $\ca{E}_k$ by surjectivity of $\bb{E}_k$. Moreover, $\ca{E}_k$ is regular by \ref{sec:geometric_regularity}, in particular it is reduced. The map $\ca{E}_k \ra \frak{M}_k$ is 1-aligned by \ref{lem:ES_implies_1aligned}. By the universal property of $\frak{M}_k^1$ (see \cite[definitions 3.1 and 12.1]{Holmes2014A-Neron-model-o} --- note that the normal-crossings assumption is superfluous) we obtain a map $\bb{F}_k\colon \ca{E}_k \ra \frak{M}_k^1$ (over $\frak{M}_k$):
\begin{equation*}
\frak{M}^1_k \stackrel{\bb{E}_k}{\longrightarrow} \ca{E}_k \stackrel{\bb{F}_k}{\longrightarrow} \frak{M}^1_k. 
\end{equation*}
The composite $\bb{F}_k \circ \bb{E}_k$ is the identity over the schematically dense open $\mathring{\frak{M}}_k$, and $\frak{M}^1 \ra \frak{M}$ is separated, so in fact $\bb{F}_k \circ \bb{E}_k$ is the identity. We also know that $\bb{E}_k$ is surjective, and thus is bijective, in particular quasi-finite. Since $\frak{M}^1_k$ is separated we can apply Zariski's Main Theorem to deduce that $\bb{E}_k$ is the composite of an open immersion and a finite map. The open immersion is clearly an isomorphism, so $\bb{E}_k$ is finite. It is also a bijection, hence is radicial. It is locally of finite presentation, and the fact that $\bb{F}_k \circ \bb{E}_k$ is the identity allows one to easily deduce that $\bb{E}_k$ is formally \'etale, but a radicial \'etale morphism is an isomorphism. 

Readers who are only interested in working over an algebraically closed field can stop here, but it is only a little more work to deduce the general case: we apply a fibrewise criterion, then repeat the above arguments. Firstly, the above result and the fibrewise criterion for flatness implies that $\bb{E}$ is flat. It is also an isomorphism on every geometric fibre, and hence is smooth. This implies (since smoothness is local on the source for the smooth topology, and $\frak{M}^1 \ra \on{Spec} \bb Z$ is smooth) that $\ca{E} \to \on{Spec} \bb Z$ is smooth. In particular it is reduced, and $\mathring{\frak{M}}$ is dense in $\ca{E}$, so $\ca{E} \to \frak{M}$ is 1-aligned, and we obtain a map $\bb{F}\colon \ca{E} \to \frak{M}^1$ as before. Separatedness of $\frak{M}^1$ implies that $\bb{F} \circ \bb{E}$ is the identity (since it is over $\mathring{\frak{M}}$). Applying Zariski's Main Theorem and formal \'etaleness just as above we deduce that $\bb{E}$ is an isomorphism as required. 
\end{proof}

We can now reap the tasty fruits of our labour, as many nice properties of $\ca{E}$ follow from the same for $\frak{M}^1$. 
\begin{corollary}\label{cor:from_Mtilde}\leavevmode
\begin{enumerate}
\item
$\ca{E}$ is smooth over $\on{Spec} \bb Z$;
\item $\ca{E}$ is separated over $\frak{M}$;
\item If $\mathring{\ca{J}} \to \mathring{\frak{M}}$ is the jacobian of the tautological smooth curve, then $\mathring{\ca{J}}$ admits a finite-type N\'eron model $\ca{N}$ over $\ca{E}$ (which is unique up to unique isomorphism);
\item If $t\colon T \to \frak{M}$ is any morphism from a regular scheme such that 
\begin{itemize}
\item
$t^*\frak{C}$ is smooth over a dense open of $T$;
\item 
$t^*\frak{C}$ is 1-aligned; 
\end{itemize}
then $t^*\ca{N}$ is the N\'eron model of $t^*\mathring{\ca{J}}$ (in particular, the N\'eron model exists). 
\item 
$\ca{E} \to \frak{M}$ is in fact uniquely characterised by the above property - more precisely, it is the terminal object in the category of such morphism to $\frak{M}$. 
\item Over the stack of stable curves, the pullback of Caporaso's balanced Picard stack $\ca{P}_{d,g}$ to $\ca E$ acquires a natural group/torsor structure (see \cite[\S 13]{Holmes2014A-Neron-model-o} for a precise statement). 
\end{enumerate}
\end{corollary}
\begin{proof}
See \cite{Holmes2014A-Neron-model-o}, and \cite{Holmes2016Quasi-compactne} for the quasi-compactness of the N\'eron model. 
\end{proof}

\section{Relation to the constructions of Main\`o}\label{sec:Maino_relation}

In this section, we begin by recalling some definitions due to Main\`o. These can be found both in her thesis \cite{Maino1998Moduli-space-of} and in her unpublished preprint \cite{Maino1998preprint}; we will give precise references to the latter, as it seems to have been more widely circulated. We will then relate her constructions to the definitions made in the present work, and use this to prove certain of her conjectures. 

Main\`o worked over an algebraically closed field of characteristic zero, but this restriction is not necessary, so we will work over arbitrary base schemes $\bb Z$ as usual (the interested reader can easily check that her proofs carry over in this setting). Main\`o worked exclusively with stable curves, over the coarse Deligne-Mumford moduli space $\overline M_g$, but find it clearer to work over the fine moduli stack. The restriction to stable curves does seem necessary for some of her results explicitly describing chains of blowups to construct moduli spaces. Because of this, we will work over the stack $\Mbar /\bb Z$ of stable (unmarked) curves. To avoid clashes with the notation for compactified enriched structures in later sections, we will in this section use the notation $\ca M^{stab}$ for the Deligne-Mumford stack of stable curves. The natural map $\Mbar  = \ca{M}^{stab} \to \frak M$ is an open immersion, and we denote by $\ca E^{stab}$ the pullback of our fine moduli stack of enriched structures from $\frak M$ to $\ca{M}^{stab}$.

Main\`o makes the following preliminary definition: 
\begin{definition}(\cite[definition 2.1]{Maino1998preprint})
Let $k$ be a separably closed field and $C/k$ a stable curve with irreducible components $C_1, \dots, C_r$. An \emph{enriched structure} on $C$ is a collection of isomorphism classes of line bundles $\ca F_1, \dots, \ca F_r$ such that there exists a regular smoothing $C'$ of $C$ with $\ca F_i = \ca O_{C'}(C_i)|_{C}$. 
\end{definition}
She then shows (\cite[proposition 2.14]{Maino1998preprint}) that this definition is equivalent to the more intrinsic one given in \ref{def:MES1}. In particular, by \ref{thm:ES_MES} we know that our definition coincides with hers over separably closed fields. 

In particular, Main\`o's above definition gives a way to associate an enriched structure on $C/k$ to any regular smoothing of $C/k$. Writing $\on{Def}(C/k)$ for the $3g-3$-dimensional $k$-vector space of first-order infinitesimal deformations, we have a linear subspace $\on{Def}^{lt}(C/k)\sub\on{Def}(C/k)$ of locally trivial deformations --- these are also known as `equisingular' deformations, as they are exactly the deformations which do not smooth any nodes. If $C'/k[[t]]$ is a regular smoothing of $C/k$, Main\`o shows in \cite[theorem 2.6]{Maino1998preprint} that the resulting enriched structure on $C/k$ depends only on the image of this smoothing in the quotient 
\begin{equation*}
N \coloneqq \frac{\on{Def}(C/k)}{\on{Def}^{lt}(C/k)}. 
\end{equation*}

\subsection{Over reduced base schemes}

Main\`o's next goal is to generalise her definition from curves over fields to curves over reduced base schemes (this is the subject of section 3 of \cite{Maino1998preprint}). The restriction to the reduced case seems unavoidable with her approach, because her definition depends heavily on specifying the restriction of line bundles to geometric fibres, and so cannot `see' non-reduced structure on the base. 

Unfortunately, the definition Main\`o gives (\cite[definition 3.3]{Maino1998preprint}) does not behave as intended. In \ref{sec:Maino_original}, we will recall her original definition, and give an example to show why the resulting functor of enriched structures is not representable. In particular, this disproves her conjectures 5.5 and 5.6, relating moduli of enriched structures to blowups of the moduli space of stable curves (see \ref{sec:maino_blowups} for a detailed description of these conjectures). 

This problem arises because Main\`o indexes the line bundles in her enriched structures by vertices rather than by relative components. In \ref{sec:our_version-of_maino}, we will present a slightly modified version of her definition, and show that it is equivalent to ours. We will prove later (\ref{sec:maino_blowups}, in particular \ref{cor:blowup_conjecture}) that her conjectures 5.5 and 5.6 \emph{do} hold with this slightly-modified definition.

\subsubsection{Our modification of Main\`o's definition}\label{sec:our_version-of_maino}
We choose first to present our modified version of Main\`o's definition. In this way, the reader who is not familiar with Main\`o's work and wishes simply to take the shortest path to the proofs of conjectures 5.5 and 5.6 can do so, skipping \ref{sec:Maino_original}. 

We begin with a family of curves over a reduced strictly hensellian local base $S$ (in particular, note that such a family is controlled)\footnote{In fact, Main\`o writes that she works over a reduced local ring with algebraically closed residue field. However, a family of curves over such a base need not admit any sections, which she needs for some intermediate constructions. We assume that she intended a strictly hensellian local base, so that sections exist through the smooth locus of every component of every fibre. }, and we have an immediate neighbourhood $s\colon S \to \schart_S$. Write $\Gamma$ for the graph over the closed point. Let $(v,G)$ be a relative component of $\Gamma$. We define the open subset $U_{(v,G)}\sub S$ to be the complement of the closed subscheme $Z(\ca J_G)$ in $S$ (see \ref{def:ideal_sheaves}).

With these notions in hand, we define an M-enriched structure on $C/S$ (the `M' stands for `Main\`o'):
\begin{definition}(c.f. \cite[definition 3.3]{Maino1998preprint})
If $C/S$ is not 1-aligned, the set of M-enriched structures on $C/S$ is defined to be empty. Otherwise, an \emph{M-enriched structure} on $\pi\colon C\to S$ is a collection of total-degree-zero line bundles $\ca F_{(G,v)}$ for $(v,G)$ relative components of $\Gamma$, such that 
\begin{enumerate}
\item
For every $s \in S$: given a relative component $(v,G)$ of $\Gamma_s$, let $(v', G')$ be the unique relative component of $\Gamma$ specialising to $(v,G)$ as in \ref{eg:rel_comp_bijection}. Let $\ca F_{(v,G)} = \ca F_{(v', G')}$. Then the $\ca F_{(v,G)}$ form an enriched structure on $C_s/s$. 
\item 
For every relative component $(v,G)$ of $\Gamma$, we have
\begin{equation*}
\ca F_{(v,G)} |_{\pi^{-1}(U_{(v,G)})} \cong \ca O_{\pi^{-1}(U_{(v,G)})}. 
\end{equation*}
\end{enumerate}
\end{definition}
The first condition requires that these sheaves give us an enriched structure on every fibre of $C/S$. 

Suppose we have an enriched structure on $C/S$ (in the sense of \ref{def:controlled_ES}). This immediately gives the data of an $M$-enriched structure, and it follows from \ref{thm:ES_MES} that this collection of line bundles satisfies Main\`o's first, fibrewise condition. To check her second condition, it is enough to observe that the ideal sheaves $s^*\ca I_{G^c}$ are themselves trivial over the relevant $\pi^{-1}U_{(v,G)}$, so the same certainly holds for their invertible quotients. In this way, we see that any enriched structure gives rise to an M-enriched structure. 

Conversely, every M-enriched structure arises in this way, but we must work a little more to prove it. Suppose we are given an M-enriched structure $(\ca F_{(v,G)})_{(v,G)}$, and let $(v,G)$ be relative component of $\Gamma$. We will construct an invertible quotient $\ca L_{(v,G)}$ of the pulled-back ideal sheaf $s^*\ca I_{G}$ which functions as the line bundle of the M-enriched structure. 

There are two quite different cases to consider. The first is where the open subset $U_{(v,G}) \sub S$ is non-empty. In this case the ideal sheaves $\ca I_{(G,v)}$ are in fact already invertible, by our alignment assumption (c.f. (\cite[corollary 3.12]{Maino1998preprint}).

The second case is where $U_{(v,G)}$ is empty; here the reducedness of $S$ and 1-alignedness of $C/S$ implies that the labels of all edges from $v$ to $G$ are zero. The structure of the restriction of $\ca F_{(v,G)}$ to $C_G$ and $C_{G^c}$ is thus as given in \ref{lem:structure_of_inv_quotients}. By \ref{lem:extending_surjections}, any line bundle satisfying these conditions arises as an invertible quotient of $s^*\ca I_G$ in a unique way. The last two lemmas were stated and proven when the base is a field, but the same proofs carry over to the case where the labels vanish; c.f. \ref{part:compactifying}. Finally, the compatibility is clear, since we impose it fibrewise over a reduced base. Putting this together, we obtain
\begin{theorem}
For $S$ any reduced strictly hensellian local scheme, and $C/S$ a stable curve, there is a natural bijection between enriched structures on $C/S$ and M-enriched structures on $C/S$. 
\end{theorem}

Main\`o then defines an enriched structure on an arbitrary reduced scheme $S$ as the limit of a diagram of M-enriched structures on all strictly hensellian local rings mapping to $S$. Since the strictly hensellian local rings are precisely the local rings in the \'etale topology, we see that M-enriched structures coincide with enriched structures over arbitrary reduced base schemes.

\subsubsection{Main\`o's original definition}\label{sec:Maino_original}
Main\`o's original definition over reduced bases does not work as intended. We will describe it briefly here, and give an example to show that the resulting functor is not representable. 

Recall the partition $\on{Part}(\Gamma)$ of the edges of $\Gamma$ into maximal circuit-connected subsets as in \ref{lem:circuit_conn_partiton}. Given $H \in \on{Part}(\Gamma)$, we write $\bar H$ for the unique connected subgraph of $\Gamma$ with those edges --- following Main\`o, we will refer to such a $\bar H$ as a \emph{link} of $\Gamma$, though the reader should be warned that this is non-standard terminology. We say a vertex is \emph{disconnecting} if it is contained in at least two distinct links. 

If $u$ and $v$ are two vertices, Main\`o defines an open subscheme $U_{uv} \sub S$ to be the locus of $s \in S$ such that the specialisations of $u$ and $v$ to the graph $\Gamma_s$ coincide; this is equivalent to requiring that, for at least one path $p$ from $u$ to $v$ in $\Gamma$, all the edges along $p$ have labels which are units in $U_{uv}$. 

Given a vertex $v$ of $\Gamma$, we write $L_v$ for the union of the vertices of all links containing $v$. So if $v$ is non-disconnecting this consists of exactly the vertices of the unique link containing $v$. Main\`o then defines open subsets $U_v \sub S$ by the formulae
\begin{equation*}
\begin{split}
U_v & = \bigcup_{u \in L_v} U_{uv} \;\; \text{if $v$ is non-disconnecting, and}\\
U_v & = \bigcap_{u \in L_v} U_{uv} \;\; \text{if $v$ is disconnecting.}\\
\end{split}
\end{equation*}

\begin{remark}
For a relative component $(v,G)$ of $\Gamma$, we see 
\begin{equation*}
U_{(v,G)} = \bigcup_{u \in G} U_{uv}, 
\end{equation*}
relating our construction above to that of Main\`o. 
\end{remark}

With these notions in hand, Main\`o makes the following definition:
\begin{definition}(\cite[definition 3.3]{Maino1998preprint})\label{def:es_maino_reduced}
A \emph{Main\`o-enriched structure} on $\pi\colon C\to S$ is a collection of line bundles $\ca F_v$ for $v \in \on{Vert}(\Gamma)$ such that 
\begin{enumerate}
\item
For every $s \in S$, write $\on{sp}_s\colon \Gamma\to \Gamma_s$ for the contraction map. For each vertex $v $ of $\Gamma_s$, let $\ca F_v = \bigotimes_{u : \on{sp}_s(u) = v} \ca F_u$. Then the $\ca F_v$ form a Main\`o-enriched structure on $C_s/s$ (as in \ref{def:MES1}). 
\item 
For every vertex $v$ of $\Gamma$, we have
\begin{equation*}
\ca F_v |_{\pi^{-1}(U_v)} \cong \ca O_{\pi^{-1}(U_v)}. 
\end{equation*}
\end{enumerate}
\end{definition}

\begin{example}\label{eg:maino_not_rep}
We now present an example to show that the resulting functor of Main\`o-enriched structures is not representable. We work over a base scheme $S = k[[t]]$ with $k$ a field. We define a labelled graph $\Gamma$ with 4 vertices and 3 edges:
\begin{equation*}
\Gamma : \;\;\;\;\;  v_1 \stackrel{0}{\rule{1cm}{0.55pt}} v_2 \stackrel{t}{\rule{1cm}{0.55pt}} v_3 \stackrel{0}{\rule{1cm}{0.55pt}} v_4. 
\end{equation*}
We choose a stable compact-type curve $C/S$ with labelled graph $\Gamma$.

Since $C/S$ is of compact-type it admits a unique enriched structure $(\ca F_{(v,G)})_{(v,G)}$ indexed by relative components. For each vertex $v_i$, define $\ca F_i$ by tensoring together the (one or two) $\ca F_{(v,G)}$ for $G$ running over the connected components of ${\Gamma\gdiff v}$. This is easily checked to satisfy the first condition of Main\`o's definition, and since the open sets $U_{v_1}, \dots, U_{v_4}$ are all empty, the second condition in Main\`o's definition is vacuous. 
%\{Owen, please can you check that you agree that the $U_{v_i}$ are empty? I found her definition here quite convoluted (and I may have interpreted it incorrectly above - could you try to check from her original def?) Thanks!} \{As far as I can tell, this all looks right to me.} 

Now choose $p$ and $q$ two sections through the smooth locus of $C/S$, with the same reduction modulo $t$; we see these as Cartier divisors on $C$. We define
\begin{equation*}
\begin{split}
\ca F_1' &= \ca F_1\\
\ca F_2' &= \ca F_2(p-q)\\
\ca F_3' &= \ca F_3(q-p)\\
\ca F_4' &= \ca F_4. \\
\end{split}
\end{equation*}
On the special fibre we have $\ca F_i' = \ca F_i$, because of the assumption that $p$ and $q$ coincide modulo $t$. On the other hand, on the generic fibre we have $\ca F_2' \otimes \ca F_3' = \ca F_2 \otimes \ca F_3$, so these again give a Main\'o-enriched structure on the generic fibre. Hence, Main\`o's first condition is satisfied. The second condition is still vacuous, as the $U_{v_i}$ are empty. Thus $\ca F_1', \dots, \ca F_4'$ also define a Main\`o-enriched structure. 
\end{example}

In this way, we have built two non-isomorphic Main\`o-enriched structures, which coincide over the closed point and also over the generic point of $S$. This contradicts representability of the functor of Main\`o-enriched structures, as we have two $k[[t]]$ points which agree at the generic and closed points, but are not equal. 

The problem lies in the fact that Main\`o attaches her line bundles $\ca F_v$ to vertices, not to relative components. Because of this, her second condition cannot be made strong enough to force the line bundles to be of the required shape. 

%===
%Graph with 3 vertices and 2 edges, over $k[[t]]$, labels $0$ and $t$. 
%\begin{equation*}
%v_1 \stackrel{t}{\to} v_2 \stackrel{0}{\to} v_3. 
%\end{equation*}
%This is 1-aligned, so there exists ES. It's compact type, so there is a `natural' ES consisting of $\ca F_1$, $\ca F_2$, $\ca F_3$, which we should write down carefully. Then let $p$ and $q$ be two sections through the smooth locus of the component $v_1$(or any component, actually?), such that $p$ and $q$ coincide modulo $t$. Then consider
%\begin{equation*}
%\begin{split}
%\ca F_1' &= \ca F_1(p-q)\\
%\ca F_2' &= \ca F_2(q-p)\\
%\ca F_3' &= \ca F_3. \\
%\end{split}
%\end{equation*}
%First, notice that on the special fibre, we have $\ca F_i' = \ca F_i$, because of the assumption that $p$ and $q$ coincide over the special fibre. On the other hand, on the generic fibre we have $\ca F_1' \otimes \ca F_2' = \ca F_1 \otimes \ca F_2$, so these again give a enriched structure on the generic fibre. Hence, Main\`o's first condition is satisfied. Then we need to check the second condition. But in fact $U_1 = U_2 = \emptyset$, ???? Really? No, I think $U_2 = \emptyset$, but $U_1 \neq \emptyset$. 

\subsection{Blowing up the moduli space of curves}\label{sec:maino_blowups}

Main\`o's other main contribution is the construction of an explicit chain of blowups of $\ca{M}^{stab}$, and an explicit open subscheme of this blowup, which she conjectures coincides with the moduli space of enriched structures. As we have seen in \ref{sec:Maino_original}, her notion of enriched structures did not define a representable functor, and so this conjecture was not true as stated. In this section, we will describe her blowups and open subscheme, and then show that the resulting space \emph{can} be naturally identified with the moduli space of enriched structures as we define them, thus proving a `corrected' version of her conjectures. 

From now on, we work in $\ca{M}_g^{stab}$, the moduli stack of stable curves of fixed genus $g$. For $1 \le i \le 3g-3$, we define the locally closed subscheme $\ca R_i \sub \ca{M}^{stab}_g$ to be the locus of curves whose graphs have $i$ nodes, are not trees, and are circuit-connected. 

We set $B_{3g-3} = \ca{M}_g^{stab}$. Inductively on $2 \le i \le 3g-3$, starting with $i = 3g-3$, we define $B_{i-1}$  to be the blowup of $B_{i}$ at the closure of the preimage of $R_i$. In other words, we first blow up $\ca R_{3g-3}$, then we blow up the strict transform of the closure of $\ca R_{3g-4}$, then the strict transform of the closure of $\ca R_{3g-5}$, etc. The final step in the tower is $B_{1}$, which we also simply denote by $B$. Main\`o proves (\cite[proposition 6.2]{Maino1998preprint}) that $B$ is regular, so the strict transform of $\ca R_1$ is a Cartier divisor, hence we would obtain the same result if we had run our induction all the way to $i=1$. 

Writing $\tilde{\ca R}_i$ for the closure of the pullback of $\ca R_i$ to $B$, Main\`o defines a closed subset $Z \tra B$ by the formula
\begin{equation*}
Z = \bigcup_{i=1}^{3g-3}\overline{ \left(
\left(  \cup_{1 \le j < i}\tilde{\ca R}_j  \right)\cap \pi^{-1} \ca R_i
\right)}, 
\end{equation*}
and then defines $En$ to be the open substack of $B$ obtained by deleting $Z$. 

Main\`o then constructs (\cite[theorem 4.3]{Maino1998preprint}), for every field-valued point of $\ca{M}_g^{stab}$, a natural bijection between the enriched structures on the corresponding stable curve $C/k$, and the $k$-points of the fibre of $En$. She then conjectures (\cite[conjecture 5.6]{Maino1998preprint}) that this natural bijection can be extended from field-valued points to points taking values in arbitrary reduced local rings (in the statement she omits the `reduced' assumption, but as she does not give a definition of enriched structures over non-reduced rings, we assume this to be an oversight). Her conjecture 5.5 is the same statement, but restricted to discrete valuation rings. 

As we have seen in \ref{sec:Maino_original}, these conjectures are both false as stated, since her definition of enriched structures over reduced local rings does not yield a representable functor. In what follows (in particular in \ref{cor:blowup_conjecture}), we will show that her conjectures \emph{do} hold when one uses our definition of enriched structure, and the requirements that the base be reduced or local are superfluous. More precisely, we will construct an isomorphism of stacks over $\ca M_g^{stab}$ between $En$ and $\ca E^{stab}$. % (really the pullback of the latter from $\ca{M}$ to $\ca{M}_g^{stab}$, but to avoid cluttering the notation we will continue to denote it by $\ca E$).
 Such an isomorphism is necessarily unique if it exists, as both stacks are separated and birational to $\ca{M}_g^{stab}$. 

We know that both stacks $En$ and $\ca E^{stab}$ are regular, and Main\`o has already provided us with a bijection on the field-valued points. Hence it is enough to construct a map between these stacks, which is compatible with Main\`o's bijection on each field-valued fibre. Since $\ca E^{stab}$ has a universal property, it is most reasonable to construct a map from $En$ to $\ca E^{stab}$. Perhaps the most satisfying way to do this would be to write down a universal enriched structure over $En$, but the details of this become somewhat painful. Instead, recall from \ref{sec:UNMA_morphisms} the isomorphism $\ca E^{stab} \to \ca{M}^{1, stab}$, from the stack of enriched structures to the universal 1-aligned stack pulled back to the moduli stack of stable curves. 
\begin{lemma}
$En$ is 1-aligned. 
\end{lemma}
\begin{proof}
Let $\bar x$ be a geometric point of $En$, with local ring $\ca O^{et}_{En, \bar x}$. Write $\Gamma$ for the labeled graph over $x$, whose labels are principal ideals in $\ca O^{et}_{En, \bar x}$. If $\gamma$ is a circuit in $\Gamma$, we need to show that all edges in $\gamma$ have the same label. 

We treat first the case where $\bar x$ maps to some $\ca R_i$, so the graph is circuit-connected. Main\`o checks that the fibre of $En$ over this point is obtained by blowing up the closure of $\ca R_i$, then deleting the `coordinate hyperplanes', i.e. the parts of the exceptional locus where some label becomes a non-trivial power of another label. The universal property of the blowup (principalising ideal sheaves), and the fact that there are no higher-power relations between labels, together imply that $En$ is 1-aligned at $\bar x$. 

The general case is similar, but more involved. Main\`o shows that the fibre is isomorphic to a product of projective spaces, one for each link of the graph, and again with the coordinate hyperplanes removed. In other words, for each link we prinicipalise the ideal sheaf generated by its labels, then delete the loci where some label is a non-trivial power of another. As before, this implies that $En$ is 1-aligned at $\bar x$. 
\end{proof}

This lemma furnishes us with a map $En \to \ca{M}^{1,stab}$, and hence $En \to \ca E^{stab}$. Using test curves it is straightforward to check that this is compatible with Main\`o's bijection on field valued points --- we omit the details, as it would require reproducing large parts of her proofs; we leave the interested reader to check the original source. Putting things together, we have 
\begin{corollary}\label{cor:blowup_conjecture}
There is a unique isomorphism $En \to \ca E^{stab}$ of stacks over $\ca{M}_g^{stab}$. This isomorphism is compatible with Main\`o's bijection on each geometric fibre over $\ca{M}_g^{stab}$. 
\end{corollary}
This proves the `corrected' versions of her conjectures 5.5 and 5.6. 

Since $En$ is clearly a scheme over $\ca{M}_g^{stab}$, we obtain
\begin{corollary}
The stack $\ca E^{stab}$ is relatively representable by a scheme over $\ca{M}_g^{stab}$. 
\end{corollary}

It seems natural to ask whether Main\`o's blowup $B$ coincides with our compactified stack of enriched structures, defined in \ref{part:compactifying}. We do not know if this is the case. It is certainly true that there are many different ways of blowing up $\ca{M}_g^{stab}$ which also contain open subschemes isomorphic to $\ca E^{stab}$ (for a trivial example, just blow up $B$ at a point in $Z$). 

\part{Compactifying the stack of enriched structures}\label{part:compactifying}

\section{Defining compactified enriched structures}
The problem of compactifying the moduli stack of enriched structures has been open since the work of Main\`o \cite{Maino1998Moduli-space-of}. Since enriched structures were only defined over fields prior to this work, this was always going to be a rather awkward problem. Main\`o considered building a compactification using `bubbling', but this did not work as the boundary components were too large; looking back, we realise that Main\`o was seeing higher parts of the universal N\'eron-Model admitting stack $\tilde{\frak M}$ (see \cite{Holmes2014A-Neron-model-o}).  While $\tilde {\frak M}$ could be viewed as a compactification of the moduli stack of enriched structures for some purposes (for example, it satisfies a restricted version of the valuative criterion for properness), it is a very unsatisfactory compactification in other ways --- since it is not proper over $\frak M$ it is not ideal for doing intersection theory etc. We will use the other common approach to compactifying moduli of line bundles: we will use torsion free sheaves of rank 1. 

There are two main parts to our definition of an enriched structure. The first is to consider invertible quotients of pullback of certain carefully chosen ideal sheaves. This part is rather easily adapted to torsion free sheaves; we simply replace the requirement that the quotients be invertible by stipulating that they be torsion free of rank 1 (see \ref{def:TFR1} for a precise definition). The second part of the definition is to require that the tensor product of all these bundles be isomorphic (locally on $S$) to the trivial bundle. This does not work well with torsion free rank 1 sheaves; by definition, if a tensor product of sheaves is trivial then they are all invertible! To get around this, we show that the condition that the tensor product of the sheaves be trivial can be replaced by a condition on the shapes of the kernels of the quotient maps, see \ref{def:compatibility}. We show that for line bundles this recovers our original definition (\ref{sec:CES_for_invertible_is_ES}), that for torsion-free rank 1 sheaves it defines a representable functor, and that the representing object is proper over $\frak M$ (\ref{sec:properness}). 

\begin{definition}\label{def:TFR1}
Let $C/k$ be a curve over a separably closed field, and let $j\colon C^{sm} \hra C$ be the inclusion of the smooth locus. We say a coherent sheaf $\cl F$ on $C$ is \emph{torsion free of rank 1} if
\begin{enumerate}
\item
$j^* \cl F$ is locally free of rank 1 on $C^{sm}$;
\item The canonical map $\cl F \to j_*j^*\cl F$ is injective. 
\end{enumerate}
Now let $C/S$ be a curve over any base, and $\cl F$ a finitely presented quasi-coherent sheaf on $C$. We say $\cl F$ is \emph{torsion free of rank 1} if $\cl F$ is $S$-flat and if for every geometric point $\bar{s} \to S$, the pullback $\cl F_{\bar{s}}$ on $C_{\bar s}$ is torsion free of rank 1 in the above sense. 
\end{definition}
This is one of many possible equivalent definitions. Such sheaves are very commonly used in compactifying moduli of invertible sheaves, see for example \cite{Mumford1964Further-comment}, \cite{DSouza1979Compactificatio}, \cite{Altman1980Compactifying-t}, \cite{Esteves2001Compactifying-t}, \cite{Caporaso2008Neron-models-an}. 

\subsection{From relative components to hemispheres}

If $V$ is a subset of the vertices of a graph $\Gamma$, we say $V$ is \emph{connected} if the subgraph induced by $V$ is connected. Recall that for us being connected means having exactly one connected component, in particular being non-empty. 

\begin{definition}
Let $\Gamma$ be a connected graph. A \emph{hemisphere} of $\Gamma$ is a connected subset $G$ of the vertices of $\Gamma$ such that the complement of $G$ is also connected. We call the edges of $\Gamma$ with exactly one end in $G$ the \emph{separating edges} of $G$. 
\end{definition}

In the first part of this paper we worked with invertible sheaves indexed by relative components $(v,G)$ of $\Gamma$. We worked quite hard to prove that the set of enriched structures was empty if the family of curves was not 1-aligned, which was key in allowing us to define pullbacks of enriched structures since it essentially said that, whenever it was unclear how to define the pullback, the set of enriched structures was empty anyway, so there was no problem! With compactified enriched structures this strategy no longer works, since it is possible to have compactified enriched structures on families which are not 1-aligned (indeed, otherwise the valuative criterion for properness could not hold). To fix this we need to generalise the notion of relative component to allow $V$ to be a connected set of vertices of $\Gamma$ rather than just a single vertex. We could thus work with relative components $(V,G)$ where $V$ was simply a connected subgraph, but there would be a fair amount of redundancy since $G$ would no longer determine $V$, so we could end up with two torsion free rank 1 quotients of the same ideal sheaf. The compatibility conditions we will impose in \ref{def:compatibility} can be phrased so as to require these quotients to be isomorphic, but this gets a little messy. Thus when defining compactified enriched structures we will work instead with hemispheres; the hemispheres of $\Gamma$ are exactly those $G$ that appear in these generalized relative components $(V,G)$. Then setting the data of a compactified enriched structure to be a tuple of torsion free rank $1$ quotients of $s^*\ca I_{G^c}$ for each hemisphere $G$ avoids the redundancy having the same $G$ appear with multiple $V$.

\subsection{Enrichment data}
In this section we define an enrichment datum on a $\sigma$-controlled curve (to be defined in a moment). After this we will define what it means for such a datum to be \emph{compatible}, and we will define a \emph{compactified enriched structure} to be a compatible enrichment datum. 

\begin{definition}\label{def:sscontrolled}
Let $S \ra \frak M$ be a controlled curve. Let $\frak{s} \in S$ be a controlling point with graph $\Gamma$, and let $S \ra \schart \ra \frak M$ be a controlled neighbourhood. Let $\schart_S$ be the immediate neighbourhood of $S$ in $\schart$. %, so $S \ra \schart$ factors via $\schart_\frak{s} \ra \schart$, and $\schart_\frak{s}$ is the immediate neighbourhood of $S$ in $\schart$. 

We say $S \ra \frak M$ is \emph{$\sigma$-controlled} if for every irreducible component $v$ of $C_\frak s$ there exists a section $\sigma_v\colon \schart_S \to C_{\schart_S}$ passing through the smooth locus of $v$. 
\end{definition}
Note that we do \emph{not} fix the section $\sigma_v$, so one cannot glue these to make sections over the whole of $\frak M$. 
\begin{lemma}\label{lem:ssc_base_for_topology}
$\sigma$-controlled curves form a base for the big \'etale site over $\frak M$. 
\end{lemma}
\begin{proof}
We know this for controlled curves by \ref{lem:locally-controlled}, and it is standard that smooth morphisms admit sections \'etale-locally. 
\end{proof}
This base for the \'etale site is where we will initially define compactified enriched structures - as before, we will use \ref{sec:sheaves_on_a_base} to then extend to the general case. 

\begin{situation}\label{sit:sscontrolled}
Let $S \ra \frak M$ be a $\sigma$-controlled curve. Let $\frak{s} \in S$ be a controlling point with graph $\Gamma$, and let $S \stackrel{s}{\ra} \schart \ra \frak M$ be a controlled neighbourhood. Let $\schart_S$ be the immediate neighbourhood of $S$ in $\schart$, so $S \ra \schart$ factors via $\schart_S \ra \schart$. Choose sections $\sigma_v$ for each vertex $v$. 
\end{situation}

\begin{definition}\label{def:hemisphere_enrichment}
Suppose we are in \ref{sit:sscontrolled}. Let $G$ be a hemisphere of $\Gamma$. An \emph{enrichment datum} for $G$ is a torsion free rank 1 quotient
\begin{equation*}
q_G\colon s^*\ca{I}_{G^c} \twoheadrightarrow \cl{F}_{G}
\end{equation*}
such that on every fibre of $C/S$, the Euler characteristic of $\ca F_G$ is equal to the Euler characteristic of the structure sheaf $\ca O_C$. We say two enrichment data $q_G\colon s^*\ca{I}_{G^c} \twoheadrightarrow \cl{F}_{G}$ and $q'_G\colon s^*\ca{I}_{G^c} \twoheadrightarrow \cl{F}'_{G}$ are \emph{equivalent} if there exists an isomorphism $\psi_G\colon \cl{F}_{G} \ra \cl{F}'_{G}$ making the obvious triangle commute. 

An \emph{enrichment datum for $\Gamma$} consists of an enrichment datum for every hemisphere of $\Gamma$, with a corresponding notion of isomorphism. 
\end{definition}

\begin{remark}\leavevmode\label{rem:ED_indep_of_neigh}
\begin{enumerate}
\item
We could perfectly well take quotients of $\ca I_G$ instead of $\ca I_{G^c}$, but we choose the latter convention to be consistent with the first part of this paper (on the non-compactified case).  

\item We have not used the sections $\sigma_v$ yet; these will be needed for defining compatibility, and can be ignored in this section. 

\item As in the case of enriched structures (see \ref{rem:changing_the_neighbourhood}), one checks easily that these notions are independent of the choice of controlled neighbourhood. 
\end{enumerate}
\end{remark}

\subsection{Local structure of pullbacks of ideal sheaves again}
In this section we prove a generalisation of \ref{lem:local_structure_of_pullbacks}, and some related results, which will be useful in what follows. Some of these results will make earlier results (such as \ref{lem:local_structure_of_pullbacks}) redundant, but we have separated them out to reduce the technicalities encountered by the reader only interested in the non-compactified case. We start by recalling a result of Faltings on the local structure of torsion free rank 1 sheaves. 
\begin{lemma}\label{lem:local_str_of_tfr1}
Let $R$ be a local ring, $C/R$ a controlled curve, and $\cl F$ a torsion free rank 1 sheaf on $C/R$. Let $e$ be an edge of the graph of $C/R$ whose label is zero in $R$, and choose an isomorphism $f$ from the $R$-algebra $A\coloneqq R[[x,y]]/(xy)$ to the completion of $C$ along the section corresponding to $e$. Then there exist elements $a$, $b \in R$ with $ab=0$ and an isomorphism of $A$-modules
\begin{equation*}
f^* \cl F \iso \frac{A\Span{U,V}}{yU+aV,xV+bU}. 
\end{equation*}
\end{lemma}
The elements $a$, $b$ are clearly not unique --- for example, if $f^* \cl F$ is invertible then we can take $a=0$ and $b$ to be any unit in $R$ or vice versa. 
\begin{proof}
This is a special case of theorem 3.5 of \cite{Faltings1996Moduli-stacks-f}, which treats also the case of non-constant degeneration and higher-rank torsion free sheaves. The authors are grateful to Emre Can Sert\"oz for pointing out this reference. 
\end{proof}

We now give a slight generalisation of \ref{lem:local_structure_of_pullbacks}. 
\begin{lemma}\label{lem:enhanced_local_structure_of_pullbacks}
Let $C/S$ be a controlled curve with graph $\Gamma$ and $S \stackrel{s}{\to} {\schart_S}$ an immediate neighbourhood, and assume that $S$ is local. Let $G$ be a hemisphere of $\Gamma$, and assume that for every separating edge $e$ of $G$, the label of $e$ is 0 on $S$. 

Write $C_G$ for the union of the irreducible components of $C$ \todo{Owen: why say ``containing points''?D: I agree; removed} corresponding to vertices in $G$, and define $C_{G^c}$ similarly, so that $C$ is a fibred coproduct over $S$ of $C_G$ with $C_{G^c}$ along the union of the sections of $C/S$ corresponding to separating edges of $G$. Write $B_G$ for this union of sections, which we will also view as Cartier divisors on $C_G$ or $C_{G^c}$ interchangeably. 

Let $T_G$ be the torsion module on $C_G$ given by 
\begin{equation*}
T = \bigoplus_{p \in B_G}\ca{O}_p^{\oplus (B_G - p)}
\end{equation*}
where $\ca{O}_p$ denotes the push forward to $C_G$ along $p$ of the structure sheaf of $S$. Then there exist isomorphisms
\begin{equation}
f_{G}\colon s^*\ca I_{G^c} |_{C_G} \stackrel{\sim}{\ra} \ca{O}_{C_G}(-B_G) \oplus T_G,
\end{equation}
\begin{equation}
f_{G^c}\colon s^*\ca I_{G^c}|_{C_{G^c}} \stackrel{\sim}{\ra} \bigoplus_{p \in B_G}\ca{O}_{C_{G^c}}(p)
\end{equation}
and for each $b \in B_G$
\begin{equation*}
g_b\colon b^*\left( \ca{O}_{C_G}(-B_G) \oplus T_G \right ) \ra b^*\left( \bigoplus_{p \in B_G}\ca{O}_{C_{G^c}}(p) \right)
\end{equation*} 
such that for each $b \in B_G$ the following diagram commutes:

\begin{tikzcd}
b^*\left( \ca{O}_{C_G}(-B_G) \oplus T_G \right ) \arrow{rr}{g_b} \arrow{d}{b^*f_{G}} & & b^*\left( \bigoplus_{p \in B_G}\ca{O}_{C_{G^c}}(p) \right)\arrow{d}{b^*f_{G^c}}\\
b^*s^*\ca I_{G^c}|_{C_G} \arrow[equals]{r} & (s\circ b)^*\ca I_{G^c} \arrow[equals]{r}& b^*s^*\ca I_{G^c}|_{C_{G^c}}. \\
\end{tikzcd}

Moreover, if we view the isomorphism $g_b$ as being a morphism
\begin{equation*}
b^*\ca{O}_{C_G}(-b) \oplus \bigoplus_{B_G - b} \ca O_p \stackrel{\sim}{\ra} b^*\ca{O}_{C_{G^c}}(b) \oplus \bigoplus_{B_G - b} \ca O_p,
\end{equation*}
then $g_b$ is the identity on the $\bigoplus_{B_G-b} \ca O_p$ parts, and induces an isomorphism $b^*\ca{O}_{C_G}(-b)  \stackrel{\sim}{\ra} b^*\ca{O}_{C_{G^c}}(b)$ (in particular $g_b$ respects the evident direct sum decomposition). 
\end{lemma}
\begin{proof}
The proof of \ref{lem:local_structure_of_pullbacks} carries over with little change. The notation is a little different since now $G$ is a hemisphere instead of being part of a relative component, but this is unimportant for the proof. `Points' becomes sections, but this again makes no real difference. Completions should be taken with respect to the defining ideal of the image of the section, which is no longer a maximal ideal, but this again does not make any difference. 
\end{proof}

We now combine \ref{lem:local_str_of_tfr1,lem:enhanced_local_structure_of_pullbacks} to obtain more detailed information on the structure of enrichment data. This is the point in our story where the restriction on the Euler characteristic in the definition of enrichment data plays a key role. 
\begin{lemma}\label{lem:enhanced_structure_of_ED}
In the notation of \ref{lem:enhanced_local_structure_of_pullbacks}, suppose $q\colon s^*\cl I_{G^c} \twoheadrightarrow \cl F$ is an enrichment datum for $G$. Write $K$ for the kernel of $q$, and fix an isomorphism \begin{equation*}
s^* \ca I_{G^c}|_{C_{G^c}} \iso \bigoplus_{b \in B_G} \ca O_{C_{G^c}}(b)
\end{equation*}
as in \ref{lem:enhanced_local_structure_of_pullbacks}. 
\begin{enumerate}
\item The coherent sheaf $K$ is supported on $C_{G^c}$, and the restricted inclusion $K \to s^*\ca I_{G^c}|_{C_{G^c}}$ factors via the natural inclusion \todo{Owen: is this again just on $C_{G^c}$, or is it on all of $C$?D: on whole of $C$... Have adjusted, is it better now? }
\begin{equation*}
\bigoplus_{b \in B_G} \ca O_{C_{G^c}} \to \bigoplus_{b \in B_G} \ca O_{C_{G^c}}(b)
\end{equation*}
\item There exists an isomorphism $K|_{C_{G^c}} \cong \bigoplus_{\#e(G, G^c)-1} \ca O_{C_{G^c}}$. 
\item Choose an isomorphism as in (2). The induced map 
\begin{equation*}
\bigoplus_{\#e(G, G^c)-1} \ca O_{C_{G^c}} \to \bigoplus_{b \in B_G} \ca O_{C_{G^c}}
\end{equation*}
is given by a matrix with entries in $\ca O_S$. 
\item The quotient $\cl F$ is invertible at the point $b\in B_G$ if and only if the image of the $b$th standard basis vector of $\bigoplus_{b \in B_G} \ca O_{C_{G^c}}$ in $\cl F$ is non-zero. 
\end{enumerate}
\end{lemma}
\begin{proof}\leavevmode
\begin{enumerate}
\item We check this locally at a singular section $p \in B_G$. We choose an isomorphism $f$ from $A \coloneqq R[[x,y]]/(xy)$ to the completion of $C$ along $p$ as in \ref{lem:local_str_of_tfr1}, assuming that $x$ vanishes on the component corresponding to $G^c$. We also choose an isomorphism $f^* \cl F$ to $A\Span{U,V}/(yU+aV, xV+bU)$. We write 
\begin{equation*}
A \otimes_{\ca O_C} \bigoplus_{b \in B_G} \ca O_{C_{G^c}}(b) \iso \frac{A}{x}\Span{L_1, \dots, L_n, X}, 
\end{equation*}
where the $L_i$ correspond to elements of $B_G \setminus p$ and $X$ corresponds to the generator for $\ca O_{C_{G^c}}(p)$. The glueing conditions of \ref{lem:enhanced_local_structure_of_pullbacks} then imply that $X$ maps to a unit multiple of $V$, and each of the $L_i$ maps to a unit multiple of $yV$. Thus if $\sum_i t_i L_i  + sX$ maps to $0$ under $q$, we must have that $y \mid s$. This proves the first claim. 

\item Given $p \in B_G$, write $\textbf{v}_p$ for the $p$th standard basis vector in $\bigoplus_{b \in B_G} \ca O_{C_{G^c}}$. The submodule $\cl V$ of $\cl F|_{C_{G^c}}$ spanned by the images of the $\textbf{v}_p$ is clearly invertible (by the computations in the proof of the first part); we claim that $\cl V$ has degree 0. This we can check on field-valued points of $S$, so we reduce to the case where $S$ is a point. The structure of torsion free rank 1 sheaves is then very well understood, and the restriction that $\cl F$ have Euler characteristic equal to that of $\ca O_C$ implies that $\cl V$ has degree 0 as required. 

Knowing that $\cl V$ has degree 0 the result is very easy; the kernel $K$ consists of relations between the $\textbf{v}_p$, and such relations must then have coefficients in $\ca O_{C_{G^c}}$, proving the claim. 

\item This is immediate from (2) since $C_{G^c}$ has connected fibres. 
\item A torsion free rank 1 sheaf is invertible if and only if it is so on every fibre over $S$, so we reduce to the case where $S$ is a geometric point. This is then an easy calculation since the structure of torsion free rank 1 sheaves on curves over fields is very well understood. 
\qedhere
\end{enumerate}
\end{proof}

\begin{lemma}\label{lem:support of kernel}
Assume we are in \ref{sit:sscontrolled}, and let $G$ be a hemisphere of $\Gamma$, and $v$ a vertex not in $G$. Let $q_G\colon s^*\ca I_{G^c} \twoheadrightarrow \cl F_G$ be an enrichment datum with kernel $K_G$. Let $\sigma_v$ be the section through the smooth locus chosen in \ref{sit:sscontrolled}. Then the annihilator of $\sigma_v^*K_G$ contains the labels of all the separating edges of $G$. 
%is contained in the closed subscheme of $S$ cut out by the labels of separating edges of $G$. {Do we want the fitting support or just the annihiator? See how we use it - at the moment the proof does the annihilator version. For \ref{lem:K_indep_of_section} the annihilator version is fine. For step 1 of \ref{lem:equivalence_of_ED_and_blowup}, still need to check. Later: it's fine! }
\end{lemma}
Recall that $\sigma_v$ goes through the smooth locus of an irreducible of $C_{\frak s}$ component corresponding to a vertex in $G^c$; if the section went through the smooth locus of an irreducible component corresponding to a vertex in $G$ then $\sigma_v^*K_G = 0$ so the lemma would hold trivially. 
\begin{proof}
This result is very clear set-theoretically \todo{Owen: It's not clear to me yet. David: But if the scheme-theoretic proof is OK then that does not matter too much...}, but we must do a little work for the scheme-theoretic version. We may assume $S$ and $\schart_S$ are affine, say $S = \on{Spec}B \to \schart = \on{Spec} A$. Let $a_1, \dots, a_n \in A$ be the labels of the separating edges of $G$, and say $a_i$ maps to $b_i \in B$. Then we find that $\sigma_v^*\ca I_{G^c} = s^*(a_1, \dots, a_n)$ (since $\sigma_v$ goes through $G^c$), so pulling back along $\sigma_v$ yields a sequence
\begin{equation*}
0 \to \sigma_v^*K_G \to \sigma_v^*(a_1, \dots, a_n) \to \sigma_v^*\cl F_G\to 0, 
\end{equation*}
which is exact since $\cl F_G$ is flat over $C$ in a neighbourhood of $\sigma_v$. This $\sigma_v^*\cl F_G$ is invertible (since $\cl F_G$ can only be non-invertible at non-smooth points of the fibres). Moreover, using that $\frak M$ has normal crossings singularities we see that 
\begin{equation*}
\sigma_v^*(a_1, \dots, a_n) = \frac{B\langle A_1, \dots, A_n\rangle}{(b_iA_j-b_jA_i:1\le i,j\le n)}. 
\end{equation*}
To simplify notation write $\psi = \sigma_v^*q_G$. Since $\cl F_G$ is $S$-flat by assumption, we see that $\sigma_v^*K_G$ is still the kernel of $\psi$. Hence we have an exact sequence 
\begin{equation*}
0 \to \sigma_v^*K_G \to \frac{B\langle A_1, \dots, A_n\rangle}{(b_iA_j-b_jA_i:1\le i,j\le n)} \stackrel{\psi}{\to} \sigma_v^*\cl F_G \to 0
\end{equation*}
where $\sigma_v^*\cl F_G$ is invertible, and we want to show $\sigma_v^*K_G$ is killed by every $b_i$. 

Well, let $\sum_ic_iA_i \in \on{Ker}\psi$. Note that $b_j\sum_i c_iA_i = \left(\sum_ib_ic_i\right)A_j$. Then 
\begin{equation*}
0 = b_j0 = b_j\psi\left(\sum_ic_iA_i \right)= \left(\sum_i b_ic_i\right)\psi(A_j), 
\end{equation*}
and since the $\psi(A_j)$ generate an invertible module this implies that $\sum_ib_ic_i = 0$. Then we see that
\begin{equation*}
b_j\sum_i c_iA_i = \left(\sum_ib_ic_i\right)A_j = 0A_j = 0
\end{equation*}
as required. 
\end{proof}

%\subsection{Modules with large support}\label{sec:large_support}
%{This is introduced to fix a problem with def of compatibility (was 6.13). To polish, maybe move, etc. }
%
%Let $S$ be a scheme and $M$ a quasi-coherent $\ca O_S$-module. We define the \emph{schematic support} $\on{SSup}M$ to be the closed subscheme of $S$ cut out by the annihilator ideal sheaf of $M$. 
%
%If $M$ is finitely generated then by \cite[\href{https://stacks.math.columbia.edu/tag/00L2}{Tag 00L2}]{stacks-project} we see that the underlying set of $\on{SSup}M$ is the usual, `set-theoretic' support $\on{Supp}M$. 
%
%Note that by \cite[\href{https://stacks.math.columbia.edu/tag/07T8}{Tag 07T8}]{stacks-project}, if $M$ is finitely generated then the formation of $\on{SSup} M$ commutes with flat base-change. However, the flatness is really necessary here; consider $S = \on{Spec} ...$
%
%We say $M$ has \emph{weakly-large support} if $M$ is finitely generated and the closed immersion $\on{SSup}M \to S$ is an isomorphism. We say $M$ has \emph{large support} if it has weakly large support universally, i.e. after arbitrary base-change over $S$. 
%
%\begin{lemma}
%Suppose $M$ is cyclic, then $M$ has large support if and only if $M$ has weakly large support. 
%\end{lemma}
%\begin{proof}
%A cyclic module with weakly large support is free of rank 1. 
%\end{proof}
%
%\begin{lemma}
%Suppose $S$ is reduced, then $M$ has large support if and only if $M$ has weakly large support. 
%\end{lemma}
%\begin{proof}
%Maybe true, maybe needed?? 
%\end{proof}

\subsection{Compatibility of enrichment data}
Suppose we are in \ref{sit:sscontrolled} and have an enrichment datum for $\Gamma$. We need some analogue of the condition in the definition of an enriched structure that the tensor product of the invertible sheaves is trivial. There are two problems with applying that condition directly in our situation:
\begin{enumerate}
\item
If a tensor product of sheaves is trivial then by definition they were all invertible, so we are not going to see very interesting torsion-free rank 1 sheaves!
\item Because we are working with hemispheres instead of relative components it is not completely clear which combinations of quotients we should require to be trivial. 
\end{enumerate}
In this section we construct a different condition, which satisfies two important properties (the proofs of which will take up much of the remainder of this paper):
\begin{enumerate}
\item
If one restricts to invertible quotients, then the resulting notion of compactified enriched structure is naturally equivalent to our previous definition of enriched structure, and this makes the stack of enriched structures into an open substack of the stack of compactified enriched structures; 
\item The resulting functor of compactified enriched structures is relatively representably by an algebraic space over $\frak{M}$, and moreover is proper over $\frak M$, making it a good notion of `compactification'. 
%\item The map from the stack of  is an open immersion. 
\end{enumerate}
In future work we plan to show that the stack of compactified enriched structures is smooth over $\bb{Z}$, give a stratification of the boundary into smooth pieces, and show that the stack of enriched structures is dense in every fibre over $\ca{M}$, making it a very good compactification! Note that the boundary is \emph{not} a normal crossings divisor for dimension reasons, cf.\ \ref{sec:2-gons} where we removed 2 points from the blowup of the completed affine plane. 

We now begin the definition of compatibility. Suppose again that we are in \ref{sit:sscontrolled} and have an enrichment datum for $\Gamma$, so we have a torsion free rank 1 quotient 
 $q_G\colon s^*\ca{I}_{G^c} \twoheadrightarrow \cl{F}_{G}$ for every hemisphere $G$ (a surjection of sheaves on the tautological curve $C/S$). Write $K_G$ for the kernel of $q_G$, so we have a short exact sequence
\begin{equation*}
0 \to K_G \to s^*\ca{I}_{G^c} \twoheadrightarrow \cl{F}_{G} \to 0
\end{equation*}
of coherent sheaves on $C$ (for later use, note that the formation of the kernel commutes with base change by the $S$-flatness of $\cl{F}_G$). Write $\ca J_G$ for the ideal sheaf on $\schart_S$ generated by the labels of separating edges of $G$. If $v \in G^c$ is any vertex then we have a canonical inclusion 
\begin{equation*}
\sigma_v^*K_G \hookrightarrow \sigma_v^*s^*\ca{I}_{G^c} = \ca J_G
\end{equation*}
of sheaves on $S$ (using that $\cl{F}_G$ is $C$-flat in a neighbourhood of $\sigma_v$). We write $\bb{K}_G \coloneqq \sigma_v^*K_G$ for this submodule of $\ca J_G$; the notation is justified by 
\begin{lemma}\label{lem:K_indep_of_section}
The submodule $\bb K_G = \sigma_v^*K_G \sub \ca J_G$ is independent of the choice of component $v \in G^c$ and of the choice of section $\sigma_v$ through the smooth locus of $v$. 
\end{lemma}
\begin{proof}
Write $i\colon Z_S(\ca J_G) \to S$ for the closed immersion. By \ref{lem:support of kernel} we see that $i_*i^*\sigma_v^*K_G = \sigma_v^*K_G$, hence we may reduce to the case where $i$ is an isomorphism, i.e. where all the labels vanish on the whole of $S$. The result is then clear from part 3 of \ref{lem:enhanced_structure_of_ED}, since the inclusion of $K_G$ in $s^*\ca{I}_{G^c}$ is isomorphic to a linear map of free $\ca O_{C_{G^c}}$-modules with matrix entries in $\ca O_S$ (note that the Euler characteristic restriction is key in that lemma).
\todo{Owen: Can you supply just a little more connecting information? I'm not sure what to do with \ref{lem:enhanced_structure_of_ED} to get this result.D: added a bit, please say if more needed. }
%We can then apply \ref{lem:local_str_of_tfr1,lem:enhanced_local_structure_of_pullbacks,lem:enhanced_structure_of_ED} to obtain detailed information on the structure of $s^*\ca I_G$. 
%Restricting to $C_{G^c}$ (which $\sigma_v$ goes through by assumption), we have that $s^*\ca I_{G^c} \cong \bigoplus_{p \in B_G} \ca{O}_{C_{G^c}}(p)$. For each $p \in B_G$, let $f_p$ be the global section of $\bigoplus_{p \in B_G} \ca{O}_{C_{G^c}}(p)$ taking value 1 on $\ca{O}_{C_{G^c}}(p)$ and 0 on the other summands. Then this section maps to a global section of $\ca{O}_{C_{G^c}}(B_G)$, and moreover (by a careful study of the glueing conditions, cf \ref{lem:extending_surjections}{update}) it lands within $\ca{O}_{C_{G^c}} \sub \ca{O}_{C_{G^c}}(B_G)$. Since $C_{G^c}$ is proper and flat with reduced connected geometric fibres we know that $\ca{O}_{C_{G^c}}(C_{G^c}) = \ca{O}_S$, so the image of $f_p$ is in fact constant. Finally, the $f_p:p \in B_G$ generate the restriction of $s^*\ca I_{G^c}$ to $C_{G^c}\setminus B_G$ (which is the intersection of the smooth locus of $C/S$ with $C_{G^c}$), so the lemma follows. 
\end{proof}
%Note that the above lemma is the only place where we need that $\upsC/\frak M$ has normal crossings singularities. Thus if one were prepared to fix a section through the smooth locus of every irreducible component the normal crossing assumption could be ignored at least in defining compactified enriched structures. Requiring that the torsion free rank 1 quotients be invertible one would obtain a new notion of enriched structure; we do not know whether this is interesting. 

Let $E$ be a non-empty subset of the edges of $\Gamma$. Let $\ca J_{E}$ be the ideal sheaf on $\schart_S$ cut out by the labels of edges in $E$. If $G$ is a hemisphere of $\Gamma$ such that the separating edges of $G$ all lie in $E$ then the ideal sheaf $\ca I_{G^c}$ is contained in $\pi^*\ca J_{E}$ on $C_{\schart_S}$, where $\pi\colon C_{\schart_S} \to \schart_S$ is the structure morphism. Pulling back along $\sigma_v \circ s$  for any $v \in G^c$ we obtain a map of coherent sheaves on $S$
\begin{equation*}
f_{E,G}\colon \ca J_G = \sigma_v^*s^*\ca I_{G^c} \ra \sigma_v^*s^* \pi^*\ca J_{E} = s^*\ca J_{E}. 
\end{equation*}
Recall from \ref{sec:table_of_notation} that we use $\ca J$ for the ideal on the base $\schart$, and $\ca I$ for the ideal upstairs on the curve. We use $K$ for kernels of maps upstairs on the curve, and $\bb K$ for (corresponding) kernels of maps downstairs on $\schart$. 

\begin{definition}\label{def:compatibility}
We say the enrichment datum is \emph{compatible} (or \emph{Fitting-compatible} for emphasis) if for every non-empty set $E$ as above the closed immersion
\begin{equation}\label{eq:compatibility}
\Fsupp \frac{s^*\ca{J}_{E}}{\sum_Gf_{E,G}(\bb{K}_G)} \to S
\end{equation}
is an isomorphism. Here the sum runs over hemispheres $G$ such that every separating edge of $G$ is contained in $E$, and $\Fsupp$ denotes the Fitting support of a coherent sheaf, see \ref{def:fitting_support}. 
\end{definition}

%For later use, we also introduce a slight generalisation: if $E_0$ is any subset of the edges of $\Gamma$, we say an enrichment datum is \emph{$E_0$-compatible} (or \emph{$E_0$-Fitting-compatible}) if the condition of \ref{def:compatibility} holds for every $E \sub E_0$. 

\begin{definition}\label{def:controlled_compactified_ES}
If $C/S$ is $\sigma$-controlled, a \emph{compactified enriched structure} on $C/S$ is a compatible enrichment datum on $C/S$. 
\end{definition}
\begin{remark}
This definition is independent of the choice of controlled neighbourhood, and also that of the sections $\sigma_v$. For the former, see \ref{rem:ED_indep_of_neigh}. The latter follows from \ref{lem:K_indep_of_section}. 
\end{remark}

The definition of compatibility risks being somewhat cryptic, so we conclude this section with two detailed examples. 
\begin{example}[Compact type]
Suppose that $C/S$ is of compact type, so hemispheres of $\Gamma$ are in natural bijection with edges of $\Gamma$. Then each of the ideal sheaves $\ca I_{G^c}$ is itself invertible, so every enrichment datum consists of isomorphisms $q_G\colon s^*\ca I_{G^c} \stackrel{\sim}{\to} \ca F_G$. The kernels $K_G$ are thus the zero module, and the same holds for their images under $f_{E, G}$. Thus, the enrichment datum being compatible is equivalent to the map 
\begin{equation*}
\Fsupp s^*\ca{J}_{E} \to S
\end{equation*}
being an isomorphism. Since the formation of the Fitting support commutes with arbitrary base-change, it suffices to check that the natural map 
\begin{equation*}
\Fsupp \ca{J}_{E} \to \schart_S
\end{equation*}
is an isomorphism. Now we have $\on{Fitt}_0 \ca{J}_{E} \subseteq \on{Ann} \ca{J}_{E} = (0)$, the first equality by \ref{fitt_0_in_ann}, and the second since all labels are non-zero and $\schart$ is integral. So we see that this map is always an isomorphism in the compact type case. To summarise, for curves of compact type, enrichment data always consist only of isomorphisms, and compatibility is automatic. 
\end{example}

\begin{example}[Neighbourhood of the 2-gon]
Let $C/k$ be the 2-gon, and $\schart$ be an immediate neighbourhood, so $\ca O_\schart(\schart)$ has elements $x$ and $y$ corresponding to the labels of the two edges. Let $s\colon \on{Spec} k \to \schart$ be the inclusion of the central fibre $x=y=0$. Note that the graph has exactly two hemispheres, each consisting of a single vertex. If the vertices are denoted $v_1$ and $v_2$, then an enrichment datum thus consists of a pair
\begin{equation*}
(q_1\colon s^* \ca I_{v_1} \twoheadrightarrow \cl F_{v_1}, q_2\colon s^* \ca I_{v_2} \twoheadrightarrow \cl F_{v_2})
\end{equation*}
of torsion free rank 1 (see \ref{def:TFR1}) quotients of the $s^*\ca I_{-}$. 

When considering compatibility, we have to run over all non-empty subsets $E$ of the edges of the graph. There are exactly three such subsets. Denoting the edges $e$ and $e'$, we treat first the case where $E$ has cardinality 1, say $E = \{ e\}$. Then there are no hemispheres with separating edges contained in $E$, so the sum in the expression (\oref{eq:compatibility}) is over the empty set, hence is the zero sub-module. We argue as in the compact-type case above to see that the map in (\oref{eq:compatibility}) is thus an isomorphism. 

It remains to treat the case where $E = \{ e, e'\}$. Here $\ca J_{E}$ is the ideal generated by $x$ and $y$, so its pullback $s^*\ca J_{E}$ is a free module of rank 2, with natural basis elements which we denote $X$, $Y$. Supposing we have sections $\sigma_{v_i}$ of $C_\schart/\schart$ passing through the smooth loci of the components $C_{v_i}$ of the central fibre. Then we have to consider the images of the $\sigma_{v_i}^* (\on{ker} q_i)$ in $s^*\ca J_{E}$. By \ref{lem:enhanced_local_structure_of_pullbacks}, $\sigma_{v_i}^* (\on{ker} q_i)$ is a $k$-module of rank 1, so its image is a 1-dimensional subspace of the module $s^*\ca J_{E} = k\Span{X, Y}$. In fact, the quotient is uniquely determined by such a submodule, so that the space of enrichment data is given by the space of pairs of lines in $k\Span{X,Y}$, hence is naturally isomorphic to $\bb P^1 \times \bb P^1(k)$. 

The map (\oref{eq:compatibility}) is an isomorphism if and only if the $k$-module
\begin{equation*}
\frac{s^*\ca{J}_{E}}{\sum_Gf_{E,G}(\bb{K}_G)}  = \frac{k\Span{X, Y}}{\sigma_{v_1}^* (\on{ker} q_1) + \sigma_{v_2}^* (\on{ker} q_2)}
\end{equation*}
is non-zero, if and only if the two 1-dimensional subspaces $\sigma_{v_i}^* (\on{ker} q_i)$ are equal. Hence, the space of compatible enrichment data is given by the diagonal in ${\bb P^1\times\bb P^1}(k)$. 

From (4) of \ref{lem:enhanced_structure_of_ED} we see that the quotient $\ca F_{v_i}$ is invertible if and only if the submodule $\sigma_{v_i}^* (\on{ker} q_i)$ in $s^*\ca I_{E}$ is not given by $X=0$ or $Y=0$; in other words, the invertible enrichment data are parametrised by a $\bb G_m(k)  \sub \bb P^1(k)$. We see then that the compatibility condition in particular imposes that either both $\ca F_{v_1}$ and $\ca F_{v_2}$ are invertible, or neither is. 
\end{example}

\subsection{Pulling back compactified enriched structures}

Suppose we are in \ref{sit:sscontrolled}. Let $\phi\colon T \to S$ be another $\sigma$-controlled curve, and let $T\stackrel{t}{\to} \schart_T \hra   \schart_S$ be the immediate neighbourhood of $T$ in $\schart_S$. We may take the sections on $C_{\schart_T}/\schart_T$ to be induced by those from $C_{\schart_S}/\schart_S$. We have a contraction map $\phi_\Gamma\colon \Gamma_S \to \Gamma_T$. 

\begin{lemma}
If $H$ is a hemisphere of $\Gamma_T$ then $\phi_\Gamma^{-1}H$ is a hemisphere of $\Gamma_S$. 
\end{lemma}
\begin{proof}
Since $\phi_\Gamma^{-1}$ preserves intersections and emptiness it is enough to show that the pullback of a connected set is connected, but this is clear since $\phi_\Gamma$ just contracts edges. 
\end{proof}

If $H$ is a hemisphere of $\Gamma_T$, let $G =\phi_\Gamma^{-1}H$. Suppose we are given an enrichment datum 
\begin{equation*}
q_G\colon s^*\ca I_G \twoheadrightarrow \cl F_G
\end{equation*}
for $G$. Note that $\phi^*s^*\ca I_G = t^*\ca I_H$, and the pullback of a torsion free rank 1 sheaf is torsion free rank 1, and surjectivity is stable under pullback. We thus obtain an enrichment datum 
\begin{equation*}
q_H \coloneqq \phi^*q_G \colon t^*\ca I_H \to \phi^* \cl F_G. 
\end{equation*}
In this way, we build an enrichment datum for $C_T/T$ from one for $C_S/S$. 

\begin{lemma}\label{lem:compatible_pullback}
Let 
\begin{equation*}
\left(q_G\colon s^*\ca I_G \twoheadrightarrow \cl F_G\right)_G
\end{equation*}
be a \emph{compatible} enrichment datum for $C/S$. Then the pullback to $T$ (as constructed above) is also compatible. 
\end{lemma}
\begin{proof}
Formation of pullbacks, quotients and the Fitting support are compatible with base-change. Formation of kernels of surjections to \emph{flat} targets are also compatible with base change. Thus the problems are essentially combinatorial.

Let $E_T'$ be a non-empty subset of the edges of $\Gamma_T$. Let $E_S'$ be the set of edges of $\Gamma_S$ which map to edges in $E_T'$ (so in particular $\# E_T' = \#E_S'$). Let $H$ be a hemisphere of $\Gamma_T$ with all separating edges of $H$ contained in $E_T'$. Then every separating edge of $\phi_\Gamma^{-1}H$ is contained in $E_S'$, since it maps to a separating edge of $H$. Conversely, if $G$ is a hemisphere of $\Gamma_S$ with every separating edge of $G$ contained in $E_S'$ then $G$ is the pullback of a hemisphere from $\Gamma_T$; indeed, any vertex $v$ of $\phi_\Gamma^{-1}\phi_\Gamma(G)$ which is not contained in $G$ must be connected to some vertex of $G$ by a chain of edges which are all contracted by $\phi_\Gamma$, but then none of these edges are in $E_S'$, contradicting the existence of such a $v$. 

In this way we see that $\phi_\Gamma^{-1}$ induces a bijection between hemispheres of $\Gamma_T$ with all separating edges in $E_T'$ and hemispheres of $\Gamma_S$ with all separating edges in $E_S'$. Thus in the notation of \ref{def:compatibility} we see that 
\begin{equation*}
\phi^* \frac{s^*\ca{J}_{E_S'}}{\sum_Gf_{E_S',G}(\bb{K}_G)}  = \frac{t^*\ca{J}_{E_T'}}{\sum_H f_{E_T',H}(\bb{K}_H)} 
\end{equation*}
where the sums run over the relevant sets of hemispheres discussed above. The result is clear since the formation of the Fitting support of a coherent module commutes with pullback, see \ref{fitt_0_in_ann}. 
%\cite[\href{http://stacks.math.columbia.edu/tag/056J}{Lemma 056J}]{stacks-project}. \todo{Owen: The argument being that if we take something supported everywhere, then pull it back, it's still supported everywhere? Is that true? I'm imagining including a point into a point with fluff, and pulling back the fluff. D: formation of support of coherent modules commutes with base change --- added a ref. }{Change ref to our lemma about the Fitting support. }
\end{proof}

\begin{definition}\label{def:CES_sscont}
Write $\cat{\sigma\textbf{-Sch}}_{\frak M}$ for the full subcategory of schemes over $\frak {M}$ whose objects are $\sigma$-controlled curves. Then by \ref{lem:ssc_base_for_topology} this is a base for the \'etale topology on $\cat{Sch}_\frak M$. 

Let $\overline{\ca E}\colon \cat{\sigma\textbf{-Sch}}^{op}_{\frak M} \to \cat{Set}$ be the functor sending a curve to the set of compactified enriched structures on it, with the pullback defined just above. This is the \emph{functor of compactified enriched structures} on the subcategory of $\sigma$-controlled curves. 
\end{definition}

\subsection{Compactified enriched structures for general $C/S$, and properties}
To define compactified enriched structures for general (not necessarily $\sigma$-controlled) $C/S$ we need to show the functor of compactified enriched structures defined in \ref{def:CES_sscont} is a sheaf for the \'etale topology (c.f. \ref{sec:ES_on_controlled_curves}). As in the non-compactified case we will make use of a representability result for a related functor, so we begin by defining some functors related to $\overline{\ca E}$. 
%More precisely, we need to show that if $C/S$ is a $\sigma$-controlled curve, $\{T_i\}_I \ra S$ is an \'etale cover by $\sigma$-controlled curves, and for each $i, j \in I$ we have a cover $\{T_{i,j,k}\}_{K_{i,j}}$ of $T_i \times_S T_j$ by $\sigma$-controlled curves, then the following diagram is an equaliser:
%\begin{equation}\label{eq:basic_equaliser}
%\bar{\ca{E}}(S) \ra \prod_i \bar{\ca{E}}(T_i) \rightrightarrows \prod_{i,j,k} \bar{\ca{E}}(T_{i,j,k}). 
%\end{equation}

Let $C/S$ be a $\sigma$-controlled curve, with graph $\Gamma$, controlled neighbourhood $S \ra \schart$ and immediate neighbourhood $S \stackrel{s}{\to}\schart_S \sub \schart$. If $G$ is a hemisphere of $\Gamma$, let
\begin{equation*}
\ca{ED}_{S,G}\colon \cat{Sch}_S^{op} \to \cat{Set}
\end{equation*}
be the functor sending an $S$-scheme $T\stackrel{f}{\to} S$ to the set of isomorphism classes of torsion free rank 1 quotients of $f^*s^*\ca{I}_{G^c}$, with the evident notion of pullback. Let
\begin{equation*}
\ca{ED}_S\colon \cat{Sch}_S^{op} \to \cat{Set}
\end{equation*}
be the fibre product over $S$ of the functors $\ca{ED}_{S,G}$ as $G$ runs over hemispheres of $\Gamma$. Note that $\ca{ED}_S(S)$ is exactly the set of enrichment data on $S$. Finally, let 
\begin{equation*}
\overline{\ca{E}}_S\colon \cat{Sch}_S^{op} \to \cat{Set}
\end{equation*}
be the subfunctor of $\ca{ED}_S$ obtained by requiring that the enrichment data are compatible. More precisely, we choose sections $\sigma_v$ of $C_S/S$ through the smooth locus of each irreducible component $v$ of any controlling fibre, and define $\bb{K}_G$ to be the pullback to $T$ along $f^*\sigma_v$ (any $v \in G^c$) of the kernel of $q_G\colon f^*s^*\ca I_{G^c} \to \ca{F}_G$. We then require that for each non-empty set $E$ of edges of $\Gamma$, the inclusion of the Fitting support of the coherent $\ca{O}_T$-module 
\begin{equation}\label{eq:sup}
\frac{f^*s^*\ca{J}_{E}}{\sum_G \bb{K}_G}
\end{equation}
to $T$ is an isomorphism (the sum running over those $G$ all of whose separating edges are contained in $E$). Again, if $S=T$ we immediately see that $\overline{\ca E}_S(S) = \overline{\ca E}(S)$. 

We made similar definitions in the non-compactified case, but here things are actually much simpler; we do not have to worry about 1-alignment, instead we have a very general lemma:
\begin{lemma}\label{lem:Ebar_S_is_Ebar}
For all $\sigma$-controlled $T/\frak{M}$ and $\frak{M}$-maps $T \to S$, we have that $\overline{\ca{E}}_S(T) = \overline{\ca E}(T)$. 
\end{lemma}
More precisely, the discussion in the proof of \ref{lem:compatible_pullback} gives an obvious map $\overline{\ca{E}}_S(T) \to \overline{\ca E}(T)$, and we will prove that this is an isomorphism. 
\begin{proof}
This almost comes from looking at the proof of \ref{lem:compatible_pullback}. The only extra thing we need to observe is that if $E$ is a nonempty set of edges of $\Gamma = \Gamma_S$ which does not come by pullback from the graph $\Gamma_T$ of $C_T/T$ (equivalently, at least one edge in $E$ is contracted under the map to $\Gamma_T$) then $f^*s^*\ca J_{E}$ is canonically trivial, and all the images of relevant $\bb{K}_G$ in it are zero, so the quotient in (\oref{eq:sup}) clearly has Fitting support equal to $T$. 
\end{proof}

\begin{lemma}\label{lem:some_representability}\leavevmode
\begin{enumerate}
\item
For each hemisphere $G$ of $\Gamma$, the functor $\ca{ED}_{S,G}\colon \cat{Sch}_S^{op} \to \cat{Set}$ is representable by a separated $S$-scheme. 
\item 
The functor $\ca{ED}_S\colon \cat{Sch}_S^{op} \to \cat{Set}$ is representable by a separated $S$-scheme. 
\item
The functor $\overline{\ca{E}}_S\colon \cat{Sch}_S^{op} \to \cat{Set}$ is representable by a separated $S$-scheme. 
\end{enumerate}
\end{lemma}
\begin{proof}
For (1), we note that $\ca{ED}_{S,G}$ is naturally a subfunctor of a quot scheme. The inclusion of this subfunctor is moreover relatively representable, since the inclusion of torsion free rank 1 sheaves in the stack of all finitely presented sheaves is relatively representable, and fixing the Euler characteristic is an open and closed condition. We thus see that $\ca{ED}_{S,G}$ is representable, and we need to check separatedness. All the torsion free rank 1 quotients appearing as $G$-enrichment data have the same Hilbert polynomial (because they are required to have the same Euler characteristic), and thus $\ca{ED}_{S,G}$ is a subfunctor of a quot functor for a  fixed Hilbert polynomial. Such a functor satisfies the valuative criterion for properness with respect to discrete valuation rings by \cite[\S 5.5.7]{Fantechi2005Fundamental-alg}. Moreover, since $C/S$ and the various ideal sheaves we consider are of finite presentation we can reduce to the case where $S$ is noetherian, so $\ca{ED}_{S,G}$ is a subfunctor of a separated functor. 

Part (2) holds since a product of separated schemes is a separated scheme. The third functor is represented by the intersection of Fitting supports of suitably chosen coherent sheaves on the second, namely, the universal versions of the sheaves in \ref{def:compatibility}. The Fitting support is closed and it's formation commutes with arbitrary base-change (\ref{fitt_0_in_ann}), so the third functor is also represented by a separated $S$-scheme. 
\end{proof}

\begin{corollary}
The functor $\overline{\ca E}$ is a sheaf for the \'etale topology. 
\end{corollary}
\begin{proof}
We need to show that if $C/S$ is a $\sigma$-controlled curve, $\{T_i\}_I \ra S$ is an \'etale cover by $\sigma$-controlled curves, and for each $i, j \in I$ we have a cover $\{T_{i,j,k}\}_{K_{i,j}}$ of $T_i \times_S T_j$ by $\sigma$-controlled curves, then the following diagram is an equaliser:
\begin{equation}
\overline{\ca{E}}(S) \ra \prod_i \overline{\ca{E}}(T_i) \rightrightarrows \prod_{i,j,k} \overline{\ca{E}}(T_{i,j,k}). 
\end{equation}
However, for such $T$ we know that $\overline{\ca{E}}(T) = \overline{ \ca E}_S(T)$ and that the functor $\overline{\ca E}_S$ is representable, and hence a sheaf for any subcanonical topology. 
\end{proof}

\begin{definition}
By \ref{lem:comparison} the sheaf $\overline{\ca E}$ extends uniquely to a sheaf on the big \'etale site over $\frak{M}$. We denote this extension also by $\overline{\ca E}$, and call it the functor of compactified enriched structures. 
\end{definition}

\begin{corollary}
The functor $\overline{\ca E}$ is representable by a separated algebraic space over $\frak{M}$. 
\end{corollary}
\begin{proof}
By definition $\overline{\ca E}$ is a sheaf, so we can check representability locally on $\frak{M}$, say on $\sigma$-controlled patches. The result is then immediate by combining \ref{lem:Ebar_S_is_Ebar} and part (3) of \ref{lem:some_representability}. The separatedness follows by again applying part (3) of \ref{lem:some_representability}, using that separatedness is local on the target (see \cite[\href{http://stacks.math.columbia.edu/tag/02YJ}{Section 02YJ}]{stacks-project}). 
\end{proof}

\section{Properness of the stack of compactified enriched structures}\label{sec:properness}

Properness is fpqc-(hence smooth-)local on the target, so it can be checked on any smooth cover. From now until the end of the present \ref{sec:properness} we will thus replace $\frak M$ by a scheme mapping smoothly to $\frak M$ (we continue to denote it $\frak M$ to avoid cluttering the notation). Moreover we can now assume that $\upsC /\frak{M}$ is $\sigma$-controlled with graph $\Gamma$. Since the compatibility condition is defined by the vanishing of a certain Fitting ideal, we see that $\overline{\ca E}$ is a closed subscheme of the fibre product over $\frak M$ of functors $\ca{ED}_{\frak M, H}$ as $H$ runs over hemispheres of $\Gamma$. Hence it is enough to show that each $\ca{ED}_{\frak M, H}$ is proper over $\frak M$. Fixing a hemisphere $H$, we will construct an isomorphism between $\ca{ED}_{\frak M, H}$ and the blowup of $\frak M$ at the ideal sheaf $\ca J_H$ (generated by labels of separating edges of $H$), which is clearly proper over $\frak M$. 

Let $\sigma$ be a section through the smooth locus of $\upsC /\frak M$ passing through a component in $H^c$. If $s\colon S \to \frak M$ is any scheme then an object of $\ca{ED}_{\frak M, H}$ is a torsion-free rank 1 quotient 
\begin{equation*}
s^* \ca I_{H^c} \twoheadrightarrow \cl F_H, 
\end{equation*}
and pulling back along $\sigma$ we obtain an invertible quotient 
\begin{equation*}
s^* \ca J_H \to \sigma^* \cl F_H
\end{equation*}
(here (and elsewhere) we abuse notation --- the section $\sigma$ is over $\frak M$, but we pull it back to make a section of $C_S/S$, and use the same notation for this new section). 
\todo{Owen: I'm having trouble parsing these combinations. Isn't $s^* \ca I_{H^c} \twoheadrightarrow \cl F_H$ taking place on $C_S$, but $\sigma$ a map from $\frak M$ to $\ca{C}$? D: All that you say is true. But I am implicitly base-changing the section to make a section over $S$.  Added a few words, please say if not happy with this. }
We thus obtain a natural transformation from torsion-free rank 1 quotients of $s^*\ca I_{H^c}$ on $\upsC_S$ to invertible quotients of $s^* \ca J_H$ on $S$. 
\begin{lemma}\label{lem:equivalence_of_ED_and_blowup}
The above natural transformation is an isomorphism of functors $\cat{Sch}_{\frak M}^{op} \to \cat{Set}$. 
\end{lemma}
\begin{proof}
For each $\frak M$-scheme $S$ we must prove that the map from torsion-free rank 1 quotients of $s^*\ca I_{H^c}$ to invertible quotients of $s^* \ca J_H$ above is a bijection. Since both sides are representable functors it is enough to check this in the case where $S$ is the spectrum of a strictly hensellian complete local ring. Write $R = \ca O_S(S)$. 

Let $Z \tra \frak M$ be the closed subscheme where the labels of the separating edges of $H$ vanish (so $Z$ is cut out by $\ca J_H$). Suppose we are given an invertible quotient $q\colon s^* \ca J_H \twoheadrightarrow R$. Write $\bb K$ for the kernel. Our first task will be to construct a corresponding submodule $K$ of $s^* \ca I_{H^c}$, and then to show that the cokernel of $K \to s^* \ca I_{H^c}$ is torsion free of rank 1. In this way we will construct an inverse to the `pullback along $\sigma$' map above, showing that the map of functors is an isomorphism (we will leave it to the reader to check that the map we define is actually an inverse). 

\noindent \textbf{Step 1: Constructing $K$}\\
Let $\iota\colon V(H^c) \tra C_S$ be the closed subscheme corresponding to $H^c$ as in \ref{def:ideal_sheaves}. Then $\pi_{H^c}\colon V(H^c) \to Z$ is proper and flat, and we define $K = \iota_* \pi_{H^c}^*\bb K$ (noting that $\bb K$ is supported on $Z$ by a simpler analogue of \ref{lem:support of kernel}). %{This is fine, but better to be clearer: the point is that the annihilator of this module contains the ideal of $Z$, so push-pull on $Z$ yields an isomorphism. }

\noindent \textbf{Step 2: Constructing a map from $K$ to $s^* \ca I_{H^c}$}\\
Let $U_H = C_S \setminus V(H^c)$. Then $K$ is zero on $U_H$, so the map to $s^* \ca I_{H^c}$ on $U_H$ is clear. Write $\pi\colon C_S \to S$ for the structure map.  If we let $V(H)$ be as defined in \ref{def:ideal_sheaves} (so $C_Z = V(H) \cup V(H^c)$) and set $U_{H^c} = C_S \setminus V(H)$, then on $U_{H^c}$ we find that $s^*\ca I_{H^c} |_{U_{H^c}}= \pi^* s^*\ca J_{H}|_{U_{H^c}}$ and $\pi^*\bb K |_{U_{H^c}} = K|_{U_{H^c}}$, so again the map is clear. Moreover, these two maps agree on the intersection of $U_H$ and $U_{H^c}$. 

However, we are not done yet, since we also need to define a (compatible) map in a neighbourhood of the sections corresponding to separating edges of $H$. Fix such an edge $e$. In order to write down our map (and for later calculations) we need to give presentations of the various modules involved. We may assume $\frak M$ is affine and that the labels of separating edges of $H$ are principal ideals on $\frak M$. Number the separating edges $e_0, \dots, e_n$. Choose a generator $\ell_i$ for each separating edge $i$. From now on we will work locally at the completion of $C_S$ at a point in the singular subscheme corresponding to $e_0$. We see that 
\begin{equation*}
s^*\ca J_H = \frac{R\langle L_0, \dots, L_n\rangle}{\ell_i L_j - \ell_j L_i : 0 \le i, j \le n}. 
\end{equation*}
Writing $A$ for the completion of $C_S$ along $e_0$, we choose an isomorphism $A \cong R[[x,y]]/(xy-\ell_0)$, and let us assume that $x$ vanishes on the component lying in $H$. Then we find 
\begin{equation*}
s^*\ca I_{H^c} = \frac{A\langle X, L_0, \dots, L_n\rangle}{yX-L_0, \ell_i X - xL_i, \ell_i L_j - \ell_j L_i : 0 \le i, j \le n}. 
\end{equation*}
Recalling that $q\colon s^*\ca J_H \twoheadrightarrow R$ is the invertible quotient we started with, we set $q_i \coloneqq q(L_i)$. We see that 
\begin{equation*}
\bb K = \left\{\sum_{i \ge 0} \alpha_i L_i \colon \sum_{i \ge 0} \alpha_i q_i = 0\right\} \sub s^* \ca J_H. 
\end{equation*}
Then 
\begin{equation*}
K = \left\{\sum_{i \ge 0} \alpha_i L_i \colon \sum_{i \ge 0} \alpha_i q_i = 0\right\} \sub \frac{R[[y]]\langle L_0, \dots, L_n\rangle}{\ell_i L_j - \ell_j L_i : 0 \le i, j \le n}. 
\end{equation*}
It is then clear how to define the map $K \to s^* \ca I_{H^c}$; simply send $(\sum_{j \ge 0} r_j y^j)L_i$ to the same symbol considered an an element of $A\langle L_0, \dots, L_n\rangle$. The compatibility with the maps defined above on $U_H$ and $U_{H^c}$ is also clear. 

\noindent \textbf{Step 3: Verifying that the cokernel is torsion free rank 1}\\
We have built a map $K \to s^* \ca I_{H^c}$, and need to check that the cokernel is torsion free of rank 1 and of the correct Euler characteristic. The fibrewise part of this is a straightforward calculation, and is omitted. The hard part is checking that this cokernel is flat over $R$. Writing $M$ for the cokernel, we have a presentation 
\begin{equation*}
M = \frac{A\langle X, L_0, \dots, L_n\rangle}{yX-L_0, \ell_i X - xL_i, \ell_i L_j - \ell_j L_i (0 \le i, j \le n), K}. 
\end{equation*}
%which simplifies to 
%\begin{equation*}
%M = \frac{A\langle X, L_1, \dots, L_n\rangle}{\ell_i X - xL_i, \ell_i L_j - \ell_j L_i (1 \le i, j \le n), K}. 
%\end{equation*}
We want to show $M$ is $R$-flat. From the constructions this is clear away from points in the closed subscheme corresponding to the section $e_0$. This subscheme is contained in the subscheme cut out by the equation $x-y=0$. We wish to apply \ref{lem:giulio_flatness} with $a = x-y$ to check flatness. The lemma does not apply directly because the finite presentation assumptions fail, so we apply the lemma \emph{before} completing along $e_0$, and then observe that the properties we need to check can be verified \emph{after} completion. Hence by \ref{lem:giulio_flatness} it is enough to check the following:
\begin{enumerate}
\item $M/(x-y)M$ is flat over $R$;
\item $(x-y)\colon M \to M$ is injective. 
\end{enumerate}

\noindent \textbf{Step 4: Flatness of $M/(x-y)M$}\\
This is rather easy from the presentation; as $R$-modules we find that 
\begin{equation*}
\begin{split}
\frac{M}{(x-y)M} & = \frac{R[[x]]}{(x^2-\ell_0)} \frac{\langle X, L_0, \dots, L_n\rangle}{xX - L_0, \ell_i X - xL_i, \ell_i L_j - \ell_j L_i (0 \le i, j \le n), K}\\
& = \frac{R\langle X, L_0, L_1, \dots, L_n, xX, xL_0, xL_1, \dots, xL_n\rangle}{xX-L_0, \ell_i X - xL_i, \ell_i L_j - \ell_j L_i (0 \le i, j \le n), K}\\
& \cong \frac{R\langle X, L_0, \dots, L_n\rangle}{\ell_i L_j - \ell_j L_i (0 \le i,j \le n), K}\\
(\text{via }q) & = R\langle X \rangle
\end{split}
\end{equation*}
which is evidently $R$-flat!

\noindent \textbf{Step 5: Injectivity of $(x-y)\colon M \to M$}\\
Let $\bb L$ be the sub-$A$-module of $M$ generated by the $L_i$. By a small diagram chase, it is enough to show 
\begin{enumerate}
\item[2.1.]
$(x-y)\colon \bb L \to \bb L$ is injective;
\item[2.2.]
$(x-y)\colon M/\bb L \to M/\bb L$ is injective
\end{enumerate}
For (2.1), we observe first that we have a presentation
\begin{equation*}
\bb L = \frac{A\langle L_0, \dots, L_n \rangle}{\ell_i L_j - \ell_j L_i (0 \le i ,j \le n), \bb K}, 
\end{equation*}
so $\bb L \cong A$ as an $A$-module (via the map $q$), and so multiplication by $(x-y)$ is easily seen to be  injective. For (2.2), it is clear from the presentations that $M/\bb L$ is again isomorphic as an $A$-module to $A$, and so injectivity again follows. 

In conclusion, we have shown how to build a torsion free rank 1 quotient of $s^* \ca I_{H^c}$ starting from an invertible quotient of $s^* \ca J_{H}$. To show bijectivity, it remains to check that this is actually an inverse of the `pull back along $\sigma$' map, but this is not hard and we omit it. 
\end{proof}

The key technical result we used above was the following beautiful lemma whose statement and proof are adapted from \cite[lemma 1.2.2]{Orecchia2014Torsion-free-ra}. 

\begin{lemma}\label{lem:giulio_flatness}
Let $R$ be a ring, $A$ a finitely presented $R$-algebra, and $M$ a finitely presented $A$ module. Let $a \in A$. Suppose
\begin{enumerate}
\item $M/aM$ is flat over $R$;
\item $a\colon M \to M$ is injective. 
\end{enumerate}
Then for every point $q \in Z(a) \sub \on{Spec} A$, the localisation $M_q$ is flat over $R$. 
\end{lemma}
\begin{proof}
\noindent \textbf{Step 1:} There is an obvious exact sequence 
\begin{equation*}
0 \to \frac{aM}{a^2M} \to \frac{M}{a^2M} \to \frac{M}{aM} \to 0. 
\end{equation*}
Now $a\colon M \to aM$ is a bijection by (2), hence $\frac{aM}{a^2M} \cong \frac{M}{aM} $ as $A$-modules, so by (1) the module $\frac{aM}{a^2M}$ is $R$-flat, hence $\frac{M}{a^2M}$ is $R$-flat. Repeating this argument, we deduce that $\frac{M}{a^nM}$ is $R$-flat for all $n \ge 1$. 

\noindent \textbf{Step 2:} Fix a point $q \in Z(a) \sub \on{Spec} A$, and let $p \in \on{Spec}R$ be its pullback. Write $A_q$ for the localisation of $A$ at $q$, $R_p$ for the localisation of $R$ at $p$, and $M_q = M \otimes_A A_q$. We need to show $M_q$ is flat over $R_p$. By \cite[\href{http://stacks.math.columbia.edu/tag/00HD}{Tag 00HD}]{stacks-project} this is equivalent to checking that for all finitely generated ideals $I \triangleleft R_p$, the canonical map $I \otimes_{R_p} M_q \to M_q$ is injective. For each integer $n \ge 1$ consider the diagram 
\begin{center}
\begin{tikzcd}
I \otimes_{R_p} M_q \arrow{r} \arrow{d}{\phi_n} & M_q \arrow{d}\\
I \otimes_{R_p}\left(\frac{M_q}{a^nM_q}\right) \arrow[hook]{r} & \frac{M_q}{a^nM_q}, 
\end{tikzcd}
\end{center}
where the bottom arrow is injective by the $R$-flatness of $\frac{M}{a^nM}$ established above. We need to show the top horizontal arrow is injective; for this we will show that $\bigcap_{n \ge 1} \on{Ker} \phi_n = 0$. Note that 
\begin{equation*}
I \otimes_{R_p}\left(\frac{M_q}{a^nM_q}\right) = \frac{I \otimes_{R_p}M_q}{a^n(I \otimes_{R_p}M_q)}, 
\end{equation*}
and so $\on{Ker}\phi_n = a^n(I \otimes_{R_p}M_q)$. Hence we want to show that $N=0$ where
\begin{equation*}
N \coloneqq \bigcap_{n \ge 1} a^n(I \otimes_{R_p}M_q). 
\end{equation*}If $R$ (and hence $A$) is Noetherian we can use the Artin-Rees lemma; let $P \coloneqq I \otimes_{R_p}M_q$, then Artin-Rees tells us that there exists an $m \ge 0$ such that 
\begin{equation*}
(a^{m+1}P) \cap N = a((a^mP) \cap N), 
\end{equation*}
in other words that $N = aN$, so by Nakayama's lemma (using that $a \in q$) we have that $N=0$ as required. 

To extend to the general (non-Noetherian) case one uses \cite[\href{http://stacks.math.columbia.edu/tag/00R1}{Tag 00R1} and \href{http://stacks.math.columbia.edu/tag/00R6}{Tag 00R6}]{stacks-project}; more details are in \cite[lemma 1.2.2]{Orecchia2014Torsion-free-ra}. 
\end{proof}

We have established an isomorphism of functors between $\ca{ED}_{\frak M, H}$ and the functor of invertible quotients of $\ca J_H$, the latter being defined entirely in terms of $\frak M$ with no reference to the curve living over it. We still need to show that this latter functor is proper. 
\begin{lemma}
Let $Q\colon \cat{Sch}_{\frak M}^{op} \to \cat{Set}$ be the functor sending a scheme $s\colon S \to \frak M$ to the set of invertible quotients of $s^* \ca J_H$. This functor is naturally equivalent to the functor of points of the blowup of $\frak M$ along $\ca J_H$; in particular it is proper. 
\end{lemma}
%We found the discussion of the functor of points of the blowup at \href{http://mathoverflow.net/questions/91357/which-functor-does-the-blowing-up-represent}{http://mathoverflow.net/questions/91357/which-functor-does-the-blowing-up-represent} helpful. 
\begin{proof}
The blowup $B \to \frak M$ is by definition the Proj of the blowup algebra, whose functor of points is spelled out in \cite[\href{http://stacks.math.columbia.edu/tag/01NS}{Section 01NS}]{stacks-project} (note that the blowup algebra is generated in degree 1). In our situation, for a scheme $s\colon S \to \frak M$, we see that $B(S)$ is the set of equivalence classes of invertible quotients of $s^*\ca J_H \twoheadrightarrow \cl L$ such that for every integer $d >0$, the induced surjection $s^*\on{Sym}^d(\ca J_H) \twoheadrightarrow \cl L$ factors via $s^* (\ca J_H^d)$. However, recall that $\ca J_H$ is assumed to correspond to a normal crossings divisor in $\frak M$, and so this latter condition is vacuous\todo{Owen: Is that because $\on{Sym}^d \ca J_H = J_H^d$ or something? D: Yes :-). }, so the blowup is simply the functor of invertible quotients of $\ca J_H$. 
\end{proof}

\section{Comparison to enriched structures}\label{sec:CES_for_invertible_is_ES}

There are three key differences between enriched structures and compactified enriched structures:
\begin{enumerate}
\item Enriched structures are built from invertible sheaves, whereas compactified enriched structures are built from torsion free rank 1 sheaves;
\item For an enriched structures we specify one quotient for each relative component of $\Gamma$, whereas for compactified enriched structures  we specify one quotient for each hemisphere of $\Gamma$. 
\item For enriched structures we require that the tensor product of the invertible sheaves be trivial, whereas for compactified enriched structures we have a rather complicated compatibility condition in terms of the Fitting supports of certain modules. 
\end{enumerate}

Suppose we have a compactified enriched structure where all the torsion free rank 1 quotients are in fact invertible. By throwing out the $\cl{F}_{G}$ where $G$ is not a connected component of the complement of a single vertex (i.e. coming from a relative component) we obtain a collection of data which is a reasonable candidate to be an enriched structure - we just need to check that the tensor product of the invertible sheaves is trivial, which we will do in the next sections. In this way we have a map from the `invertible locus' of compacted enriched structures to enriched structures, which clearly behaves well with pullbacks, giving a map of functors. We will verify that this map is an isomorphism, and hence that the invertible locus of the stack of compactified enriched structures is canonically identified with the stack of enriched structures.

\begin{definition}
Let $C/S$ be $\sigma$-controlled. We say an enrichment datum on $C/S$ is \emph{invertible} if each of the torsion free rank 1 quotients is in fact an invertible sheaf. Similarly, we say a compactified enriched structure is \emph{invertible} if the underlying enrichment datum is invertible. 

A \emph{vertex-enrichment datum} is the same data as an enrichment datum, except that we restrict to hemispheres $G$ which are connected components of the complements of single vertices (in other words, which come from relative components). There is an obvious parallel notion of an invertible vertex-enrichment datum. We say an invertible vertex-enrichment datum is \emph{$\otimes$-compatible} if the tensor product of all the invertible sheaves is $S$-locally trivial. We say it is \emph{Fitting-compatible} if it satisfies the analogue of \ref{def:compatibility} (summing only over $G$ coming from relative components). 
\end{definition}
Thus a `$\otimes$-compatible invertible vertex-enrichment datum' is exactly the same thing as an enriched structure as defined in \ref{def:controlled_ES}! 

A vertex enrichment datum is easily obtained from an enrichment datum by forgetting some of the data. This clearly preserves invertibility, and sends a Fitting-compatible enrichment datum to a Fitting-compatible vertex enrichment datum. We need to show two things:
\begin{enumerate}
\item
Every Fitting-compatible invertible vertex enrichment datum is $\otimes$-compatible. 
\item The resulting functor from invertible compactified enriched structures to enriched structures is an equivalence. 
\end{enumerate}
We will actually prove these in a somewhat roundabout way. This is essentially because the case of non-integral base schemes is rather hard to handle, so instead we side-step it by using that the stack of invertible compactified enriched structures is reduced (since it is regular). We first prove the result in the case where all labels vanish, then deduce the result for integral base schemes, and then finally deduce the general case.

\subsection{When all the labels vanish}\label{sec:all_lab_vanish}
Suppose we are in \ref{sit:sscontrolled}. Suppose moreover that all the labels on all edges in $\Gamma$ are zero, and that $S$ is local; let $R \coloneqq \ca O_S(S)$. Recall that $\schart$ is affine, and the labels of edges in $\Gamma$ are principal ideals of $\ca O_\schart(\schart)$. We choose once and for all a generator of each principal ideal; say the edge $e$ has as label the principal ideal generated by $l_e$ (so all the $l_e$ vanish on the image of $S$ in $\schart_S$).  
%Given an invertible compactified enriched structure on $C/S$, we will show that the corresponding invertible vertex-enrichment datum is $\otimes$-compatible (\ref{}). We will then show that the resulting map from invertible compactified enriched structures on $C/S$ to enriched structures on $C/S$ is a bijection (\ref{}). 
For each vertex $v$ of $\Gamma$, let $C_v$ be the corresponding irreducible component of $C_S$ (c.f. \ref{lem:connected_irred_components}). 

%Let $\tilde{C}$ be the blowup of $C$ along the union of the sections corresponding to edges, and for each vertex $v$ let $C_v$ be the connected component of $\tilde{C}$ containing $v$, so that $C$ is the fibred coproduct of the $C_v$ along the union of the sections corresponding to edges. %Then the union of the images of these sections is exactly the locus where $C \to S$ is not smooth - we denote it $\on{Sing} \pi$. 

We begin by fixing an invertible enrichment datum $$\frak E \coloneqq (q_G\colon s^*I_{G^c} \twoheadrightarrow \cl F_G)_G$$ on $C/S$. It is easy to check with (4) of \ref{fitt_0_in_ann} that if $\frak E$ is Fitting-compatible (i.e. it is an invertible compactified enriched structure) then the corresponding vertex enrichment datum $\frak E_{vert}$ is Fitting-compatible. The main goal of this section is to prove the following theorem, whose proof will occupy the remainder of \ref{sec:all_lab_vanish}. 

%{Maybe move statement and proof of the next theorem to the start of the section, to give some structure to the random-looking collection of lemmas? Need to have the maps defined by the time we state it. }
%{Here we have\\
%- invertible compactified enrichment datum\\
%- compatible invertible enrichment datum\\
%-  invertible compactified enriched structure\\
% which all mean the same. NO! The first does not assume compatibility, the second two both do, and are stated as equivalent. But still, maybe it could be made mode clear? }

\begin{theorem}\label{thm:transfer_of_compatibility_when_labels_vanish}
In the above setting: 
\begin{enumerate}
\item
If we start with a Fitting-compatible invertible enrichment datum $\frak E$ then the invertible Fitting-compatible vertex-enriched structure $\frak E_{vert}$ obtained by forgetting extra quotients is also $\otimes$-compatible; 
\item The resulting functor from invertible compactified enriched structures to enriched structures is an equivalence. 
\end{enumerate}
\end{theorem}

First, suppose that $\frak E_{vert}$ is Fitting-compatible. We need to show that the invertible vertex-enriched structure $\frak E_{vert}$ is $\otimes$-compatible. It is enough to check this compatibility on each non-loop circuit-connected component of $\Gamma$ separately, so we fix such a component $\Upsilon$ of $\Gamma$ and let $E(\Upsilon)$ be the set of its edges and $V(\Upsilon)$ the set of its vertices. For a relative component $(v,G)$ we write $e(v,G)$ for the set of edges from $v$ to $G$, and we write $\pi(\Upsilon)$ for the set of those relative components $(v, G)$ such that $e(v,G) \subseteq E(\Upsilon)$. Observe that $e(v, G)$ is contained in a unique circuit-connected component (we will prove a more general version of this in \ref{lem:connecteing_edges_cc}), and for each $v$ there is at most one $G \in \pi^0(C - v)$ with $(v,G) \in \pi(\Upsilon)$. Note that the elements of $e(v,G)$ correspond to nodes lying on the irreducible component $C_v$; abusing notation, we will also denote by $e(v,G)$ the corresponding Cartier divisor on $C_v$.

\begin{lemma}
Let $(v,G)$ and $(u,H) \in \pi(\Upsilon)$ with $u \neq v$. Then 
\begin{enumerate}
\item 
 $v \notin H^c$ (i.e. $v \in H$), and
\item
If $e \in e(u, H)$ and $e$ has an end at $v$, then $e \in e(v, G)$. 
\end{enumerate}
\end{lemma}
\begin{proof}
\begin{enumerate}
\item
We need to show that $H^c \cap V(\Upsilon) = \{ u\}$. If not then $\Upsilon$ contains an edge from $u$ to a different relative component at $u$, which contradicts circuit-connectivity of $\Upsilon$. 
\item
Similarly; if not then $\Upsilon$ contains an edge from $v$ to a different relative component at $v$, which contradicts circuit-connectivity of $\Upsilon$. 
\end{enumerate}
\end{proof}

Combined, these two conditions imply that for every vertex $u \neq v$, the coherent sheaf $s^*\ca I_{H^c}|_{C_v \setminus e(v,G)}$ is canonically trivial (it corresponds to an ideal sheaf whose closed subscheme does not meet this locus). 

A \emph{local trivialisation} of this invertible vertex-enriched structure consists of, for each $(v,G)\in\pi(\Upsilon)$, a choice of isomorphism $f_G\colon \cl{F}_G|_{C_v} \iso \ca{O}_{C_v}(e(v,G))$ (here we are thinking of $e(v,G)$ as a horizontal Cartier divisor on $C_v$). Because of the canonical triviality of the $s^*\ca I_{H^c}|_{C_v \setminus e(v,G)}$, our choice of $f_G$ induces a map
\begin{equation*}
\left. F_G\colon s^*\ca I_{G^c} \otimes \bigotimes_{(v', G') \in \pi(\Upsilon), G' \neq G} s^*\ca I_{G'^c} \right|_{C_v \setminus e(v,G)} \to \ca{O}_{C_v \setminus e(v,G)}. 
\end{equation*}

Let $\pi_C\colon C \to S$ be the structure map. For each $e \in e(v,G)$ we note that $F_G(e) \coloneqq F_G(l_e\otimes 1 \dots \otimes 1 ) $ lies in $\pi_C^{-1}\ca O_S \sub \ca{O}_{C_v \setminus e(v,G)}(C_v \setminus e(v,G))$ (cf.\ \ref{lem:enhanced_structure_of_ED}), and is a unit since $\cl F_G$ is invertible (again by \ref{lem:enhanced_structure_of_ED}). If $(u,H)\in \pi(\Upsilon)$ \todo{[comment too long, see tex file]Owen: we've now got two meanings of $\pi$ in the same sentence, right? The structure map and the ``relative components'' function?D: Oops. Technically not in same sentence, just same paragraph I think , so that's OK, right...? I changed the structure map to $\pi_C$ --- it's not pretty, but this notation is I think only used briefly, and I am afraid of introducing another notation clash elsewhere. } is another relative component, we say $F_G$ and $F_H$ \emph{match up} if for every edge $e$ from $u$ to $v$, we have that $F_G(e) = F_H(e)$.  Note that the expression $F_H(e)$ makes sense, defining it by the analogue of the formula $F_G(e)$ (since the edge $e$ has an end at $H$).

%{I'm not sure how to interpret $F_H(e)$. In which factor do we insert the term $\ell_e$? Well, it shouldn't matter b definition of the tensor product/multilinear maps?! But this is not correct - as an element of $R$ we have $\ell_e = 0$, so we should instead be interpreting it in the pulled-back module as the pullback from $S$. So I think it does matter. Oh, we just insert it in the first place again! So it's basically fine, but the notation is quite confusing, having a tensor product over $H$s, then taking a particular $H$ and applying the same formula. }

%We set
%\begin{equation*}
%F_G(e) \coloneqq F_v(l_e \otimes 1 \otimes \cdots 1) \in \ca{O}_S^\times \sub \ca{O}_{C_v \setminus e(v,G)}(C_v \setminus e(v,G)), 
%\end{equation*}
%observing that $F_G(e) = f_G(q_G(e))$. 

\begin{lemma}\label{lem:matching_and_trivialisation}
Suppose we have chosen a local trivialisation. The following are equivalent:
\begin{enumerate}
\item For all $(v,G), (u,H) \in \pi(\Upsilon)$, the maps $F_G$ and $F_H$ match up. 
%\item
%There exists a surjection 
%\begin{equation*}
%F\colon \bigotimes_{v \in V(\Upsilon)} s^*I_v \to \ca{O}_F
%\end{equation*}
%such that for each $v \in V(\Upsilon)$, we have that $F|_{C_V\setminus e(v,G)} = F_v$; 
\item There exists an isomorphism $\Phi\colon \bigotimes_{(v,G) \in \pi(\Upsilon)}\cl F_G \iso \ca{O}_C$ such that for each $(u,H) \in \pi(\Upsilon)$, the restriction to the component $C_u\setminus e(u,H)$ of the composite 
\begin{equation*}
 \bigotimes_{(v,G) \in \pi(\Upsilon)} s^*\ca I_{G^c} \stackrel{\otimes q_G}{\lra} \bigotimes_{(v,G) \in \pi(\Upsilon)}\cl F_G \stackrel{\Phi}{\to} \ca{O}_C
\end{equation*}
is equal to $F_H$. 
\end{enumerate}
\end{lemma}
\begin{proof}
Let $e\colon u - v$ be an edge in $\Upsilon$, with $(v,G)$ and $(u,H) \in \pi(\Upsilon)$. Recall that $R = \ca O_S(S)$, define $A \coloneqq R[[x,y]]/(xy)$, and fix an isomorphism between $A$ and the completion of the local ring of $C_S$ along the section corresponding to $e$, such that $x$ vanishes on the component corresponding to $v$. Then we have natural isomorphisms of $A$-modules 
\begin{equation*}
s^* \ca I_{G^c} |_A= \frac{A\langle X, L_e, B_1, \dots, B_n\rangle}{yX - L_e, xB_1, \dots, xB_n}, 
\end{equation*}
where $L_e$ is a generator corresponding to the label generator $l_e$, and $B_1, \dots, B_n$ correspond to the labels of the other edges from $v$ to $G$. Similarly we have
\begin{equation*}
s^* \ca I_{H^c} |_A= \frac{A\langle Y, L_e, C_1, \dots, C_m\rangle}{xY - L_e, yC_1, \dots, yC_m}. 
\end{equation*}
Tensoring the above presentations together over $A$ we obtain a presentation for the $A$-module 
\begin{equation*}
M \coloneqq s^* \ca I_{G^c}|_A \otimes_A s^* \ca I_{H^c}|_A. 
\end{equation*}
Note that the restrictions of the $s^* \ca I_{F^c}$ for $F \notin \{G,H\}$ are canonically trivial on this locus, so we can ignore them. 

Considering the condition (2), note that the map $\Phi$ is uniquely determined on the smooth locus by the condition on the restriction to the $C_u\setminus e(u,H)$. So the only question is whether these maps on the smooth locus can be patched together over the nodes. The condition (2) is thus equivalent to the existence of an $A$-module map $\Psi\colon M \to A$ satisfying the following two conditions:
\begin{itemize}
\item
 after inverting $x$, the map $\Psi$ sends $\frac{X}{x} \otimes L_e$ to $F_H(e)$ \emph{i.e. it gives the correct map on the $u$-component};
 \item  after inverting $y$, the map $\Psi$ sends $L_e \otimes \frac{Y}{y} $ to $F_G(e)$ \emph{i.e. it gives the correct map on the $v$-component}. 
\end{itemize}

We now move to proving the lemma. Suppose first that (1) holds. Then we construct a map $\Psi$ as above by sending
\begin{itemize}
\item
$X \otimes Y $ to $F_G(e) = F_H(e)$; 
\item $L_e \otimes Y$ to  $yF_G(e)$;
\item $X \otimes L_e$ to $xF_H(e)$;
\item$ B_i \otimes Y$ to $yF_G(1 \otimes \dots \otimes l_i \otimes \dots \otimes 1)$;
\item $X \otimes C_j$ to $xF_H(1 \otimes \dots \otimes l_j \otimes \dots \otimes 1)$;
\item$ B_i \otimes C_j$ to $0$.
\end{itemize}
It is straightforward to verify that this map is well-defined. 

Conversely, suppose (2) holds. Then we have a map $\Psi$ as above. After inverting $y$ we find that
\begin{equation*}
\Psi_y(X \otimes Y) = \Psi_y(yX \otimes \frac{Y}{y}) = \Psi_y(L_e \otimes \frac{Y}{y}) = F_G(e), 
\end{equation*}
and after inverting $x$ we find
\begin{equation*}
\Psi_x(X \otimes Y) = \Psi_x(\frac{X}{x} \otimes xY) = \Psi_x(\frac{X}{x} \otimes L_e) = F_H(e)
\end{equation*}
which implies that $F_H(e) = F_G(e)$. 
\end{proof}

\begin{lemma}\label{lem:matching_and_triviality}
With $\Upsilon$ as above, the following are equivalent:
\begin{enumerate}
\item There exists $(t_G)_{(v,G) \in \pi(\Upsilon)} \in (\ca O_S^\times)^{\pi(\Upsilon)}$ such that for all pairs of relative components $(u,H)$, $(v,G)$ we have that $t_GF_G$ matches up with $t_{H}F_{H}$ (where the multiplication is just composition with the obvious scaling);
\item
There exists an isomorphism $\bigotimes_{(v,G) \in \pi(\Upsilon)} \cl F_G \iso \ca{O}_C$. 
\end{enumerate}
\end{lemma}
\begin{proof}
Suppose (1) holds. Then (2) follows immediately from $(1 \implies 2)$ of \ref{lem:matching_and_trivialisation}. 

Conversely, suppose (2) holds. Then define $F'_G$ for $(v,G) \in \pi(\Upsilon)$ by composing $\otimes q_G$ with the given isomorphism --- these match up by $(2\implies 1)$ of \ref{lem:matching_and_trivialisation}. Then for each $(v,G)$, the maps $F_G$ and $F'_G$ are related by an automorphism of $\ca{O}_S$ acting on $\ca O_{C_v \setminus e(v,G)}$, and we define $t_G$ such that multiplication by it is that automorphism. 
\end{proof}

Let $(v,G)\in \pi(\Upsilon)$ be a relative component, and define $F_G(e)$ for $e \in e(v,G)$ as above - this depends on the choice of trivialisation $f_G$. If we had chosen a different trivialisation (say $f'_G$) then this would differ from $f_G$ by multiplication by a unit in $R = \ca O_S(S)$, and so the resulting $F_G(e)$ would differ by the same scalar. As such, the free rank 1 $R$-submodule of $\oplus_{e(v,G)}R$ spanned by $(F_G(e))_{e \in e(v,G)}$ is independent of the choice of $f_G$, and we denote it by $L_G$. 

Since we have chosen generators for the labels of edges in $\schart_S$, we obtain a canonical identification of $R$-modules $s^* \ca J_G \iso \oplus_{e(v,G)} R$. The later module has a canonical basis and is thus naturally isomorphic to its dual. In this way we obtain an identification between $\oplus_{e(v,G)} R$ and the dual of $s^* \ca J_G$ (here by `dual' of an $R$-module $N$ we mean $\on{Hom}(N,R)$ as an $R$-module, and we write it $N^\vee$). This construction feels very ad-hoc, but really it was the choice of the generators $l_e$ that was ad-hoc. The construction is justified by the next lemma. 
\begin{lemma}\label{lem:kernels_vs_lines}
Let $u$ be a vertex in $G^c$. Applying $- \mapsto \on{Hom}_{R-mod}(-, R)$ to the exact sequence 
\begin{equation*}
0 \to \bb K_G \to s^*J_G \to \sigma_u^* \cl F_G \to 0
\end{equation*}
yields an exact sequence
\begin{equation*}
0 \to (\sigma_u^* \cl F_G)^\vee \to \oplus_{e(v,G)} R \to \bb K_G^\vee \to 0,
\end{equation*}
and the submodule $(\sigma^* \cl F_G)^\vee \to \oplus_{e(v,G)} R$ is exactly the submodule $L_G$ defined above. 
\end{lemma}
\begin{proof}
Each term in the first sequence is free as an $R$-module, and so dualising preserves exactness. The second assertion is a straightforward computation. 
\end{proof}

We temporarily need a slightly more relaxed notion of Fitting compatibility than that provided by \ref{def:compatibility}; later we will use it to recover the full form of Fitting compatibility. 
\begin{definition}
Suppose we are given a subset of a vertex-enrichment datum which contains all relative components in $\pi(\Upsilon)$. We say this collection of data is $\Upsilon$-Fitting-compatible if the closed immersion
\begin{equation}\label{eq:compatibility}
\Fsupp \frac{s^*\ca{J}_{\Upsilon}}{\sum_{(v,G) \in \pi(\Upsilon)}\bb{K}_G} \to S
\end{equation}
be an isomorphism. 
\end{definition}

\begin{lemma}\label{lem:matching_implies_supp_compatbility}\label{lem:supp-compatibility_implies_matching}
With $\Upsilon$ as above, the following are equivalent:
\begin{enumerate}
\item There exists $(t_G)_{(v,G) \in \pi(\Upsilon)} \in (\ca O_S^\times)^{\pi(\Upsilon)}$ such that for all pairs of relative components $(u,H)$, $(v,G)$ we have that $t_GF_G$ matches up with $t_{H}F_{H}$ (where the multiplication is just composition with the obvious scaling). 
\item The collection of quotients $(q_G\colon s^*\ca I_{G^c} \twoheadrightarrow \cl F_G)_{(v,G)}$ is $\Upsilon$-Fitting-compatible. 
\end{enumerate}
\end{lemma}
%In the statement (2) the reader can reasonably complain that the option of $\Upsilon$-Fitting-compatibility (just after ref{def:compatibility}) was defined for enrichment data, and the collection $(q_G\colon s^*\ca I_{G^c} \twoheadrightarrow \cl F_G)_{(v,G)}$ is not an enrichment datum. However, $\Upsilon$-Fitting compatibility only depends on those hemispheres whose separating edges are contained in $\Upsilon$, so this abuse of terminology is hopefully forgivable. Note that, since $\Upsilon$ is 2-vertex-connected, $\Upsilon$-Fitting-compatibility is by \ref{fitt_0_in_ann} equivalent to requiring that the closed immersion
%\begin{equation}\label{eq:compatibility}
%\Fsupp \frac{s^*\ca{J}_{\Upsilon}}{\sum_{(v,G) \in \pi(\Upsilon)}\bb{K}_G} \to S
%\end{equation}
%be an isomorphism. 

\begin{proof}
In this proof, when we write $\sum \bb K_G$ the sum runs over $(v,G) \in \pi(\Upsilon)$, and we write $F_v = F_G$ if $(v,G) \in \pi(\Upsilon)$ (note $v$ determines $G$ and vice versa by circuit connectivity of $\Upsilon$). Using \ref{lem:kernels_vs_lines} we can construct a presentation for $\frac{s^*J_\Upsilon}{\sum \bb K_G}$ using the $F_G(e)$; we find
\begin{equation*}
\frac{s^*\ca J_\Upsilon}{\sum \bb K_G} = \frac{\bigoplus_{v \in V(\Upsilon)}R\langle V \rangle}{\langle F_u(e)V - F_v(e)U \text{ for } e\colon u- v \text{ in } \Upsilon\rangle}. 
\end{equation*}
Now pick any $v_0 \in V(\Upsilon)$ as a basepoint. If $\gamma$ is a directed path in $\Upsilon$ from $v_0$ back to $v_0$, then when we write 
\begin{equation*}
\prod_\gamma \frac{F_u(e)}{F_v(e)}
\end{equation*}
we mean the product over all edges $u \stackrel{e}{\to} v$ in $\gamma$. Then we see that 
\begin{equation*}
\frac{s^*\ca J_\Upsilon}{\sum \bb K_G} = \frac{R\langle V_0 \rangle}{\langle V_0 \left(1- \prod_\gamma \frac{F_u(e)}{F_v(e)}\right) :\gamma\text{ from $v_0$ to $v_0$}\rangle}
\end{equation*}
(in particular we see that $\frac{s^*\ca J_\Upsilon}{\sum \bb K_G}$ is cyclic as an $R$-module). Then, from the definition of the Fitting ideal (\ref{def:fitting_ideal}), the canonical closed immersion $\Fsupp \frac{s^*\ca J_\Upsilon}{\sum \bb K_G}  \to S$ is an isomorphism if and only if for every path $\gamma$ as above we have $1- \prod_\gamma \frac{F_u(e)}{F_v(e)} = 0$, which (by a simple combinatorial argument) is equivalent to the `matching up condition' (1). 
\end{proof}

%\begin{lemma}\label{lem:supp-compatibility_implies_matching}
%Consider the following statements:
%\begin{enumerate}
%\item There exists $(t_v)_{v \in V(\Gamma)} \in (\ca O_S^\times)^{V(\Gamma)}$ such that for all pairs of vertices $u$, $v$ we have that $t_vF_v$ matches up with $t_{v'}F_{v'}$ (where the multiplication is just composition with the obvious scaling). 
%\item For every set $A$ of edges of $\Gamma$ which is the set of edges in some circuit-connected component, the invertible vertex enrichment datum $(q_G\colon s^*I_{G^c} \twoheadrightarrow \cl F_G)_{(v,G)}$ is $A$-Fitting-compatible. 
%\end{enumerate}
%Then $(2) \implies (1)$. 
%\end{lemma}
%\begin{proof}
%page 45
%\end{proof}

\begin{proof}[Proof of (1) of \ref{thm:transfer_of_compatibility_when_labels_vanish}]
By \ref{lem:supp-compatibility_implies_matching}, there exists a scaling of the trivialisations which matches up (in the notation of the following two paragraphs), and then by \ref{lem:matching_and_triviality} the invertible vertex-enrichment datum is $\otimes$-compatible. 

%(2): If we start with an enriched structure, we immediately have an invertible vertex enrichment datum (they consist of the same data). By \ref{lem:matching_and_triviality} and \ref{lem:matching_implies_supp_compatbility} the invertible vertex enrichment datum is $\Upsilon$-Fitting-compatible for every circuit-connected component $\Upsilon$. Then by \ref{lem:extending_CIVED_to_hemispheres} there is a unique way to extend it to an invertible enrichment datum which is $\Upsilon$-Fitting-compatible for every circuit-connected component $\Upsilon$. Finally by \ref{lem:2vcc_comp_implies_comp} this invertible enrichment datum is actually compatible. 
\end{proof}

Finally, we show how to reconstruct a compactified enriched structure from a Fitting-compatible vertex enrichment datum. 

\begin{lemma}\label{lem:connecteing_edges_cc}
If $H$ is a hemisphere and $E_H$ the set of separating edges for $H$, then there is a unique circuit-connected component $\Upsilon$ containing $E_H$. 
\end{lemma}
\begin{proof}
Given two separating edges $e$, $e'$ with endpoints $u$, $u' \in H$ and $v$, $v' \in H^c$, choose a minimal path from $u$ to $u'$ in $H$ and from $v$ to $v'$ in $H^c$, then the union of these paths with $e$ and $e'$ is a circuit $\gamma_{e, e'}$. If two circuit-connected subgraphs have an edge in common then their union is also circuit connected (by \ref{lem:circuit_conn_partiton}), hence the union of all the $\gamma_{e, e'}$ as $e$ and $e'$ vary over separating edges is itself circuit-connected. Again by \ref{lem:circuit_conn_partiton}, it is therefore contained in a (unique) circuit-connected component. 
\end{proof}

If we have a Fitting-compatible vertex enrichment datum, we need a way to extend it so an enrichment datum, i.e. to construct an (invertible) quotient for every hemisphere, while preserving compatibility. The next lemma shows that we can do this in a unique way, by constructing the kernel of the quotient map. 

\begin{lemma}\label{lem:extending_CIVED_to_hemispheres}
Suppose $(q_G\colon s^*\ca I_{G^c} \twoheadrightarrow \cl L_G)_{(v,G) }$ is a Fitting-compatible invertible vertex enrichment datum. Let $H$ be a hemisphere of $\Gamma$. Let $\Upsilon$ be the unique circuit-connected component of $\Gamma$ which contains all separating edges of $H$. Then there exists a unique short exact sequence 
\begin{equation*}
0 \to \bb{K}_H \to s^* \ca J_H \to \ca O_S(S) \to 0
\end{equation*}
so that the natural map
\begin{equation*}
\frac{s^*\ca J_\Upsilon}{\sum_{(v,G)\in \pi(\Upsilon)}\bb K_G} \to \frac{s^*\ca J_\Upsilon}{(\sum_{(v,G)\in\pi(\Upsilon)}\bb K_G) + \bb K_H}
\end{equation*}
is an isomorphism. Moreover the corresponding quotient of $s^*\ca I_{H^c}$ is invertible (under the correspondance of \ref{lem:equivalence_of_ED_and_blowup}). 
\end{lemma}
Note that in this argument it only matters that the labels of edges in $\Upsilon$ vanish on $S$ - this will be important in \ref{lem:unique_extension}. 
\begin{proof}
Noting above that $s^*\ca J_\Upsilon/\sum_{(v,G)\in\pi(\Upsilon)}\bb K_G$ is cyclic, the $\Upsilon$-Fitting-compatibility condition implies that $s^*\ca J_\Upsilon/\sum_{(v,G)\in\pi(\Upsilon)}\bb K_G \cong \ca O_S(S) =: R$. Then we have a natural map
\begin{equation*}
\phi\colon s^* \ca J_H \to \frac{s^* \ca J_\Upsilon}{\sum_{(v,G)\in\pi(\Upsilon)} \bb K_G} \cong R, 
\end{equation*}
and we see from the presentations in the proof of \ref{lem:matching_implies_supp_compatbility} that it is a surjection, hence the kernel $\bb K_\phi$ is free of rank $\#e(H, H^c)-1$ as an $R$-module. 

In order for the compatibility to hold we must choose $\bb K_H$ to be contained in $\bb K_\phi$, and we must have that $\bb K_H$ is free of rank $\#e(H, H^c)-1$ as an $R$-module. Since also the cokernel of $\bb K_H \to s^*\ca J_H$ is free of rank 1 we see that the only option is to have $\bb K_H = \bb K_\phi$. 

It remains to check the invertibility of the corresponding quotient of $s^*\ca I_{H^c}$. If we construct $F_H(e)\in R$ for $e \in \Upsilon$ in an analogous fashion to the construction of the $F_G(e)$ for $(v,G)$ relative components as above, then invertibility is equivalent to all the $F_H(e)$ being units in $R$. Perhaps the nicest way to see that they are indeed units is to give a direct recipe to build $F_H(e)$ out of the $F_G(e)$. First, scale the $F_G(e)$ so that they match up (cf.\ (2) of \ref{lem:matching_and_triviality}). Then if $e\colon u - v$ is an edge in $\Upsilon$ with $v \in H$ and $u \in H^c$, then we define $F_H(e) = F_v(e) = F_u(e)$. That the $F_H(e)$ so defined are in $R^\times$ is immediate. A small calculation is then required to check that (under the duality as in \ref{lem:kernels_vs_lines}) these $F_H(e)$ really correspond to the $\bb K_H$ defined above, yielding the invertibility. 
\end{proof}

\begin{lemma}\label{lem:2vcc_comp_implies_comp}
Let $\frak E$ be an invertible enrichment datum which is $\Upsilon$-Fitting-compatible for every circuit-connected component $\Upsilon$. Then $\frak E$ is Fitting-compatible. 
\end{lemma}

\begin{proof}
This is immediate by combining the following two claims with part (4) of \ref{fitt_0_in_ann}. 

\noindent \textbf{Claim 1:} Let $\Upsilon$ be a circuit-conneced component, and $\Upsilon_0 \sub \Upsilon$ a non-empty subset. Then natural map $s^*\ca J_{\Upsilon_0} \to s^*\ca J_\Upsilon$ induces a surjection
\begin{equation*}
\frac{s^* \ca J_{\Upsilon_0} }{\sum_{e(H,H^c)\sub \Upsilon_0} \bb K_H} \to \frac{s^* \ca J_\Upsilon }{\sum_{e(H,H^c)\sub \Upsilon} \bb K_H}. 
\end{equation*}
\noindent \textbf{Proof of claim 1:} That there is an induced map is clear since the $\bb K_H$ we consider for $\Upsilon_0$ are a subset of the ones we consider for $\Upsilon$. The surjectivity follows from the presentation of $\frac{s^* \ca J_\Upsilon }{\sum_H \bb K_H}$ given in the proof of \ref{lem:matching_implies_supp_compatbility}.  This concludes the proof of claim 1. 

\noindent \textbf{Claim 2:}
Let $\Upsilon$, $\Upsilon'$ be two distinct (hence having no common edges) circuit-connected components. Let $\Upsilon_0 \sub \Upsilon$ and $\Upsilon_0' \sub \Upsilon'$. Then the natural isomorphism $ s^* \ca J_{\Upsilon_0} \oplus s^* \ca J_{\Upsilon_0'} \to s^* \ca J_{\Upsilon_0 \cup \Upsilon_0'}$ induces an isomorphism
\begin{equation*}
\frac{s^* \ca J_{\Upsilon_0}}{\sum_{e(H,H^c)\sub \Upsilon_0} \bb K_H} \oplus \frac{s^* \ca J_{\Upsilon_0'}}{\sum_{e(H,H^c)\sub \Upsilon_0'} \bb K_H} \iso \frac{s^* \ca J_{\Upsilon_0 \cup \Upsilon_0'}}{\sum_{e(H,H^c)\sub \Upsilon_0\cup \Upsilon_0'} \bb K_H}. 
\end{equation*}
\noindent \textbf{Proof of claim 2:}
Because ${\Upsilon_0}$ and $\Upsilon_0'$ are contained in distinct circuit-connected components, the $\bb{K}_H$ appearing for ${\Upsilon_0} \cup \Upsilon_0'$ are exactly the union of those appearing for ${\Upsilon_0}$ and those appearing for $\Upsilon_0'$. This concludes the proof of the second claim, and thus the lemma. 
\end{proof}

\begin{proof}[Proof of (2) of \ref{thm:transfer_of_compatibility_when_labels_vanish}]
%(1): By \ref{lem:supp-compatibility_implies_matching}, there exists a scaling of the trivialisations which matches up (in the notation of the following two paragraphs), and then by \ref{lem:matching_and_triviality} the invertible vertex-enrichment datum is $\otimes$-compatible. 

If we start with an enriched structure, we immediately have an invertible vertex enrichment datum (they consist of the same data). By \ref{lem:matching_and_triviality} and \ref{lem:matching_implies_supp_compatbility} the invertible vertex enrichment datum is $\Upsilon$-Fitting-compatible for every circuit-connected component $\Upsilon$. Then by \ref{lem:extending_CIVED_to_hemispheres} there is a unique way to extend it to an invertible enrichment datum which is $\Upsilon$-Fitting-compatible for every circuit-connected component $\Upsilon$. Finally by \ref{lem:2vcc_comp_implies_comp} this invertible enrichment datum is actually compatible. 
\end{proof}
\subsection{Integral base schemes}

We adopt \ref{sit:sscontrolled}. In this section we will prove the analogue of \ref{thm:transfer_of_compatibility_when_labels_vanish} in the case where $S$ is integral. First, two lemmas which hold without restriction on $S$:

\begin{lemma}\label{lem:alignment_from_inv_quotient}
Assume that there exists an invertible enrichment datum on $C/S$. Then $C/S$ is 1-aligned. 
\end{lemma}
\begin{proof}
This needs only a trivial modification of the proof of \ref{lem:ES_implies_1aligned} --- the proof of the latter never used the condition that the tensor product of the invertible sheaves be trivial, only the existence of an invertible quotient of $s^* \ca I_{H^c}$. 
\end{proof}

\begin{lemma}\label{lem:unique_extension}
Suppose we are given a subset of an enrichment datum which contains a vertex enrichment datum. Fix a hemisphere $H$, and suppose all labels of all separating edges of $H$ vanish on $S$. All separating edges of $H$ are contained in some circuit-connected-component, say $\Upsilon$. Suppose that our given subset of an enrichment datum is $\Upsilon$-Fitting-compatible. Then there is a unique invertible quotient of $s^*\ca I_{H^c}$ which is $\Upsilon$-Fitting-compatible with all the other given quotients. 
\end{lemma}
\begin{proof}
This is a slight generalisation of \ref{lem:extending_CIVED_to_hemispheres}, and the same proof works. The main difference is that here we only assume the labels of edges in $\Upsilon$ vanish on $S$ (rather than all edges as in \ref{lem:extending_CIVED_to_hemispheres}), but this is not important for the proof. 
\end{proof}

\begin{corollary}[c.f. \ref{thm:transfer_of_compatibility_when_labels_vanish}]\label{cor:reduced_comparison}
Suppose we are in \ref{sit:sscontrolled}. Suppose moreover that $S$ is \emph{reduced}. Then
\begin{enumerate}
\item
If we start with a compatible invertible enrichment datum $\frak E$ (i.e.\ an invertible compactified enriched structure) then the invertible Fitting-compatible vertex-enriched structure $\frak E_{vert}$ obtained by forgetting extra quotients is $\otimes$-compatible; 
\item Suppose also that every connected component of $S$ is irreducible. Then the resulting functor from invertible compactified enriched structures to enriched structures is an equivalence. 
\end{enumerate}
\end{corollary}
\begin{proof}
(1): Suppose that $C/S$ is as in the statement and that we are given a compatible invertible enrichment datum. We extract a vertex-enrichment datum $(q_{v,G}\colon s^*\ca I_{G^c} \twoheadrightarrow \cl F_{v,G})_{v,G}$, and consider the line bundle $\bigotimes \cl F_{v,G}$. The locus in $S$ where this bundle is trivial on the fibres is a closed subscheme of $S$ and contains all points of $S$ by \ref{thm:transfer_of_compatibility_when_labels_vanish} (applied in the case when the base is a point), hence is equal to $S$ since $S$ is reduced. 

(2): Unlike for (1) we cannot quite deduce this formally from the case over a field. Note that both sets are empty if $C/S$ is not 1-aligned by \ref{lem:alignment_from_inv_quotient}, so we may assume that $C/S$ is 1-aligned. Fix a hemisphere $H$ of $\Gamma$. Let $z\colon Z_H \tra S$ be the closed subscheme where all the labels of separating edges of $H$ vanish. We consider two cases:

\noindent \textbf{Case 1:} The map $z$ is a surjection. Then we are essentially in the case we considered before (all relevant labels vanishing on $S$), and we are done by \ref{lem:unique_extension}. 

\noindent \textbf{Case 2:} The map $z$ is not a surjection. By alignment, $Z_H$ is cut out by monogenic ideal, and since $z$ is not a surjection we see that $Z_H$ is a Cartier divisor in $S$ (i.e.\ cut out by a regular element). Then $s^* \ca I_{H^c}$ is invertible, so we take $K_H = 0$ and compatibility is clear. 
%Let $Comp$ be the set of relative components $(v,G)$ of $\Gamma$ where $v$ is in $H^c$ and $H \sub G$. Then following \ref{lem:unique_extension} we build a Fitting-compatible kernel $\bb K_H \tra z^*s^*\ca J_{H^c}$ so that the cokernel is invertible on $C_{Z_H}$. {Need an extra step to go from quotient for $\ca J$ to quotient of $\ca I$. Use stuff from`properness' section. }Then by composing{WTF? What are we `composing'? Adjunction goes the other way...} we have a map $z_*\bb K_H \to s^* \ca I_{H^c}$, and we write $\ca{F}_H$ for the cokernel. supp-compatibility of the resulting collection of quotients is clear, we just need to check that $\cl F_H$ is invertible. Clearly it is coherent, and $C_S$ is reduced, so by \cite[ex. II.5.8]{Hartshorne1977Algebraic-Geome} it is enough to check that all fibres have rank 1. We do this in two cases. First, is $p$ is a point lying over a point in $Z_H$. Then the result follows from \ref{lem:unique_extension}. On the other hand, if $p$ does not lie over a point in $Z_H$ then the label of some (hence all, by alignment) separating edge of $H$ is a unit at $p$, so $\bb K_H$ vanishes at $p$ and $s^*\ca I_{G^c}$ is invertible (even canonically trivial) at $p$, so the result follows. 
\end{proof}

\subsection{Deducing the general case}

Recall that $\ca E/\frak M$ is the stack of enriched structures, and write $\overline{\ca E}^{inv}/\frak M$ for the stack of invertible compactified enriched structures - this is evidently an open substack of the stack $\overline{\ca E}$ of compactified enriched structures. Since $\ca E$ is regular (\ref{cor:from_Mtilde}) we know by \ref{cor:reduced_comparison} that $\ca E(\ca E) = \overline{\ca E}^{inv}(\ca E)$, and the former contains the identity, so we obtain a map $\Psi\colon \ca E \to \overline{\ca E}^{inv}$. 

\begin{theorem}
The map $\Psi\colon \ca E \to \overline{\ca E}^{inv}$ is an isomorphism. 
\end{theorem}
\begin{proof}
If $m \in \frak{M}$ is a point then the map $\Psi_m\colon \ca E_m \to \overline{\ca E}^{inv}_m$ is an isomorphism by \ref{thm:transfer_of_compatibility_when_labels_vanish}. In particular, $\Psi$ is a bijection on points. 

We may check the theorem after base-change to a regular irreducible scheme $M \to \frak M$, with the property that the pullback of the locus of smooth curves is dense in $M$. Then $\ca E_M$ is irreducible, so the same holds for $\overline{\ca E}^{inv}_M$. 

If $M^0\sub {M}$ denotes the (dense open) locus of smooth curves, then $ \ca E_M \to {M}$ is an isomorphism over ${M}^0$, and the same is true for $\overline{\ca E}^{inv}_{M} \to {M}$. Moreover the pullback of ${M}^0$ to $\ca E_M$ is dense in $\ca E_M$, so its image in $\overline{\ca E}^{inv}_M$ is dense there also. Thus we see that the map $\Psi_M$ is a birational map of integral schemes, and that $\ca E_M$ and $\overline{\ca E}^{inv}_M$ have the same dimension. 

If $p$ is a point of $\ca E$ lying over a point $m \in {M}$, then by \ref{thm:transfer_of_compatibility_when_labels_vanish} we have an isomorphism of tangent spaces
\begin{equation*}
T_p\ca E_m \iso T_{\Psi(p)} \overline{\ca E}^{inv}_m. 
\end{equation*}
The dimension of the former (and hence the latter) tangent space is computed in \ref{prop:regularity_of_E}. Moreover, the image in $T_m{M}$ of $T_{\Psi(p)}\overline{\ca E}^{inv}$ satisfies the same upper bound as that obtained for $T_p\ca E$ in \ref{lem:dim_ker}, since $\overline{\ca E}^{inv} \to {M}$ is 1-aligned by \ref{lem:alignment_from_inv_quotient}. Thus repeating the elementary dimension computation of \ref{prop:regularity_of_E} (and using that $\ca E_M$ and $\overline{\ca E}^{inv}_M$ are irreducible schemes of the same dimension) we see that $\overline{\ca E}^{inv}_M$ is regular. 

Thus we see that $\Psi_M$ is a birational bijection between regular schemes, so by \cite[\href{http://stacks.math.columbia.edu/tag/0AB1}{Lemma 0AB1}]{stacks-project} it is an isomorphism. 
\end{proof}

\appendix
\section{Defining sheaves on a base for a Grothendieck topology}\label{sec:sheaves_on_a_base}

This appendix contains the abstract results which allow us to define a sheaf by specifying its values on a full subcategory.

\begin{definition}[Sheaf with respect to a coverage]
 Let $\ca{C}$ be a category.
 A \emph{coverage} $K$ on $\ca{C}$ is an assignment $A\mapsto K(A)$ for each object $A$ of $\ca{C}$, where $K(A)$ is a collection of families of morphisms to $A$, such that if $\{f_i\colon A_i\to A\}\in K(A)$ and $g\colon A'\to A$ is any morphism in $\ca{C}$, then there is a family $\{f'_j\colon A'_j\to A'\} \in K(A')$ such that each $g\circ f'_j \colon A'_j\to A$ factors via some $f_i$.
 
 The families $\{f_i\colon A_i\to A\}\in K(A)$ are called \emph{covering families} and the pair $(\ca{C},K)$ is called a \emph{site}.
 A \emph{sheaf} on $(\ca{C},K)$ is a presheaf $\ca{F}\colon \ca{C}\op\to\cat{Set}$ such that for each covering family $\{A_i\to A\}$ in $K$, the restriction map $\ca{F}(A)\to\prod_i \ca{F}(A_i)$ is injective with image the set of those tuples $(a_i)_i$ such that, for every commutative square
 \[\begin{tikzcd}
  B \arrow{r}\arrow{d} & A_i\arrow{d} \\
  A_j \arrow{r} & A,
 \end{tikzcd}\]
 the common restrictions of $a_i$ and $a_j$ to $B$ agree.
 The category of presheaves on $\ca{C}$ will be denoted below as $\cat{PSh}(\ca{C})$, and given a coverage $K$ on $\ca{C}$, we denote by $\cat{Sh}(\ca{C}, K)$ the full subcategory of $\cat{PSh}(\ca{C})$ consisting of sheaves on $(\ca{C},K)$.
\end{definition}

\begin{definition}[Grothendieck pretopology]
 If $\ca{C}$ has pullbacks, we say that a coverage on $\ca{C}$ is a \emph{Grothendieck pretopology} if it satisfies the following three extra conditions:
 \begin{enumerate}
 \item (Compatibility with pullback.) For each covering family $\{A_i\to A\}$ and each morphism $A'\to A$ in $\ca{C}$, the pullback family $\{A'\times_A A_i\to A'\}$ is also a covering family.
 \item (Reflexivity.) For each object $A$ of $\ca{C}$, the singleton family $\{\on{id}_A\colon A\to A\}$ is a covering family.
 \item (Transitivity.) If $\{A_i\to A\}$ is a covering family, and for each $i$ we have another covering family $\{A_{ij}\to A_i\}$, then the family of all composites $\{A_{ij}\to A_i\to A\}$ is also a covering family.
 \end{enumerate}
 If $K$ is a Grothendieck pretopology on $\ca{C}$, then a presheaf $\ca{F}\colon \ca{C}\op \to \cat{Set}$ is in $\cat{Sh}(\ca{C},K)$ if and only if for each covering family $\{A_i\to A\}$ in $K$, the usual diagram
 \[\ca{F}(A)\to\prod_i \ca{F}(A_i)\rightrightarrows\prod_{i,j} \ca{F}(A_i \times_A A_j)\]
 is an equalizer.
\end{definition}

\begin{definition}[Base for a coverage]
Let $\ca{C}$ be a category equipped with a coverage, and let $\ca{B}$ be a full subcategory of $\ca{C}$. We say $\ca{B}$ is \emph{a base for $\ca{C}$} if every object in $\ca{C}$ admits a cover by objects in $\ca{B}$.
\end{definition}

\begin{lemma}[Comparison Lemma]\label{lem:comparison}
 Let $\ca{C}$ be a category with pullbacks, and let $K$ be a Grothendieck pretopology on $\ca{C}$ making every representable presheaf a sheaf (i.e.\ the topology is subcanonical).
 Let $\ca{B}$ be a base for $(\ca{C},K)$.
 Then
 \begin{enumerate}
 \item The assignment $K|_\ca{B}\colon A\mapsto K(A)\cap \ca{B}$, sending an object of $\ca{B}$ to the collection of covers of it by more objects of $\ca{B}$, is a coverage on $\ca{B}$.
 \item Let $\ca{F}$ be a presheaf on $\ca{B}$.
 Then $\ca{F}$ is a sheaf on $(\ca{B}, K|_\ca{B})$ if and only if for every cover $\{A_i\to A\}_{i\in I}$ in $K|_\ca{B}$, and for every family of covers $\{A_{ijk}\to A_i\times_A A_j\}_{k\in I_{ij}}$ of the fibered products $A_i\times_A A_j\in\ca{C}$ by objects in $\ca{B}$, the diagram
 \beq\label{eq:sheaf-equalizer}
  \ca{F}(A)\to \prod_{i\in I}\ca{F}(A_i) \rightrightarrows \prod_{\substack{i,j\in I\\k\in I_{ij}}}\ca{F}(A_{ijk})
 \eeq
 is an equalizer.
 \item The restriction functor $\ca{F}\mapsto \ca{F}|_\ca{B}$ is an equivalence $\cat{Sh}(\ca{C},K)\iso\cat{Sh}(\ca{B},K|_\ca{B})$.
 \end{enumerate}
 Thus a presheaf on $\ca{B}$ extends (uniquely) to a sheaf on $(\ca{C},K)$ if and only if diagram (\oref{eq:sheaf-equalizer}) is an equalizer for all such $A$, $A_i$, and $A_{ijk}$.
\end{lemma}

\begin{proof}\leavevmode
\begin{enumerate}

\item Let $\{A_i\to A\}_{i\in I}$ be a family in $K|_\ca{B}(A)$ and $g\colon A'\to A$ any morphism in $\ca{B}$.
We must exhibit a covering of $A'$ in $K|_\ca{B}$ such that each composite morphism to $A$ factors via some $A_i\to A$.

First form the pullback family $\{A'\times_A A_i\to A'\}_{i\in I}$ in $\ca{C}$, and for each $i$ choose a cover of $A'\times_A A_i$ by objects $\{A'_{ij}\}_{j\in I_i}$ of $\ca{B}$.
Then because $K$ is a Grothendiek pretopology, the family of all composites $\{A'_{ij} \to A'\otimes_A A_i \to A'\}_{i\in I,j\in I_i}$ is also in $K$, and in fact in $K|_\ca{B}$.
Furthermore, each composite $A'_{ij}\to A'\to A$ factors canonically via $A_i\to A$, as desired.
Thus $K|_\ca{B}$ is a coverage.

\item Let $\{A_i\to A\}_{i\in I}$ be a cover in $K|_\ca{B}$, and for each $i,j\in I$, let $\{A_{ijk}\to A_i\times_A A_j\}_{k\in I_{ij}}$ be a cover by objects in $\ca{B}$; we must show that a tuple $(a_i)_i\in \prod_i \ca{F}(A_i)$ has the same two images in $\prod_{i,j,k}\ca{F}(A_{ijk})$ if and only if for every commutative square formed by an object $B\in\ca{B}$ with morphisms to some $A_i$ and $A_j$, the two restrictions $a_i|_B$ and $a_j|_B$ agree.
Then the sheaf condition for $\ca{F}$ is exactly the same as diagram (\oref{eq:sheaf-equalizer}) being an equalizer.

The ``if'' direction is easy, since composing the map $A_{ijk}\to A_i\times_A A_j$ with the projections to $A_i$ and $A_j$ make the square with $A$ commute.
For the ``only if'' direction, let $(a_i)_i$ be a tuple such that for all $i,j\in I$ and $k\in I_{ij}$, we have $a_i|_{A_{ijk}} = a_j|_{A_{ijk}}$, and let $A_i\leftarrow B \to A_j$ be a span of objects in $\ca{B}$ over $A$. 
We will show that $a_i|_B = a_j|_B$.

In $\ca{C}$, the maps $B\to A_i$ and $B\to A_j$ give a single morphism $B\to A_i\times_A A_j$, and we can therefore form the pullback covering family $\{B \times_{(A_i\times_A A_j)} A_{ijk}\to B\}_{k\in I_{ij}}$ in $\ca{C}$.
For each $k\in I_{ij}$, let $\{B_{ijk\ell}\to B \times_{(A_i\times_A A_j)} A_{ijk}\}_{\ell\in I_{ijk}}$ be a cover of the fibered product in $\ca{C}$ by objects of $\ca{B}$.
Then by transitivity, $\{B_{ijk\ell}\to B\}_{k\in I_{ij},\ell\in I_{ijk}}$ is a covering family in $K|_\ca{B}$.
Thus $\ca{F}(B)\to\prod_{i,j,k,\ell}\ca{F}(B_{ijk\ell})$ is injective, so it suffices to show that the restrictions of $a_i$ and $a_j$ to each $B_{ijk\ell}$ agree.
But this is true, because the morphisms from $B_{ijk\ell}$ to $A_i$ and $A_j$ factor through $A_{ijk}$, where we are assuming $a_i$ and $a_j$ agree.

\item For this part, we appeal to a similar result, Corollary A.4.3 in \cite{Mac-Lane1994Sheaves-in-geom}; this is why we have the subcanonicality hypothesis.
All we need to do is translate their result in the sieve language into our covering family language; this amounts to checking that if $S$ is a sieve on $B\in\ca{B}$---i.e.\ a collection of morphisms to $B$ closed under precomposition by any composable morphism in $\ca{B}$---then $S$ contains a covering family in $K|_\ca{B}$ if and only if the collection $(S)$ of all morphisms to $B\in\ca{C}$ factoring via some morphism in $S$ contains a covering family in $K$.

The ``only if'' direction is trivial: such a covering family in $K|_\ca{B}$ is also the desired covering family in $K$.
Now suppose that $(S)$ contains a $K$-covering family $\{A_i\to B\}_{i\in I}\subset\ca{C}$.
This means that each $A_i\to B$ factors via some $B_i\to B$ in $S$.
Now choose for each $A_i$ a cover $\{B_{ij}\to A_i\}_{j\in I_i}$ by objects in $\ca{B}$. 
Then by transitivity and the sieve condition on $S$, the composite family $\{B_ij\to B\}_{i\in I, j\in I_i}$ is a covering family in $K|_{\ca{B}}$ contained in $S$, as desired.\qedhere
\end{enumerate}
\end{proof}

\begin{remark}\label{rem:calculate-sheaf-from-base}
 \Cref{lem:comparison}(3) only says that a sheaf on $\ca{B}$ extends uniquely (up to unique natural isomorphism) to a sheaf on $\ca{C}$, but in fact we can compute $\ca{F}(A)$ from $\ca{F}|_\ca{B}$ for any object $A\in\ca{C}$.
 First, choose covering families $\{A_i\to A\}_{i\in I}$ and $\{A_{ijk}\to A_i\times_A A_j\}_{k\in I_{ij}}$ with all $A_i$ and $A_{ijk}$ in $\ca{B}$.
 Then \ref{lem:comparison}(2) tells us (by applying the case $\ca{B}=\ca{C}$) that diagram (\oref{eq:sheaf-equalizer}) is an equalizer, and if $\ca{F}|_\ca{B}$ is known then this gives a way to calculate $\ca{F}(A)$.
\end{remark}

\begin{example}
 As a simple example, we obtain the well-known fact that a presheaf on the category of schemes is a sheaf in the \'etale topology if and only if it is a Zariski sheaf and satisfies descent for each \'etale cover of one affine scheme by another.
 For if we apply \cref{lem:comparison} to the full subcategory $\cat{Aff}$ of affine schemes, we find that sheaves on the big Zariski site correspond to presheaves on $\cat{Aff}$ that satisfy the sheaf condition with respect to open covers.
 Sheaves on the big \'etale site correspond to those presheaves on $\cat{Aff}$ that satisfy the sheaf condition with respect to open covers and \'etale surjections.
 Thus to check that a presheaf on the big \'etale site is a sheaf, it suffices to check that it is a Zariski sheaf (thus corresponding to a Zariski sheaf on $\cat{Aff}$) and satisfies descent for affine \'etale surjections (so that it corresponds to an \'etale sheaf on $\cat{Aff}$).
\end{example}

%\begin{example}\label{example:controlling_base}
%For us, the fundamental example of a base is given by: let $C/S$ a prestable curve over a stack, let $\ca{C}$ be the category of schemes over $S$ with any topology at least as fine as the \'etale one. Then the full subcategory $\ca{B}$ consisting of schemes $T/S$ such that $C_T/T$ admits a controlling graph is a base (ref). 
%\end{example}
%
%I think we define the functor on the base of things with graphs. Then to prove representability, we take a smooth cover of S by controlled things. Then on each of those we can make the quot scheme. But then this quot functor restricts on the base to this thing we wrote down, so the sheaves coincide, so we win. 

\bibliographystyle{alpha} %amsplain}
\bibliography{prebib.bib}

\end{document}